\documentclass{amsart}
\usepackage{amssymb, amsmath,mathtools,bm}
\usepackage{mathrsfs}
\usepackage{amscd} 
\usepackage{bbm}
\usepackage{verbatim}
\usepackage{stmaryrd}
\usepackage[dvipsnames]{xcolor}
\usepackage{a4wide}
\usepackage{enumitem}
\usepackage[backend=biber,maxbibnames=199,doi=false,isbn=false,url=false]{biblatex}
\usepackage[colorlinks,linkcolor={blue},citecolor={blue},urlcolor={purple},]{hyperref}

\usepackage[colorinlistoftodos,prependcaption,textsize=tiny]{todonotes}

\newcommand{\dd}{\,\mathrm{d}}

\renewcommand{\MR}{\mathrm{MR}}
\newcommand{\UH}{\mathcal{L}_2(U,H)}

\newcommand{\<}{\langle}
\newcommand{\?}{\rangle}

\newcommand{\into}{\hookrightarrow}

\newcommand{\loc}{\mathrm{loc}}
\newcommand{\ceqq}{\coloneqq}
\newcommand{\eqqc}{\eqqcolon}
\newcommand{\col}{\colon}
 
\renewcommand{\H}{\mathbb{H}} 
\renewcommand{\O}{\mathcal{O}} 
\newcommand{\bphieps}{\bm{\varphi}^{\eps}}
\newcommand{\wtveps}{%
\mspace{2mu}%
  \widetilde{\mspace{-2mu}\rule{0pt}{1.3ex}\smash[t]{v^\eps}}%
}
\newcommand{\wt}{\widetilde}
\newcommand{\y}{\Xi}
\newcommand{\gtn}{G_{\theta,n}}
\newcommand{\gvn}{G_{v,n}}
\newcommand{\pr}{\mathbb{P}_{\h}}

\newcommand{\ft}{F_{\theta}}
\newcommand{\fvt}{f}
\newcommand{\gvtn}{g_n}

\newcommand{\fv}{F_{v}}

\newcommand{\pz}{\partial_3}
\newcommand{\pzz}{\partial_{33}} 
\newcommand{\DeltawR}{\Delta_{{\rm R}}^{{\rm w}}}

\newcommand{\LTphi}{\mathcal{L}_{\psi}}
\newcommand{\Lvphi}{\mathcal{L}_{\phi}}
\newcommand{\Lvp}{\mathcal{P}_{\phi}}

\newcommand{\Hs}{\mathbb{H}}

\newcommand{\Tor}{\mathbb{T}}

\newcommand{\Br}{\mathrm{B}}
\newcommand{\h}{{\rm H}}
\newcommand{\n}{{\rm N}}
\newcommand{\Dr}{\Delta_{{\rm R}}^{{\rm w}}}

\newcommand{\forcetwo}{\mathcal{E}}
\newcommand{\force}{\mathcal{F}}
\renewcommand{\div}{\mathrm{div}}

\renewcommand{\tilde}{\widetilde}

\theoremstyle{plain}
\newtheorem{theorem}{Theorem}[section]
\theoremstyle{remark}
\newtheorem{remark}[theorem]{Remark}

\theoremstyle{plain}

\newtheorem{lemma}[theorem]{Lemma}
\newtheorem{proposition}[theorem]{Proposition}
\newtheorem{definition}[theorem]{Definition}

\newtheorem{assumption}[theorem]{Assumption}

\numberwithin{equation}{section}

\def\N{{\mathbb N}}

\def\Q{{\mathbb Q}}
\def\R{{\mathbb R}}

\newcommand{\E}{\mathsf{E}}

\newcommand{\BB}{\mathcal{B}}
\newcommand{\F}{\mathcal{F}}

\renewcommand{\P}{\mathsf{P}}
\newcommand{\hhp}{\mathbb{P}}
\newcommand{\PP}{\mathcal{P}}

\newcommand{\g}{\gamma}

\newcommand{\om}{\omega}
\newcommand{\Om}{\Omega}
\newcommand{\eps}{\varepsilon}

\usepackage{stmaryrd}

\allowdisplaybreaks

\begin{document}

\author[Antonio Agresti]{Antonio Agresti}
\address[A. Agresti]{Department of Mathematics Guido Castelnuovo\\ Sapienza University of Rome\\ P.le Aldo Moro
5\\ 00185 Rome\\ Italy}
\email{antonio.agresti92@gmail.com}
\author[Esm\'ee Theewis]{Esm\'ee Theewis}
\address[E. Theewis]{Delft Institute of Applied Mathematics\\
Delft University of Technology \\ P.O. Box 5031\\ 2600 GA Delft\\The
Netherlands}
\email{e.s.theewis@tudelft.nl}

\thanks{The first author is a member of GNAMPA (IN$\delta$AM). The second author is supported by the VICI subsidy VI.C.212.027 of the Netherlands Organisation for Scientific Research (NWO)}

\date{October 1, 2026}

\title[LDP for the stochastic primitive equations]{The large deviation principle for the stochastic 3D primitive equations with transport noise}



\begin{abstract}
We prove the small-noise large deviation principle for the three-dimensional primitive equations with transport noise and turbulent pressure.  
Transport noise is important for  geophysical fluid dynamics applications, as it takes  into account the effect of small scales on the large scale dynamics.  
The main mathematical challenge  is that we allow for the transport noise to act on the full horizontal velocity,  
therefore leading to a non-trivial turbulent pressure, which requires an involved analysis to obtain the necessary energy bounds. Both Stratonovich and It\^o noise are treated. 
\end{abstract}

\maketitle
\addtocontents{toc}{\protect\setcounter{tocdepth}{2}}


\section{Introduction}

The primitive equations are one of the fundamental models in geophysical fluid dynamics, and they are often used to analyze oceanic and atmospheric dynamics, see e.g.\  \cite{pedlosky2013geophysical,vallis2017atmospheric}. In the deterministic setting, the 3D primitive equations can be derived from the 3D (anisotropic) compressible Navier--Stokes equations in domains where the vertical scale is much smaller than the horizontal scale, by employing the Boussinesq and the small aspect ratio limit, see e.g.\  \cite[\S 1.2]{HH20}. The latter, also referred to as hydrostatic approximation, has been rigorously justified in various settings in \cite{AG01_hydrostatic,FGHHKW20,FGHKW25,LT19}. 

Our primary goal of this manuscript is to prove a large deviation principle (LDP in what follows) for the primitive equations with transport noise on the three-dimensional domain $\O\ceqq\Tor^2\times (-h,0)$, where $h>0$. A special case of the equations that we consider is given by the following system:  

\begin{equation}
\label{eq:primitive_intro}
\left\{
\begin{aligned}
&\dd v^\varepsilon -\Delta v^\varepsilon \dd t =  [-\nabla_{\h} P^\varepsilon - (v^\varepsilon\cdot \nabla_\h )v^\varepsilon - w^\varepsilon \partial_3 v^\varepsilon 
]\dd t+ \sqrt{\varepsilon}\, \textstyle{\sum_{n\geq 1}} \big[(\phi_n\cdot \nabla )v^\varepsilon -\nabla \wt{P}_n^\varepsilon\big]\circ \dd \beta^n_t,\\
&\dd \theta^\varepsilon -\Delta \theta^\varepsilon  \dd t =[- (v^\varepsilon\cdot \nabla_\h )\theta^\varepsilon - w^\varepsilon \partial_3 \theta^\varepsilon]\dd t 
+ \sqrt{\varepsilon}\, \textstyle{\sum_{n\geq 1}} (\psi_n\cdot \nabla )\theta^\varepsilon \circ \dd \beta^n_t,\\
&\nabla_\h \cdot v^\varepsilon + \partial_3 w^\varepsilon =0,\\
&\partial_3 P^\varepsilon+ \theta^\varepsilon=0, \\ 
&\partial_3 \wt{P}_n^\varepsilon=0,\\
& v^\varepsilon(0)=v_0, \quad \theta^\varepsilon(0)=\theta_0,
\end{aligned}
\right. 
\end{equation}
where the unknowns are  $v^\varepsilon\col[0,\infty)\times \Om\times \O\to \R^2$ and $w^\varepsilon\col[0,\infty)\times \Om\times \O\to \R$ denoting respectively the horizontal and vertical component of the velocity of the fluid,  the pressures $P^\varepsilon\col[0,\infty)\times \Om\times \O\to \R$ and $\wt{P}_n^\varepsilon\col[0,\infty)\times \Om\times \O\to \R$, and the temperature $\theta^\varepsilon\col[0,\infty)\times \Om\times \O\to \R$. Moreover, $\nabla_\h=(\partial_1,\partial_2)$ denotes the horizontal gradient, $(\beta^n)_{n\geq 1}$ is a sequence of standard Brownian motions on a filtered probability space $(\Om,\F,\P, (\F_t)_{t\geq 0})$, and $\circ$ denotes the Stratonovich product. 
The reader is referred to Section \ref{sec:prelims} for the remaining  unexplained notations. 
The above system is complemented with the following boundary conditions: 
\begin{equation}
\label{eq:boundary_conditions_strong_weak_intro}
\begin{aligned}
\pz  v^\eps (\cdot,-h)=\pz  v^\eps(\cdot,0)=0 \ \  \text{ on }\Tor^2,&\\
\pz  \theta^\eps(\cdot,-h)=\pz  \theta^\eps(\cdot,0)+\alpha \theta^\eps(\cdot,0)=0\ \  \text{ on }\Tor^2.&
\end{aligned}
\end{equation}

The mathematical study of primitive equations started with the seminal works by Lions, Temam, and Wang \cite{LTW92_1,LTW92_2,LTW92_3}, who obtained existence of Leray's type global solutions to the 3D primitive equations with initial data of finite energy, i.e.\ $v\in L^2(\O;\R^2)$ and $\theta_0\in L^2(\O)$.  
Later, a turning point was reached when Cao and Titi proved global well-posedness of the 3D primitive equations for initial data in $H^1(\O;\R^2)\times H^1(\O)$ in \cite{CT07}. 
The latter result was subsequently extended in many directions, see e.g.\   \cite{KZ07,HK16,GGHHK20,HH20} and the references therein. Moreover, after the seminal work by Cao and Titi \cite{CT07}, the primitive equations have attracted a lot of interest in the SPDE community, see e.g.\  \cite{AHHS22,AHHS25nonisothermal,BS21,DGHT11,DGHTZ12,DMM24,GKVZ14,GTW17_PRIMITIVE} and the references therein. 

The addition of additive and multiplicative noise to the primitive equations (or, more generally, models for geophysical flows) can be used, for instance, to account for numerical and empirical uncertainties and errors. Physical derivations of geophysical models with noise (including the stochastic primitive equations) can be found in e.g.\  \cite{AHHS25nonisothermal,AHN02,DM25,debussche2024second,H15_variational,holm2021stochastic,L25_Ober,P25_Ober}. 
In the current work, we focus on transport noise and its lower-order perturbation, and handle both It\^o noise and Stratonovich noise. See \eqref{eq:PE original} and  \eqref{eq:primitive_Stratonovich_boundary_conditions} for the full systems. There are several reasons for choosing transport noise, also known as Kraichnan's noise due to the fundamental works \cite{K68,K94}.  Firstly, transport noise can be physically motivated via two-scale expansion \cite{FP22,debussche2024second}, location uncertainty \cite{M14_local}, a stochastic Lagrangian viewpoint \cite{MR04_SIAM,FL23_book}, and many more  \cite{F08,BCF91,BF20}. 
In this context, the Stratonovich formulation arises naturally as it can be seen as a limit of smooth approximations of Brownian paths (such results are typically referred to as Wong--Zakai results). For this reason, the Stratonovich formulation is often considered more realistic from a physical point of view. 
Secondly, from a mathematical viewpoint, transport noise presents all the key challenges needed in handling most of the variants of stochastic perturbation of primitive equations, while being relatively simple to handle mathematically.

\subsection{Motivation and LDP} 
The effect of small noise on geophysical flows is known to be of central importance for oceanic and atmospheric dynamics  \cite{GH11_rare,galfi21,SGD}. In the presence of very small noise, one expects the solution to the stochastic primitive equations \eqref{eq:primitive_intro} (with small $\eps>0$) to resemble the solution to the deterministic primitive equations (i.e.\  \eqref{eq:primitive_intro} with $\varepsilon=0$). However,  a `rare event' may happen: with a non-zero probability, the small-noise solution may be significantly different from the deterministic solution. 
Being able to quantify the probabilities for such rare events is notably important when it concerns real world phenomena, and certainly in a climate context.   
To mention one relevant example, in \cite{SGD}, using a 2D Boussinesq model, it is shown   that the presence of a small noise could lead to an overturning collapse of the Atlantic Meridional Overturning Circulation, which would  severely affect the global climate. 
The work in this paper provides a mathematical foundation for potential similar analyses of 3D ocean and atmosphere models. 

One way to capture (extremely) small probabilities as described above is by means of an LDP. 
The latter provides an exponential rate of convergence in the small-noise limit $\varepsilon\downarrow 0$ of the solution $(v^\varepsilon,\theta^\eps)$ of \eqref{eq:primitive_intro} to the solution $(v^0,\theta^0)$ corresponding to 
\eqref{eq:primitive_intro} with $\varepsilon=0$. 
Roughly speaking, the LDP tells us the following: for (deterministic) paths $(V,\Theta)\col [0,T]\to H^1(\O;\R^2)\times L^2(\O)$, we have 
\[
\P((v^\eps,\theta^\eps)\text{ is close to } (V,\Theta))\approx e^{- I(V,\Theta)/\varepsilon}\; \text{ for   sufficiently small $\eps>0$}, 
\]
where $I(\cdot)$ is the \emph{rate function}. The \emph{rate} $I(V,\Theta)$ thus measures the likeliness of $(v^\eps,\theta^\eps)$ to realize (approximately) as the given  path $(V,\Theta)$. 
Above, we consider only the part $(v^\eps,\theta^\eps)$ of the solution to \eqref{eq:primitive_intro}, since the unknown vertical velocity $w^\eps$ and $\nabla P^\eps$ and $\nabla\wt{P}_n^\varepsilon$ are in turn fully determined by  $(v^\eps,\theta^\eps)$. 
For precise statements, we refer to Definition \ref{def: LDP} and Theorem \ref{th:PE LDP}. 

Having motivated the LDP, we now state a special case of our main results.  Let us mention that besides the Stratonovich transport noise discussed in this introduction, we will also treat It\^o transport noise.  The most general versions of our results can be found in Theorem \ref{th:PE LDP} and Theorem \ref{th:stratonovich}. 
\begin{theorem}[LDP for the primitive equations with Stratonovich transport noise]
\label{t:intro_3D_primitive}
Suppose that $(\phi_n)_{n\geq 1},(\psi_n)_{n\geq 1}\in C^{1+\alpha}(\O;\ell^2)$ for some $\alpha>0$, and $\nabla \cdot \psi_n=0$ in distribution on $\O$ for all $n\geq 1$. Moreover, letting $\phi^j_n$ and $\psi^j_n$ be the components of $\phi_n$ and $\psi_n$ for $j\in \{1,2,3\}$, respectively; assume that 
\begin{align}
\label{eq:phi1_phi2_condition_intro}
\phi_n^1,\phi_n^2&\text{ are independent of }x_3,\\
\label{eq:phi1_phi2_condition_intro2}
\phi_n^3 =\psi_n^3 &=0 \text{ on }\partial\O=\Tor^2\times \{-h,0\}.
\end{align}
Then, the solutions (Definition \ref{def:sol_strong_weak}) to the stochastic primitive equations \eqref{eq:primitive_intro} satisfy the LDP  for $\eps\downarrow 0$  (see Definition \ref{def: LDP}), with rate function given by 
$$
I(V,\Theta)=\frac12\inf\Big\{\int_{0}^T \|\varphi(t)\|_{\ell^2}^2\dd t  \,:\, (V,\Theta)=(V^\varphi,\Theta^\varphi)\Big\}
$$
where $(V^\varphi,\Theta^\varphi)$ is the unique solution to the following skeleton equation:   
\begin{equation}
\label{eq:primitive_intro skeleton} 
 \left\{
\begin{aligned}
& v' -\Delta v  = -\nabla_{\h} P - (v\cdot \nabla_\h )v - w  \partial_3 v 
+  \textstyle{\sum_{n\geq 1}}\big[(\phi_n\cdot \nabla )v -\nabla \wt{P}_n \big]\varphi_n,\\
& \theta'- \Delta \theta  = - (v\cdot \nabla_\h )\theta - w  \partial_3 \theta 
+  \textstyle{\sum_{n\geq 1}} (\psi_n\cdot \nabla )\theta \varphi_n,\\
&\nabla_\h \cdot v + \partial_3 w  =0,\\
&\partial_3 P + \theta=0, \\ 
&\partial_3 \wt{P}_n =0,\\
& v(0)=v_0, \quad \theta(0)=\theta_0, 
 \\[.45em]
& \text{ and $(v,\theta)$ satisfies the boundary conditions \eqref{eq:boundary_conditions_strong_weak_intro}}.\end{aligned}
\right. 
\end{equation} 
\end{theorem}

Before going further, let us comment on the assumptions in the above theorem. Firstly, the regularity of the noise coefficients shows that we can deal with $\phi$ given by a regular Kraichnan noise, cf.\ \cite[\S5]{GY25} and \cite[\S2]{agresti23roughtransport}.
The condition \eqref{eq:phi1_phi2_condition_intro2} is natural because it ensures that the flow generated by $\phi_n^3$ does not escape from the domain (in other words, it is a no-penetration condition). This is consistent with the interpretation of the transport noise through stochastic Lagrangian flows \cite{MR04_SIAM}.
On the other hand, the condition \eqref{eq:phi1_phi2_condition_intro} has only a partial physical meaning, see Section 2 (and especially Remark 2.2) in \cite{AHHS25nonisothermal}. However, the condition \eqref{eq:phi1_phi2_condition_intro} is at the moment the weakest condition on $\phi$ allowing for a global well-posedness of the stochastic system \eqref{eq:primitive_intro}.

\smallskip

\subsection{Related literature}
\label{ss:literature}

Let us mention some fundamental literature on large deviations in the context of fluid dynamics models. Many small-noise large deviation proofs, especially for SPDEs with multiplicative noise, rely on the weak convergence approach developed in \cite{budhidupuis01}. The latter method has been applied in \cite{sritharansundar06,duanmillet09,chueshovmillet10,RoZhZh,TV24} for  fluid models including the 2D Boussinesq equations, 2D Navier--Stokes equations, and 3D tamed Navier--Stokes equations. Also, large deviations for a class of SPDEs with transport noise were studied in \cite{GL24}. 

For the stochastic 3D primitive equations, small-noise large deviation principles have been proved in \cite{DZZ} for multiplicative, smoothening noise (see \cite[H0]{DZZ}) and in \cite{Slavik} for more general, gradient-dependent noise. In \cite{Slavik}, transport noise can be added to the vertical average of the horizontal velocity $\bar{v}=\frac1h \int_{-h}^0 v(\cdot,\zeta)\dd \zeta$ (see \cite[Ex.\ 2, (2.29)]{BS21}), which is often referred to as the \emph{barotropic mode}. However, to the best of our knowledge, the LDP for the physically relevant case (see e.g.\ \cite[\S2]{AHHS25nonisothermal}) of full transport noise acting on the horizontal velocity itself has not yet been treated in the literature. 
The challenge in dealing with transport noise acting on the horizontal velocity is that the noise does not disappear when analyzing the evolution of the \emph{baroclinic mode} $\wt{v}=v-\bar{v}$, and this fact makes the various a priori estimates needed to show the LDP much more involved, cf.\ \cite{AHHS22,AHHS25nonisothermal}.

The next subsection is devoted to explaining the strategy and novelty of our work. 

\subsection{Strategy and novelties}
\label{ss:strategy}
Compared to the existing literature, the main mathematical difficulties in proving Theorem \ref{t:intro_3D_primitive} (and its extension Theorem \ref{th:PE LDP}) can be summarized as follows:
\begin{itemize}
\item Lack of coercivity.
\item Transport noise acting directly on the horizontal velocity $v$.
\end{itemize}

We begin by discussing the first point: \textbf{lack of coercivity}. In recent years, an extensive literature on LDPs for SPDEs in the so-called variational setting has been developed, see e.g.\   \cite{liu09,hongliliu21,kumarmohan22,pan24,TV24}. 
Details on the variational approach to SPDEs can be found in \cite{LR15,AV22variational}. 
In such an approach, one reformulates an SPDE like \eqref{eq:primitive_intro} into the abstract form:  
\begin{equation}
\label{eq:see_intro}
\left\{
\begin{aligned}
&\dd Y^\varepsilon + A(t,Y^\varepsilon)\dd t=  \sqrt{\varepsilon} \, B(t,Y^\varepsilon)\dd W\; ({+\text{ It\^o--Stratonovich  correction}}),\\
&Y^\varepsilon(0)=(v_0,\theta_0), 
\end{aligned}
\right.
\end{equation}
where $Y^\varepsilon=(v^\varepsilon,\theta^\varepsilon)$, and $A \col \R_+\times V\to V^*$ and $B\col \R_+\times V\to \mathcal{L}_2(\ell^2,H)$ represent the coefficients in the SPDE \eqref{eq:primitive_intro}, with $(V,H,V^*)$ a Gelfand triple  (in particular, $V\hookrightarrow H\hookrightarrow V^*$) for which the well-posedness of \eqref{eq:see_intro} with initial data in $H$ holds, at least locally in time.  
Before going further, we comment on the choice of spaces $H$ and $V$ needed for local well-posedness of stochastic 3D primitive equations. 
As one can see from a scaling argument \cite[\S1.2]{agresti23roughtransport}, the 3D primitive equations have the same scaling as the 3D Navier--Stokes equations, the \emph{energy setting} ($H=L^2\times L^2$, $V=H^1 \times H^1$) is supercritical, and therefore is not possible to prove well-posedness.
Starting from the above mentioned seminal work \cite{CT07}, for the 3D primitive equations, the \emph{strong setting} ($H=H^1\times H^1$, $V=H^2\times H^2$) is often employed, and it is (almost)  critical, see \cite[\S1.2]{agresti23roughtransport}. 

Now, for the strong setting, up to this point, the stochastic primitive equations could be reformulated in the abstract form \eqref{eq:see_intro} (see Section \ref{sec: Strat} and \cite[\S8]{AHHS22} for a reformulation of the Stratonovich noise as an It\^o noise). 
However, it is impossible to apply the aforementioned LDP literature on abstract variational settings in this case, as we will explain next. 

The problem is that,
up to the best of our knowledge, all the available LDP results for abstract SPDEs as in \eqref{eq:see_intro}, have until now  required variants of the following \emph{coercivity} condition: There exist $M,\nu>0$ such that, for all $Y\in V$,
\begin{equation}
\label{eq:coercivity}
\langle A(Y), Y \rangle + \frac{1}{2}\|B(Y)\|_{\mathcal{L}_2(\ell^2;H)}^2\leq - \nu\|Y\|_{V}^2+M( \|Y\|_H^2+1),
\end{equation} 
where $\langle \cdot ,\cdot \rangle$ denotes the duality pairing between  $V^*$ and $V$. 
The above condition readily implies  $L^2(\Om;L^\infty(0,T;H)\cap L^2(0,T;V))$-estimates for solutions to \eqref{eq:see_intro} (see e.g.\   \cite{AV22variational,LR15}), and these estimates can be used to prove global well-posedness for \eqref{eq:see_intro} in a rather straightforward way. 
Although the coercivity condition \eqref{eq:coercivity} is often valid for 2D fluid models (see e.g.\  \cite[\S7.3]{AV25survey} or \cite{FL23_book}), it fails for 3D models. In particular, the coercivity condition \eqref{eq:coercivity} fails for the 3D primitive equations if $H=H^1\times H^1$ and $V=H^2\times H^2$, i.e.\ in the strong setting. Indeed, the global well-posedness of the 3D primitive equations relies on subtle cancellations and cannot be captured by a crude condition like \eqref{eq:coercivity}, in the deterministic nor the stochastic case.  
Consequently, energy estimates  are more difficult to derive, and in contrast to the coercive case, for the 3D primitive equations 
one only has the following type energy estimates at disposal (cf.\ \cite[Th.\ 3.6, Rem.\ 3.10]{AHHS25nonisothermal}):
\begin{equation}
\label{eq:bound_Om_weak}
\P\Big(\sup_{t\in [0,T]}\|Y \|_{H}^2 + \int_{0}^T \|Y \|_{V}^2\dd t  >\gamma \Big) \lesssim_T \frac{1}{\log \log \g} \quad \text{ for }\  \g> \mathrm{e}. 
\end{equation}
The above is expected to be optimal for solutions of 3D primitive equations, and is far weaker than the energy estimate that one would have if the coercivity condition \eqref{eq:coercivity} would have held.  

\smallskip

One of the main contributions of the current work is to obtain an LDP for an SPDE which satisfies only a very weak bound with respect to the $\Om$-variable. 
This in particular illustrates that the coercivity condition \eqref{eq:coercivity} is not needed to obtain an LDP result, an LDP can be proved with only a weak control on the $\Om$-variable. 
Yet, in our case, deriving the LDP from the weak convergence approach \cite{budhidupuis01}  from scratch does require a very lengthy proof. Instead, we leverage an abstract LDP  result by the second author \cite{T25L2LDP}, which is exactly tailored to $L^2$-settings (similar to variational settings) with a lack of coercivity. 
This latter result enables us,  in this paper, to focus solely on the PDE estimates and probabilistic arguments \emph{specific} to the stochastic primitive equations.  
To apply  \cite{T25L2LDP}, it suffices to define $(A,B)$ such that \eqref{eq:see_intro} is equivalent to \eqref{eq:primitive_intro}--\eqref{eq:boundary_conditions_strong_weak_intro}, and then verify the following:
\begin{enumerate}[label=\emph{\roman*.},ref=\textup{\roman*.}]
  \item\label{1} some structural assumptions on $(A,B)$ (no monotonicity or coercivity needed),
  \item\label{2} well-posedness of \eqref{eq:primitive_intro}--\eqref{eq:boundary_conditions_strong_weak_intro},
  \item\label{3} a deterministic a priori estimate  for the skeleton equation \eqref{eq:primitive_intro skeleton},
  \item\label{4} a stochastic a priori estimate that only needs to hold \emph{in probability} for an SPDE associated to \eqref{eq:primitive_intro}--\eqref{eq:boundary_conditions_strong_weak_intro}. 
\end{enumerate}
In particular, this also reveals the core ingredients for a successful application of the weak convergence approach.  
Since \ref{1}\ and \ref{2}\ were already proved in \cite{AHHS22}, we can focus on the two last-mentioned a priori estimates in the current paper. 
This streamlines our LDP proof, leaving us to the main challenge: proving the a priori estimates of \ref{3}\ and \ref{4}\ for two equations associated to the stochastic primitive equations. 

This then brings us to the difficulty mentioned in the second bullet: \textbf{transport noise acting directly on the horizontal velocity $v$}. Since our transport noise acts on the full velocity $v$, we have to deal with the 3D domain $\O$, and the available Sobolev embeddings are much more restrictive than in the case of the 2D domain $\Tor^2$, see also Subsection \ref{ss:literature}.  

A challenge in proving \ref{3}  is the treatment of the  terms in \eqref{eq:primitive_intro skeleton} that contain $(\varphi_n)$. These terms represent a $B(v,\theta)\varphi$-term for $B$ from the abstract formulation \eqref{eq:see_intro}. Now, $B(v,\theta)\varphi$ only belongs to $L^1(0,T;H)$ instead of $L^2(0,T;V^*)$, thus more standard $L^2$-estimates cannot be used. In fact, even well-posedness of \eqref{eq:primitive_intro skeleton} is far from trivial in our setting with transport noise acting on $v$, but it  follows from  \cite{T25L2LDP} once we have proved \ref{1}\ and \ref{3}. 

Other difficulties arise when verifying the stochastic a priori estimate mentioned in \ref{4}. As for \eqref{eq:primitive_intro}--\eqref{eq:boundary_conditions_strong_weak_intro}, we can again only work with weak global bounds in probability, and we will rely on intricate  applications of a stochastic Gr\"onwall lemma.

\smallskip

\subsubsection*{Acknowledgements}
The authors thank Mark Veraar for helpful comments and the anonymous referee for careful reading and suggestions.

\section{Preliminaries, assumptions and main result}\label{sec:prelims}
 
Over the next subsections, we discuss the necessary preliminaries on the stochastic primitive equations and present our assumptions and   main results. We begin with the case of It\^o noise and discuss Stratonovich noise in Subsection \ref{sub:stratresult}. 

Throughout this paper, $(\beta_n)$ is a sequence of independent real-valued Brownian motion on a filtered probability space $(\Omega,\F,\P,(\F_t))$. 
Consider the following system of equations: 
\begin{equation}\label{eq:PE original}
\left\{
\begin{aligned}
&\dd v^\varepsilon - \Delta v^\varepsilon  \dd t = \big[ -\nabla_{\h} P^\varepsilon - (v^\varepsilon\cdot \nabla_\h )v^\varepsilon - w^\varepsilon \partial_3 v^\varepsilon 
+ \fv (\cdot,v^\eps,\theta^\eps,\nabla v^\eps)+\mathcal{P}_{\gamma} (\cdot,v^\eps) \big]\dd t\\
&\qquad\qquad\qquad\qquad\qquad\qquad\qquad\qquad+ \sqrt{\varepsilon}\, \textstyle{\sum_{n\geq 1}} \big[(\phi_n\cdot \nabla )v^\varepsilon -\nabla \wt{P}_n^\varepsilon+\gvn(\cdot,v^\eps)\big]  \dd \beta^n_t,\\
&\dd \theta^\varepsilon- \Delta \theta^\varepsilon   \dd t = \big[- (v^\varepsilon\cdot \nabla_\h )\theta^\varepsilon - w^\varepsilon \partial_3 \theta^\varepsilon+ \ft(\cdot,v^\eps,\theta^\eps,\nabla v^\eps )  \big]\dd t\\
&\qquad\qquad\qquad\qquad\qquad\qquad\qquad\qquad
+ \sqrt{\varepsilon} \,\textstyle{\sum_{n\geq 1}}\big[ (\psi_n\cdot \nabla )\theta^\varepsilon+\gtn(\cdot,v^\eps,\theta^\eps,\nabla v^\eps)\big] \dd \beta^n_t,\\
&\nabla_\h \cdot v^\varepsilon + \partial_3 w^\varepsilon =0,\\
&\partial_3 P^\varepsilon+\kappa \theta^\varepsilon=0, \\ 
&\partial_3 \wt{P}_n^\varepsilon=0,\\
& v^\varepsilon(0)=v_0, \quad \theta^\varepsilon(0)=\theta_0,
\end{aligned}
\right. 
\end{equation}
on $\O=\Tor^2\times(-h,0)$, complemented with the same  boundary conditions \eqref{eq:boundary_conditions_strong_weak_intro} as in the introduction.  
We will reformulate  \eqref{eq:PE original} analogously to \cite[\S3]{AHHS22} and the references therein. 

\subsection{Notation}
Before we start, a few comments on our notation are in order. 
For $v=(v_1,v_2)\col\O\to \R^2$ we will write $\nabla v=(\partial_iv_j)_{i,j}=[\nabla v_1 \nabla v_2]\in\R^{3\times 2}$. 
We will   omit writing codomains $\R^n$ or $\R^{m\times n}$ in $L^p$-norms and $H^p$-norms when the dimensions are clear from the context.   
For example, we will write $\|v\|_{H^p(\O)}$, $\|\nabla v\|_{L^p(\O)}$ for $v$ as above. 
Similarly, we will sometimes write $\|\cdot\|_{\ell^2}$ instead of $\|\cdot\|_{\ell^2(\N;\R^2)}$ or $\|\cdot\|_{\ell^2(\N;\R^{3\times 2})}$, and in combined norms, we write  $L^p(\O;\ell^2)$ instead of $L^p(\O;\ell^2(\N;\R^2))$,    $L^2(0,t;L^p(\O))$ instead of $L^2(0,t;L^p(\O;\R^2))$  and $\ell^2(\N;L^2(\Tor^2))$ instead of $\ell^2(\N;L^2(\Tor^2;\R^{2\times 2}))$. 

We add a subscript $\h$ when we refer to the horizontal variables belonging to $\Tor^2$, e.g.\  we write 
\[
x=(x_{\h},x_3)\in\Tor^2\times (-h,0), \quad
\nabla_{\h}=(\partial_1,\partial_2), \quad \Delta_{\h}=\nabla_{\h}\cdot\nabla_{\h}. 
\]

We reserve angle brackets $\langle\cdot,\cdot\rangle$ for  the duality pairing between a Banach space $V$ and its dual space $V^*$, and we reserve round brackets $(\cdot,\cdot)_H$ for the inner product on a Hilbert space $H$. 

Lastly, we recall the identifications $\mathcal{L}_2(\ell^2;H^k(\O))\cong\ell^2(\N;H^k(\O))\cong H^k(\O;\ell^2)$ for $k\in\N_0$, which we use throughout this paper. 

\subsection{Preliminaries}\label{sub:prelim} 

First, let us introduce two linear operators that will be used in the reformulation. For a more detailed discussion we refer to \cite[\S2.1, \S3.1]{AHHS22}. 

For the temperature $\theta$, we will use an analytically weak setting and work with a  \emph{weak Robin Laplacian} $\DeltawR\col H^{1}(\O) \to H^{1}(\O)^*$, defined by 
\begin{align}\label{eq:weak robin pairing}
&(\DeltawR\theta)(\varphi)\ceqq  -\int_{\O}\nabla \varphi\cdot \nabla \theta\dd x - \alpha\int_{\Tor^2} \varphi(\cdot,0) \theta(\cdot,0) \dd x_{\h}, \qquad \theta,\varphi\in H^1(\O).
\end{align} 
The above definition is derived from an integration by parts for  $\<\Delta\theta,\varphi\?=(\Delta\theta,\varphi)_{L^2}$  when $\theta$ is smooth and satisfies the  Robin boundary conditions stated for $\theta^\eps$ in \eqref{eq:boundary_conditions_strong_weak_intro}. 

For the velocity $v$, we will use the \emph{hydrostatic Helmholtz projection} $\hhp\col L^2(\O;\R^2)\to L^2(\O;\R^2)$  defined by 
\begin{equation}\label{eq:def hhp}
    \hhp[g]\ceqq g- \Q_{\h}\Big[\frac{1}{h}\int_{-h}^0 g(\cdot,\zeta)\dd \zeta\Big], \qquad g\in L^2(\O;\R^2).
\end{equation}
Here, $\Q_{\h}\col L^2(\Tor^2;\R^2)\to L^2(\Tor^2;\R^2)$ is given by $\Q_{\h}f\ceqq \nabla_{\h} \Psi_{f}$ with $\Psi_{f}\in H^1(\Tor^2)$ the unique solution to $
\Delta_{\h}\Psi_f=\nabla_{\h}\cdot f$ and ${\int_{\Tor^2}}\Psi_f\dd x_{\h}=0$.  
Now, $\hhp$ is an orthogonal projection onto the following subspace of $L^2(\O;\R^2)$: 
\[
\mathbb{L}^2(\O)\ceqq  \{f\in L^2(\O;\R^2):\nabla_{\h}\cdot \int_{-h}^0 f(\cdot,\zeta)\dd \zeta=0 \text{ on }\Tor^2\} =\mathrm{Ran}(\hhp).
\]
The latter will be equipped with the norm inherited from $L^2(\O;\R^2)$: $\|\cdot\|_{\mathbb{L}^2(\O)}\ceqq \|\cdot\|_{L^2(\O;\R^2)}$. We denote by $\Q$  the complementary projection of $\hhp$:  
\[
\Q[g]\ceqq \Q_{\h}\Big[\frac{1}{h}\int_{-h}^0 g(\cdot,\zeta)\dd \zeta\Big], \qquad g\in L^2(\O;\R^2).
\]
In addition, we will be using the following subspaces of $H^k(\O;\R^2)$ for $k\in\N$:
\[
\mathbb{H}^k(\O)\ceqq H^k(\O;\R^2)\cap \mathbb{L}^2(\O), \qquad \|\cdot\|_{\mathbb{H}^k(\O)}\ceqq \|\cdot\|_{{H}^k(\O)},
\]
and for $k=2$, we consider the further subspace
\begin{equation}
\label{eq:def_H_2_N_strong_weak}
\Hs_{\n}^2(\O)\ceqq \big\{v \in \Hs^{2}(\O)\,:\, \pz  v(\cdot,-h)=\pz v(\cdot,0)=0 \text{ on }\Tor^2\big\}.
\end{equation}
The following property of the hydrostatic Helmholtz projection will be used (see \cite[(4.18)]{AHHS22}):  
\begin{equation}\label{eq: grad hhp}
   \|\partial_j\hhp f\|_{L^2(\O;\R^2)}\leq   \|\partial_j f\|_{L^2(\O;\R^2)},\qquad\quad f\in H^1(\O; \R^2),\, j\in\{1,2,3\}.
\end{equation} 

Finally, we use the following definition of large deviation principle (LDP). 

\begin{definition}\label{def: LDP}
Let $\mathcal{E}$ be a Polish space, let $(\Om,\F,\P)$ be a probability space and let $(Y^\eps)_{\eps\in(0,\eps_0)}$ be a collection of $\mathcal{E}$-valued random variables on $(\Om,\F,\P)$, for some $\eps_0>0$. Let $I\colon \mathcal{E}\to[0,\infty]$ be a function. Then $(Y^{\eps})$ \emph{satisfies the large deviation principle (LDP) on $\mathcal{E}$} with rate function $I\colon \mathcal{E}\to[0,\infty]$ if
\begin{enumerate}[label=\emph{(\roman*)},ref=\textup{(\roman*)}]
    \item $I$ has compact sublevel sets,
    \item for all open $E\subset \mathcal{E}$: $\liminf_{\eps\downarrow0}\eps\log \P(Y^{\eps}\in E)\geq -\inf_{z\in E}I(z)$,
    \item for all closed $E\subset \mathcal{E}$: $\limsup_{\eps\downarrow0}\eps\log \P(Y^{\eps}\in E)\leq -\inf_{z\in E}I(z)$.
\end{enumerate} 
\end{definition}

\subsection{Primitive equations reformulation and solution notion}

For the reformulation of 
the small-noise primitive equations \eqref{eq:PE original}, we start by noting that \eqref{eq:PE original} can be equivalently written in terms of the unknown $v=(v^k)_{k=1}^2 \col [0,\infty)\times \Om\times \O\to \R^2$ which contains only the first two components of the unknown velocity field $u$. To do so, we use that the divergence-free condition and boundary conditions are equivalent to setting $w=w(v)$, where 
\begin{equation}\label{eq:def_w_v}
	\big(w(v)\big)(t,x)\ceqq -\int_{-h}^{x_3}\nabla_{\h}\cdot v(t,x_{\h},\zeta)\dd\zeta, \qquad t\in \R_+, x=(x_{\h},x_3)\in\Tor^2\times (-h,0) =\O,
\end{equation}
and imposing that a.s.:  
\[
\int_{-h}^{0}\nabla_{\h}\cdot v (t,x_{\h},\zeta)\dd\zeta=0 \quad \text{for all $t\in \R_+$ and $x_{\h}\in \Tor^2$.}
\]
Then, we apply the hydrostatic Helmholtz
projection $\hhp$ (see Subsection \ref{sub:prelim} and \cite[(3.1)-(3.2)]{AHHS22}), so that the small-noise stochastic primitive equations \eqref{eq:PE original} become equivalent to 
\begin{equation}
\label{eq:primitive_weak_strong}
\begin{cases}
\dd v^\eps -\Delta v^\eps  \dd t=\hhp\big[ -(v^\eps\cdot \nabla_{\h}) v^\eps- w(v^\eps)\pz  v^\eps+\mathcal{P}_{\gamma} (\cdot,v^\eps) \\
\qquad\qquad \qquad \qquad\quad  
+\nabla_{\h}\int_{-h}^{x_3}  \kappa(\cdot,\zeta)\theta^\eps(\cdot,\zeta) \dd\zeta + \fv (\cdot,v^\eps,\theta^\eps,\nabla v^\eps) \big]\dd t \\ 
 \qquad \qquad\qquad \qquad \qquad \qquad 
+\sqrt{\eps}\sum_{n\geq 1}\hhp \big[(\phi_{n}\cdot\nabla) v^\eps  +\gvn(\cdot,v^\eps)\big] \dd\beta_t^n, \\
\dd \theta^\eps -\DeltawR \theta^\eps \dd t=\big[ -(v^\eps\cdot\nabla_{\h}) \theta^\eps -w(v^\eps) \pz  \theta^\eps+ \ft(\cdot,v^\eps,\theta^\eps,\nabla v^\eps ) \big]\dd t\\
 \qquad \qquad\qquad \qquad \qquad +\sqrt{\eps}\sum_{n\geq 1}\big[(\psi_n\cdot \nabla) \theta^\eps+\gtn(\cdot,v^\eps,\theta^\eps,\nabla v^\eps)\big] \dd\beta_t^n, 
\\
v^\eps(\cdot,0)=v_0,\quad \theta^\eps(\cdot,0)=\theta_0, 
\end{cases}
\end{equation}
on $\O=\Tor^2\times(-h,0)$, complemented with the following boundary conditions: 
\begin{equation}
\label{eq:boundary_conditions_strong_weak}
\begin{aligned}
\pz  v^\eps (\cdot,-h)=\pz  v^\eps(\cdot,0)=0 \ \  \text{ on }\Tor^2,&\\
\pz  \theta^\eps(\cdot,-h)= \pz  \theta^\eps(\cdot,0)+\alpha \theta^\eps(\cdot,0)=0\ \  \text{ on }\Tor^2.&
\end{aligned}
\end{equation}
Here $\alpha\in \R$ is given, $w(v^\eps)$ is as in \eqref{eq:def_w_v} and a.s.\ for all $t\in \R_+$,
\begin{align}\label{eq:def_P_gamma}
\mathcal{P}_{\gamma} (t,v)
&\ceqq  \Big(\sum_{n\geq 1} \sum_{m=1}^2 \gamma_n^{\ell,m} (t,x)
\big(\Q [(\phi_n\cdot\nabla) v + \gvn(\cdot,v)]\big)^m\Big)_{\ell=1}^2,
\end{align}
where $\big(\Q [\cdot]\big)^m$ denotes the $m$-th component of the vector $\Q [f]$ (see Subsection \ref{sub:prelim}).

Note that $(\beta^n)_{n\geq 1}$ induces a unique $\ell^2$-cylindrical Brownian motion  $\Br_{\ell^2}$ through
\begin{align}\label{eq:cyl}
\Br_{\ell^2}(f)\ceqq \sum_{n\geq 1} \int_{\R_+} \langle f(t), e_n\rangle \dd\beta^n_t, \ \hbox{ with } \ e_n=(\delta_{kn})_{k}, \,  f\in L^2(\R_+;\ell^2).
\end{align}

The following assumption corresponds to \cite[Ass.\ 3.1]{AHHS22} and will be needed for local well-posedness.

\begin{assumption} There exist $M,\delta>0$ for which the following hold.
\label{ass:well_posedness_primitive}
\begin{enumerate}[label=\textup{(\arabic*)},ref=\ref{ass:well_posedness_primitive}\textup{(\arabic*)}] 
\item\label{it:well_posedness_measurability} For all $n\geq 1$ and $j\in \{1,2,3\}$, the maps $$\phi_n^j,\psi_n^j,\kappa: \R_+\times \O\to \R$$ are Borel measurable;
\item\label{it:well_posedness_primitive_phi_smoothness} for all $t\in \R_+$, $j,k\in \{1,2,3\}$ and $\ell,m\in \{1,2\}$,
\begin{align*}
\Big\|\Big(\sum_{n\geq 1}| \phi^j_n(t,\cdot)|^2\Big)^{1/2} \Big\|_{L^{3+\delta}(\O)}+
\Big\|\Big(\sum_{n\geq 1}|\partial_k \phi^j_n(t,\cdot)|^2\Big)^{1/2} \Big\|_{L^{3+\delta}(\O)} &\leq M,\\
\Big\|\Big(\sum_{n\geq 1}|\gamma^{\ell,m}_n(t,\cdot)|^2\Big)^{1/2}\Big\|_{L^{3+\delta}(\O)}&\leq M;
\end{align*}
\item\label{it:well_posedness_primitive_L_infty_bound}  
for all $t\in \R_+$, $x\in \O$ and $j\in \{1,2,3\}$,
\begin{align*}
\Big(\sum_{n\geq 1} | \psi^j_n(t,x) |^2\Big)^{1/2} \leq M;
\end{align*}
\item\label{it:well_posedness_primitive_kone_smoothness}
for all $t\in \R_+$, $x_{\h}\in \Tor^2$, $j\in \{1,2,3\}$ and $i\in \{1,2\}$,
$$
\| \kappa(t,x_{\h},\cdot) \|_{L^2(-h,0)} +\|\partial_i \kappa(t,\cdot) \|_{L^{2+\delta}(\Tor^2;L^2(-h,0))} \leq M;
$$
\item\label{it:well_posedness_primitive_parabolicity} there exist $\nu\in (0,2)$ such that for all $t\in \R_+$, $x\in \O$ and $\xi\in \R^3$,
\begin{align*}
\sum_{n\geq 1} \Big(\sum_{j=1}^3 \phi^j_n(t,x) \xi_j\Big)^2\leq \nu |\xi|^2,
\ \ \text{ and }\ \ 
\sum_{n\geq 1} \Big(\sum_{j=1}^3 \psi^j_n(t,x) \xi_j\Big)^2 \leq \nu |\xi|^2;
\end{align*}
\item\label{it:nonlinearities_measurability}
for all $n\geq 1$, the maps 
\begin{align*}
&\fv\colon \R_+\times \O\times \R^2\times\R\times \R^6\to \R^2, \quad
\ft: \R_+\times \O\times \R^2\times\R\times \R^6\to \R,\\
&\gvn\colon\R_+\times \O\times\R^2\to \R^2, \quad\hbox{and} \quad
\gtn\colon\R_+\times \O\times \R^2\times\R\times \R^6\to \R
\end{align*} 
are  Borel measurable;

\item\label{it:nonlinearities_strong_weak} 
for all $T\in (0,\infty)$ and $i\in \{1,2\}$, 
\begin{align*}
\fv^i (\cdot,0),\ft(\cdot,0)&\in L^2((0,T)\times \O), \\
(\gvn^i(\cdot,0))_{n\geq 1}&\in  L^2((0,T);H^1(\O; \ell^2)) \hbox{ and } \\
 (\gtn(\cdot,0))_{n\geq 1}&\in L^2((0,T)\times \O; \ell^2).
\end{align*}
Moreover, for all $n\geq 1$, $t\in \R_+$,  $x\in \O$, $y,y'\in \R^2$, $Y,Y'\in \R^6$ and $z,z'\in \R$,
\begin{align*}
&|\fv(t,x,y,z,Y)-\fv(t,x,y',z',Y')|+
|\ft(t,x,y,z,Y)-\ft(t,x,y',z',Y')|\\
&
+
\|(\gtn(t,x,y,z,Y)-\gtn(t,x,y',z',Y'))_{n}\|_{\ell^2}\\
&\qquad\qquad
\lesssim (1+|y|^4+ |y'|^4)|y-y'|+
(1+|z|^{2/3}+|z'|^{2/3})|z-z'|\\
&\qquad\qquad
+(1+|Y|^{2/3}+|Y'|^{2/3})|Y-Y'|.
\end{align*}
Finally, for all $t\in\R_+$, $\O\times \R^2\ni (x,y)\mapsto \gvn(t,x,y)$ is continuously differentiable and for all $k\in\{0,1\}$, $j\in \{1,2,3\}$, $x\in \O$, and $y,y'\in \R^2$ 
\begin{align*}
&\|(\partial_{x_j}^{k}\gvn(t,x,y)-\partial_{x_j}^{k}\gvn(t,x,y'))_{n}\|_{\ell^2}
\lesssim (1+|y|^4+ |y'|^4)|y-y'|,\\
&\|(\partial_{y}\gvn(t,x,y)-\partial_{y}\gvn(t,x,y'))_{n}\|_{\ell^2}\lesssim (1+|y|^2+|y'|^2)|y-y'|.
\end{align*}
\end{enumerate}
\end{assumption}

\begin{remark}\label{rem:assumptionslocal} 
 In Assumption  \ref{it:well_posedness_primitive_phi_smoothness}, the derivatives are defined in the distributional sense. 
Together with the Sobolev embedding $H^{1,3+\delta}(\O;\ell^2)\into C^{\alpha}(\O;\ell^2)$ with  
$\alpha=\frac{\delta}{3+\delta} \in (0,1)$, it gives for  $j\in\{1,2,3\}$:
\begin{equation*}
\|(\phi^j_n(t,\cdot))_{n}\|_{L^\infty(\O;\ell^2)}\leq\|(\phi^j_n(t,\cdot))_{n}\|_{C^{\alpha}(\O;\ell^2)} \lesssim_{\delta} M, \quad
\text{ a.s.\ for all }t\in \R_+.
\end{equation*}

Furthermore, Assumption \ref{it:well_posedness_primitive_phi_smoothness} fits the scaling of the Kraichnan's noise, as discussed in  \cite[Introduction, Rem.\ 3.2]{AHHS22}, 
noting that $H^{3/2+\gamma}(\O)\into H^{1,3+\delta}(\O)$    for  $\gamma\in (0,1)$ and $\delta=\frac{3\gamma}{1-\gamma}$.
\end{remark}

Let us now also introduce the deterministic \emph{skeleton equation} for the stochastic primitive equations, which for $\varphi=(\varphi_n)_{n\geq 1}\in L^2(0,T;\ell^2)$ is defined by
\begin{align}\label{eq:primitive_skeleton}
\begin{cases}
v' -\Delta v =\hhp\big[ -(v\cdot \nabla_{\h}) v- w(v)\pz  v+\mathcal{P}_{\gamma} (\cdot,v) \\
\qquad\qquad \qquad \qquad\quad  
+\nabla_{\h}\int_{-h}^{x_3}  \kappa(\cdot,\zeta)\theta(\cdot,\zeta) \dd\zeta + \fv (\cdot,v,\theta,\nabla v) \big] \\ 
 \qquad \qquad\qquad \qquad \qquad \qquad 
+\sum_{n\geq 1}\hhp \big[(\phi_{n}\cdot\nabla) v  +\gvn(\cdot,v)\big]\varphi_n, \\
\theta' -\DeltawR \theta =\big[ -(v\cdot\nabla_{\h}) \theta -w(v) \pz  \theta+ \ft(\cdot,v,\theta,\nabla v ) \big] \\
 \qquad \qquad\qquad \qquad \qquad +\sum_{n\geq 1}\big[(\psi_n\cdot \nabla) \theta+\gtn(\cdot,v,\theta,\nabla v)\big]\varphi_n  , 
\\
v(\cdot,0)=v_0,\quad \theta(\cdot,0)=\theta_0, 
\end{cases}
\end{align}
and is equipped with the same boundary conditions \eqref{eq:boundary_conditions_strong_weak} as the stochastic primitive equations. 

For \eqref{eq:primitive_weak_strong}--\eqref{eq:boundary_conditions_strong_weak} and \eqref{eq:primitive_skeleton},\eqref{eq:boundary_conditions_strong_weak}, we use the strong-weak setting and define solutions as follows. 
 
\begin{definition}[$L^2$-strong-weak solutions] 
\label{def:sol_strong_weak}
Let Assumption \ref{ass:well_posedness_primitive} be satisfied and let $(v_0,\theta_0)\in \Hs^1(\O)\times L^2(\O)$. Let $T>0$. 
\begin{enumerate}[label=\textit{(\arabic*)},ref=\textit{(\arabic*)}]
\item\label{it:defsol1} 
We say that $(v^\eps,\theta^\eps)$ is 
an \emph{$L^2$-strong-weak solution} to \eqref{eq:primitive_weak_strong}--\eqref{eq:boundary_conditions_strong_weak} on $[0,T]$ if $v^\eps\col [0,T]\times \Omega\to \Hs_{\n}^2(\O)$ and $\theta^\eps\col [0,T]\times \Omega\to H^1(\O)$ are progressively measurable stochastic processes such that $(v^\eps,\theta^\eps)\in L^2(0, {T};\Hs_{\n}^2(\O)\times H^1(\O))$ a.s., 
and  
\begin{itemize}
\item a.s.:   
\begin{align}\label{eq:integrability_strong_weak}
\begin{split}
(v^\eps\cdot \nabla_{\h}) v^\eps+ w(v^\eps)\pz  v^\eps +\fv (v^\eps,\theta^\eps,\nabla v^\eps)+\mathcal{P}_{\gamma}(\cdot,v^\eps)&\in L^2(0, {T};L^2(\O;\R^2)),\\
-\nabla_{\h}{\cdot}(v^\eps \theta^\eps) -\partial_{3}(w(v^\eps)\theta^\eps)&\in L^2(0, {T};(H^1(\O))^*),\\
 \ft(v^\eps,\theta^\eps,\nabla v^\eps ) &\in L^2(0, {T};L^2(\O)),\\
(\gvn(v^\eps))_{n\geq 1 }&\in L^2(0, {T};H^1(\O;\ell^2(\N;\R^2))),\\
(\gtn(v^\eps,\theta^\eps,\nabla v^\eps))_{n\geq 1 }&\in L^2(0,{T};L^2(\O;\ell^2)),
\end{split}
\end{align} 
\item a.s.: the following identities hold  in $\mathbb{L}^2(\O)\times H^{-1}(\O)$ for all $t\in [0, T]$:  
\begin{align*} 
v^\eps(t)-v_0
&=\int_0^t\Big( \Delta v^\eps(s)+ \hhp\Big[\nabla_{\h}\int_{-h}^{x_3} \kappa(\cdot,\zeta)\theta^\eps(\cdot,\zeta) \dd\zeta\\
&\qquad
- (v^\eps\cdot \nabla_{\h}) v^\eps- w(v^\eps)\pz  v^\eps + \fv (v^\eps,\theta^\eps,\nabla v^\eps) +\mathcal{P}_{\gamma}(\cdot,v^\eps)\Big]\Big)\dd s\\
&\qquad
+\sqrt{\eps}\int_0^t \big(\hhp[ (\phi_{n}\cdot\nabla) v^\eps   +\gvn(v^\eps)] \big)_{n}\dd\Br_{\ell^2}(s),\\
\theta^\eps(t)-\theta_0
&=
\int_0^t  \big[\Dr \theta^\eps-\nabla_{\h}{\cdot}(v^\eps \theta^\eps) -\pz  (w(v^\eps) \theta^\eps)+ \ft(v^\eps,\theta^\eps,\nabla v^\eps )\big]\dd s\\
&+\sqrt{\eps}
\int_0^t \big( (\psi_{n}\cdot\nabla) \theta^\eps   +\gtn(v^\eps,\theta^\eps,\nabla v^\eps) \big)_{n}\dd \Br_{\ell^2}(s).
\end{align*}
\end{itemize}
\item Analogously, for the skeleton equation, we say that $(v,\theta)$ is  
an \emph{$L^2$-strong-weak solution} to \eqref{eq:primitive_skeleton}, \eqref{eq:boundary_conditions_strong_weak} on $[0,T]$ if  $v\col [0,T] \to \Hs_{\n}^2(\O)$ and $\theta\col [0,T] \to H^1(\O)$ are strongly measurable, 
  $(v,\theta)\in L^2(0,  {T};\Hs_{\n}^2(\O)\times H^1(\O))$ and \eqref{eq:integrability_strong_weak} holds and \eqref{eq:primitive_skeleton} holds in integrated form as an identity in $\mathbb{L}^2(\O)\times H^{-1}(\O)$. 
\end{enumerate}
 \end{definition}

For global well-posedness, we also need the following assumption, coinciding with \cite[Ass.\ 3.5]{AHHS22}.
 
 \begin{assumption}\label{ass:global_primitive}  
Assumption \ref{ass:well_posedness_primitive} holds and in addition:
\begin{enumerate}[label=\textup{(\arabic*)},ref=\ref{ass:global_primitive}\textup{(\arabic*)}]
\item\label{it:independence_z_variable} for all $n\geq 1$, $x=(x_{\h},x_3)\in \Tor^2\times (-h,0)=\O$, $t\in \R_+$, $j,k\in \{1,2\}$ and  
\begin{align*}
\phi_n^j(t,x)\text{ and }\gamma^{j,k}_n(t,x) \text{ are independent of $x_3$};
\end{align*} 
\item\label{it:sublinearity_Gforce}
there exist $C>0$ and $\Xi\in L^2_{\loc}(\R_+;L^2(\O))$ such that for all $t\in \R_+$, $j\in \{1,2,3\}$, $x\in \O$, $y\in \R^2$, $z\in \R$ and $Y\in \R^{6}$, 
\begin{align*}
|\fv(t,x,y,z,Y)|
&\leq C(\y(t,x)+|y|+|z|+|Y|),\\
|\ft(t,x,y,z,Y)|
&\leq C(\y(t,x)+|y|+|z|+|Y|),\\
\|(\gvn (t,x,y))_{n}\|_{\ell^2}
+\|(\partial_{x_j}\gvn (t,x,y))_{n}\|_{\ell^2}&\leq C(\y(t,x) + |y|), \\
\|(\partial_{y}\gvn(t,x,y))_{n}\|_{\ell^2}
&\leq C, \\
\|(\gtn(t,x,y,z,Y))_{n}\|_{\ell^2}
&\leq C(\y(t,x)+|y|+|z|+|Y|).
\end{align*}
\end{enumerate}
\end{assumption}

The well-posedness result below follows from the global well-posedness proved in \cite[Th.\ 3.6]{AHHS22}. 
\begin{theorem}\label{th:SPDE well posed}
Let Assumptions \ref{ass:well_posedness_primitive} and \ref{ass:global_primitive} hold, let $(v_0,\theta_0)\in\H^1(\O)\times L^2(\O)$ and let $\eps\in[0,1]$. Then for every $T>0$, the \eqref{eq:primitive_weak_strong}--\eqref{eq:boundary_conditions_strong_weak} has a unique $L^2$-strong-weak solution on $[0,T]$. 
\end{theorem}
\begin{proof}
  The result  follows from \cite[Th.\ 3.6]{AHHS22}, using that for $\eps\in[0,1]$, the rescaled functions  $(\sqrt{\eps}\phi_n,\sqrt{\eps}G_{v,n},\sqrt{\eps}\psi_n,\sqrt{\eps}G_{\theta,n})$ also satisfy the conditions for $(\phi_n,G_{v,n},\psi_n,G_{\theta,n})$ in \cite[Ass.\ 3.1, Ass.\ 3.5]{AHHS22}, thus the rescaled pair of coefficients $(A,\sqrt{\eps} B)$ satisfies \cite[Ass.\ 3.1, Ass.\ 3.5]{AHHS22}. The condition $\eps\leq 1$ is used to guarantee the parabolicity \cite[Ass.\ 3.1(3)]{AHHS22}.
\end{proof}

Our main result for the case of It\^o noise in as follows. For the notion of a large deviation principle, see Definition \ref{def: LDP}.

\begin{theorem}
    \label{th:PE LDP} 
    Let Assumptions \ref{ass:well_posedness_primitive} and \ref{ass:global_primitive} hold, let $(v_0,\theta_0)\in \H^1(\O)\times L^2(\O)$ and let $T>0$. For $\eps\in(0,1]$, let $(v^\eps,\theta^\eps)$ be the $L^2$-strong-weak solution to \eqref{eq:primitive_weak_strong}--\eqref{eq:boundary_conditions_strong_weak}. Then $((v^\eps,\theta^\eps))$ satisfies the large deviation principle on $L^2(0,T;\H_{\n}^2(\O)\times H^1(\O))\cap C([0,T];\H^1(\O)\times L^2(\O))$ with rate function  $I\colon L^2(0,T;\H_{\n}^2(\O)\times H^1(\O))\cap C([0,T];\H^1(\O)\times L^2(\O))\to [0,+\infty]$ given by
\begin{equation*}
I(V,\Theta)=\frac{1}{2}\inf\Big\{\textstyle{\int_0^T}\|\varphi(s)\|_{\ell^2}^2\dd s : \varphi\in L^2(0,T;\ell^2), \,(V,\Theta)=(V^{\varphi},\Theta^{\varphi})\Big\},
\end{equation*} 
where $\inf\varnothing\coloneqq +\infty$ and $(V^{\varphi},\Theta^{\varphi})$ is defined as the $L^2$-strong-weak solution $(v,\theta)$ to the skeleton equation \eqref{eq:primitive_skeleton}. 
\end{theorem}

\subsection{The case of Stratonovich-type transport noise}\label{sub:stratresult}

We now consider a variant of \eqref{eq:primitive_weak_strong}--\eqref{eq:boundary_conditions_strong_weak} driven by Stratonovich transport noise:
\begin{align}
&\begin{cases}
\displaystyle{\dd v^\eps -\Delta v^\eps\, \dd t=\hhp\Big[  -(v^\eps\cdot \nabla_{\h})v^\eps- w(v^\eps)\partial_3 v^\eps}\\
\qquad \qquad\quad\qquad  
 \displaystyle{+\nabla_{\h} \int_{-h}^{\cdot}  \kappa(\cdot,\zeta)\theta^\eps(\cdot,\zeta) \dd\zeta+ \fv(\cdot,v^\eps,\theta^\eps) \Big]\dd t }\\
\qquad \qquad \ \ \ \ \ \qquad  \qquad   \qquad   \ \ 
 \displaystyle{+\sqrt{\eps}\sum_{n\geq 1}\hhp [(\phi_{n}\cdot\nabla) v^\eps ]\circ \dd\beta_t^n}, 
& \ \ 
\text{on }\O,\\
\displaystyle{d \theta^\eps -\DeltawR \theta^\eps \dd t=\Big[ -(v^\eps\cdot \nabla_{\h})\theta^\eps- w(v^\eps)\partial_3 \theta^\eps + \ft(\cdot,v^\eps,\theta^\eps ) \Big]\dd t} \\
\qquad \qquad\qquad\qquad\ \ \ \ \  \qquad  \qquad  
\displaystyle{+\sqrt{\eps}\sum_{n\geq 1}(\psi_{n}\cdot\nabla) \theta^\eps \circ \dd\beta_t^n}, 
&\ \ 
\text{on }\O,\\
v^\eps(\cdot,0)=v_0,\ \ \theta^\eps(\cdot,0)=\theta_0, &\ \ 
\text{on }\O,
\end{cases}\label{eq:primitive_Stratonovich}\\[.6em]
&\quad\qquad\qquad\begin{aligned}
\partial_3 v^\eps (\cdot,-h)=\partial_3 v^\eps(\cdot,0)=0 \ \  \text{ on }\Tor^2,&\\
\partial_3 \theta^\eps(\cdot,-h)= \partial_3 \theta^\eps(\cdot,0)+\alpha \theta^\eps(\cdot,0)=0\ \  \text{ on }\Tor^2.&
\end{aligned}\label{eq:primitive_Stratonovich_boundary_conditions}
\end{align}
In \eqref{eq:primitive_Stratonovich},  $\circ$ denotes Stratonovich integration. 

We will assume the following for the coefficients, which coincides with \cite[Ass.\ 8.1, Ass.\ 8.4]{AHHS22}. 

\begin{assumption}
\label{ass:locglob_Stratonovich_strong_weak}
There exists $M,\delta>0$ for which the following hold.
\begin{enumerate}[label=\textup{(\arabic*)},ref=\textup{(\arabic*)}] 
\item\label{it:well_posedness_measurability_Stratonovich} For all $n\geq 1$ and $j\in \{1,2,3\}$, 
$\phi^j_n,\psi^j_n\col \O \to \R$ and $\kappa\col \R_+\times   \O\to \R$ 
are Borel measurable;
\item\label{it:well_posedness_primitive_phi_psi_smoothness_Stratonovich} 
for all  $x\in \O$ and $j,k\in \{1,2,3\}$,
\begin{align*}
\Big\|\Big(\sum_{n\geq 1}| \phi^j_n|^2\Big)^{1/2} \Big\|_{L^{3+\delta}(\O)}+
\Big\|\Big(\sum_{n\geq 1}|\partial_k \phi^j_n|^2\Big)^{1/2} \Big\|_{L^{3+\delta}(\O)} 
\leq M,&\\
\Big(\sum_{n\geq 1} | \psi^j_n(x) |^2\Big)^{1/2} + 
\Big\|\Big(\sum_{n\geq 1}|\nabla\cdot\psi_n|^2 \Big)^{1/2}\Big\|_{L^{3+\delta}(\O)}
\leq M;&
\end{align*}
\item\label{it:well_posedness_primitive_kone_smoothness_Stratonovich}
$\kappa$ satisfies Assumption \ref{it:well_posedness_primitive_kone_smoothness};
\item\label{it:independence_z_variable_Stratonovich}
for all $n\geq 1$, $j\in \{1,2\}$ and $x=(x_{\h},x_3)\in \Tor^2\times (-h,0)=\O$,   
$\phi_n^j(x)$ is independent of $x_3$;
\item\label{it:BCpsiphi} for all $n\geq 1$ and $j\in\{1,2\}$,
$$
\Big\|\sum_{n\geq 1} \phi_n^j(\cdot,0)\phi_n^3(\cdot,0)\Big\|_{H^{\frac{1}{2}+\delta}(\Tor^2)}
+
\Big\|\sum_{n\geq 1} \phi_n^j(\cdot,-h)\phi_n^3(\cdot,-h)\Big\|_{H^{\frac{1}{2}+\delta}(\Tor^2)}
\leq M;
$$
\item\label{it:psi_3_null_regularity_assumption_phi_boundary} 
for all $n\geq 1$ and  $x_{\h}\in \Tor^2$,
\begin{align*}
\psi_n^3(x_{\h},0)=
\psi_n^3(x_{\h},-h)=0 \text{ and }\phi^3_n(x_{\h},0)=
\phi^3_n(x_{\h},-h)=0;
\end{align*}
\item $\fv$ and $\ft$ are as in Assumption \ref{it:nonlinearities_strong_weak} and Assumption \ref{it:sublinearity_Gforce}.  
\end{enumerate}
\end{assumption}

As in \cite[\S8.2]{AHHS22} (see also \eqref{eq:Ito_corrections_strong_weak}--\eqref{eq:def_LTp} in Section \ref{sec: Strat} of this manuscript), the Stratonovich SPDE \eqref{eq:primitive_Stratonovich} can be reformulated as an SPDE in the It\^o form by considering some additional correction terms in the deterministic part. Hence, $L^2$-strong-weak solutions to \eqref{eq:primitive_Stratonovich}--\eqref{eq:primitive_Stratonovich_boundary_conditions} can be defined in full analogy with Definition \ref{def:sol_strong_weak}, modifying only the identities in the second bullet of part \ref{it:defsol1} accordingly. 
Due to \cite[Th.\ 8.5]{AHHS22}, 
Assumption \ref{ass:locglob_Stratonovich_strong_weak} already 
ensures global well-posedness of \eqref{eq:primitive_Stratonovich}--\eqref{eq:primitive_Stratonovich_boundary_conditions} for all $\eps\geq 0$. We now have the following LDP result. The proof of Theorem \ref{th:stratonovich} will be given in Section \ref{sec: Strat}.

\begin{theorem}\label{th:stratonovich}   
Let Assumption \ref{ass:locglob_Stratonovich_strong_weak} be satisfied. Let 
$(v_0,\theta_0)\in \H^1(\O)\times L^2(\O)$ and let $T>0$. For $\eps\in(0,1]$, let $(v^\eps,\theta^\eps)$ be the $L^2$-strong-weak solution to \eqref{eq:primitive_Stratonovich}--\eqref{eq:primitive_Stratonovich_boundary_conditions}. Then $((v^\eps,\theta^\eps))$ satisfies the large deviation principle on $L^2(0,T;\H_{\n}^2(\O)\times H^1(\O))\cap C([0,T];\H^1(\O)\times L^2(\O))$ with the same  rate function as that  in Theorem \ref{th:PE LDP}. 
\end{theorem}

\section{Proof of the LDP -- case of It\^o noise}\label{sec:proof main result}

Here we outline the proof of Theorem \ref{th:PE LDP}. Our strategy is to  reformulate the small-noise stochastic  primitive equations \eqref{eq:PE original} as an abstract stochastic evolution equation, and then  apply Theorem \ref{th:LDP general}.  
The  Stratonovich noise analog (Theorem \ref{th:stratonovich}) will later be treated in Section \ref{sec: Strat}.

We will rewrite the small-noise stochastic  primitive equations \eqref{eq:primitive_weak_strong}--\eqref{eq:boundary_conditions_strong_weak} in the form:  
\begin{equation}\label{eq:SPDE}
   \begin{cases}
  &\dd Y^\eps(t)=-A(t,Y^\eps(t))\dd t+\sqrt{\eps}B(t,Y^\eps(t))\dd \Br_{\ell^2}(t), \quad t\in[0,T], \\
  &Y^\eps(0)=u_0\in H,
\end{cases}
\end{equation} 
where $\Br_{\ell^2}$ is the $\ell^2$-cylindrical Brownian motion defined by \eqref{eq:cyl}, and we consider the following Gelfand triple:  
\begin{align}
    V=\H_{\n}^2(\O)\times H^1(\O),\qquad H=\H^1(\O)\times L^2(\O), \qquad V^*=\mathbb{L}^2(\O)\times H^{-1}(\O). \label{eq:Gelfand}
\end{align}
See Subsection \ref{sub:prelim} for the definitions of the spaces above. 
This choice of  Gelfand triple encodes the strong-weak setting, as well as the  divergence-free condition and part of the boundary conditions. For later convenience, we endow the spaces $V$ and $H$ with   inner products 
\begin{align}\label{eq:inner prod defs}
\begin{split}
&\big((v_1,\theta_1),(v_2,\theta_2)\big)_{V}\ceqq \big(v_1,v_2\big)_{\H^2(\O)}+ C_0 \big(\theta_1,\theta_2\big)_{ H^1(\O)}, \\
&\big((v_1,\theta_1),(v_2,\theta_2)\big)_{H}\ceqq \big(v_1,v_2\big)_{\H^1(\O)}+ C_0 \big(\theta_1,\theta_2\big)_{ {L}^2(\O)},
\end{split}
\end{align} 
where $C_0>0$ is a constant depending only on the $M$ and $\nu$ introduced in Assumption \ref{ass:well_posedness_primitive}, which will be fixed in Lemma \ref{lem:coer} below. 
Note that consequently, $C_0$ also appears in the pairing $\langle \cdot,\cdot \rangle\col V^*\times V\to \R$ induced by the inner product on $H$. 
We will consider solutions $Y^\eps$ to \eqref{eq:SPDE} that take values in the following maximal regularity space:
\begin{equation}\label{eq: def MR space}
    \MR(0,T)\coloneqq C([0,T];\H^1(\O)\times L^2(\O))\cap L^2(0,T;\H_{\n}^2(\O)\times H^1(\O)). 
\end{equation}

As explained in the introduction, with such a reformulation for initial data in $H$ defined above, a major challenge arises: $(A,B)$ is not coercive. In particular, the stochastic primitive equations cannot be studied in a variational framework, and for the LDP, one cannot utilize the various abstract large deviation results for variational settings such as \cite{pan24,TV24}.  
Instead, we will use the LDP results from \cite{T25L2LDP}, 
which is suitable for equations that have the same solution space $\MR(0,T)$ as in  the variational setting, but for which the coefficients are non-coercive. It offers practical sufficient conditions that ensure that the LDP holds for solutions $(Y^\eps)$ to \eqref{eq:SPDE}, see \ref{1}--\ref{4} in the Introduction. This reduces our amount of work significantly compared to proving an LDP via the weak convergence approach from the start. 
For convenience, special cases of the results from  \cite{T25L2LDP}  that suffice  for our applications are stated in Appendix \ref{appendix}, see Theorem \ref{th:LDP general}.

Now we will   rewrite the small-noise stochastic  primitive equations \eqref{eq:primitive_weak_strong}--\eqref{eq:boundary_conditions_strong_weak} in the form \eqref{eq:SPDE} with Gelfand triple $(V,H,V^*)$ defined by \eqref{eq:Gelfand}.  
Recall that the divergence-free condition and the  boundary conditions lead to a decomposition  $(u,\theta)=(v,w(v),\theta)$, where $v$ is the horizontal component of the velocity and $w(v)$ is defined by \eqref{eq:def_w_v}. Thus, for $\eps>0$, it suffices to consider $Y^\eps\ceqq(v^\eps,\theta^\eps)$. 

Now, analogous to \cite[(5.1)]{AHHS22}, under Assumption \ref{ass:well_posedness_primitive} for the stochastic primitive equations, we have the following reformulation of \eqref{eq:primitive_weak_strong}--\eqref{eq:boundary_conditions_strong_weak}: 
\begin{align}\label{reformulation}
&\text{
$Y^\eps$ is  an $L^2$-strong-weak solution to \eqref{eq:primitive_weak_strong}--\eqref{eq:boundary_conditions_strong_weak} in the sense of Definition \ref{def:sol_strong_weak}}\notag\\
&\hspace{5cm}\iff \\
&\text{ $Y^\eps$ is a  strong solution to \eqref{eq:SPDE} in the sense of Definition \ref{def:sol}, with } \notag\\
&\qquad\quad\text{ Gelfand triple \eqref{eq:Gelfand} and with $(A,B)$ defined by \eqref{eq:defAB}, \eqref{eq: defs coeff}}. \notag
\end{align}
Here, for the equivalence \eqref{reformulation} to hold, we define the coefficients $A,B$ in \eqref{eq:SPDE} as follows:  
\begin{equation}\label{eq:defAB}
  A(t,y)\ceqq A_0(t,y)-F(t,y), \qquad B(t,y)\ceqq B_0y+G(t,y),
\end{equation}
with  
\begin{align}\label{eq: defs coeff}
\begin{cases}
&A_0(t,v,\theta)=\begin{pmatrix}-\Delta v-\PP_{\gamma,\phi}(t,\cdot,v)-(\mathcal{J}_\kappa {\theta})(t,\cdot)\\-\DeltawR\theta\end{pmatrix},\\[11pt]
&B_0(v,\theta)=\begin{pmatrix}(\hhp[(\phi_n\cdot\nabla)v])_{n\geq 1}\\((\psi_n\cdot\nabla)\theta)_{n\geq 1}\end{pmatrix},\\[11pt]
&F(t,(v,\theta))=\begin{pmatrix} \hhp\big[ -(v\cdot\nabla_H)v-w(v)\pz v+F_v(t,\cdot,v,\theta,\nabla v) +\PP_{\gamma,G}(t,\cdot,v)\big] \\  -(v\cdot\nabla_H)\theta-w(v)\pz \theta+F_\theta(t,\cdot,v,\theta,\nabla v)  \end{pmatrix}, \\[11pt]
&G(t,(v,\theta))=\begin{pmatrix}
    (\hhp [G_{v,n}(t,\cdot,v)])_{n\geq 1}\\
    (G_{\theta,n}(t,\cdot,v,\theta,\nabla v))_{n\geq 1}
\end{pmatrix},
\end{cases}
\end{align} 
where
\[
(\mathcal{J}_\kappa {\theta})(t,x)\ceqq \nabla_H\int_{-h}^{x_3} \kappa(t,x_H,\zeta){\theta}(x_H,\zeta) \dd\zeta, \qquad t\in\R_+,x=(x_H,x_3)\in\O,{\theta}\in H^1(\O), 
\]
and $\PP_\gamma$ is defined by $\PP_\gamma\ceqq \PP_{\gamma,\phi}+\PP_{\gamma,G}$, with 
\begin{align}\label{eq:def Pgamma}
\begin{split}
&\PP_{\gamma,\phi}(t,x,v)\ceqq\Big(\sum_{n\geq 1}\sum_{m=1}^2 \gamma_n^{l,m}(t,x) \big(\Q[(\phi_n\cdot\nabla)v] \big)^m \Big)_{l=1}^2, \\ 
&\PP_{\gamma,G}(t,x,v)\ceqq \Big(\sum_{n\geq 1}\sum_{m=1}^2 \gamma_n^{l,m}(t,x) \big(\Q[G_{v,n}(\cdot,v)] \big)^m \Big)_{l=1}^2,
\end{split}
\end{align}
and $(\Q[\cdot])^m$ denotes the $m^{\text{th}}$ component of the vector $\Q[\cdot]$ defined in Subsection \ref{sub:prelim}. 

\begin{remark}
Let us note two minor details for the equivalence \eqref{reformulation}. Similar comments apply to the skeleton equation. 
\begin{itemize}
    \item Concerning ``$\Rightarrow$": for Definition \ref{def:sol} to hold, it is required that  $(v^\eps,\theta^\eps)\in C([0,T];\Hs^1(\O)\times L^2(\O))$ a.s.\ (besides belonging a.s.\ to $L^2(0, {T};\Hs_{\n}^2(\O)\times H^1(\O))$), but this is automatically satisfied due to \eqref{eq:integrability_strong_weak}. 
\item Reversely, for ``$\Leftarrow$": in Definition \ref{def:sol_strong_weak}, we require \eqref{eq:integrability_strong_weak}, containing stronger  integrability conditions than in Definition \ref{def:sol}. However, due to the structural conditions on the coefficients in Assumption \ref{ass:well_posedness_primitive}, the integrability requirements are equivalent. 
\end{itemize}
\end{remark}

With the reformulation \eqref{reformulation} at hand, one can apply Theorem \ref{th:LDP general} to the primitive equations, provided that Assumptions  \ref{ass:critvarsettinglocal} and \ref{ass:coer replace} are verified for $(A,B)$ defined by \eqref{eq:defAB}. This will exactly be our proof setup for  Theorem \ref{th:PE LDP}. The verification of Assumption \ref{ass:coer replace} for the above coefficients will occupy the remainder of this paper. 

As a warm up, let us right away verify that Assumption \ref{ass:critvarsettinglocal} is satisfied for the stochastic primitive equations. The measurability conditions in \ref{it:gen0} trivially follow from those in Assumption \ref{ass:well_posedness_primitive}. Therefore, we  begin  with   \ref{it:coercivelinear} in the next lemma. 

\begin{lemma}[Coercivity of the linear part]\label{lem:coer} 
Let $(V,H,V^*)$ and $(A_0,B_0)$  be as in  \eqref{eq:Gelfand} and  \eqref{eq: defs coeff}. Then there exists a $C_0>0$ such that when $V$ and $H$ are equipped with the inner products \eqref{eq:inner prod defs}, then there exist   $\hat{M},\hat{\nu}>0$ depending only on the constants in Assumption \ref{ass:well_posedness_primitive} such that 
$$
\langle A_0 (v,\theta),(v,\theta)\rangle 
-\frac{1}{2}\|B_0(v,\theta)\|_{\mathcal{L}_2(\ell^2,H)}^2
\geq \hat{\nu}\|(v,\theta)\|_{V}^2
-\hat{M}\|(v,\theta)\|_{H}^2 \qquad \text{for all $(v,\theta)\in V$}. 
$$
\end{lemma}

\begin{proof}
As in \cite[Rem.\ 4.3]{AHHS22}, by standard arguments it is enough to consider the case $\gamma\equiv 0$. 
In that case, we can write $A_0=A_1+A_2$, where 
$$
A_1 (v,\theta)= \begin{pmatrix}-\Delta v\\-\DeltawR\theta\end{pmatrix},
\qquad 
A_2 (v,\theta)= \begin{pmatrix} -(\mathcal{J}_\kappa {\theta})(t,\cdot)\\0\end{pmatrix}.
$$

Recalling \eqref{eq:weak robin pairing} and using continuity of the trace map $H^{1/2+\delta}(\O)\to L^2(\Tor^2)$ with $\delta\in(0,1/2)$, we have for all $\mu>0$:
\begin{align}\label{eq: est theta robin laplace}
(\DeltawR\theta)(\theta)\leq   
-\|\nabla \theta\|_{{L}^2(\O)}^2+|\alpha|C\|\theta(\cdot,0)\|_{L^2(\Tor^2)}^2 
&\leq   -\|\nabla \theta\|_{{L}^2(\O)}^2+|\alpha|C\tilde{C}\|\theta\|_{H^{1/2+\delta}(\O)}^2 \notag\\
&\leq  - (1 -\mu)\|\nabla \theta\|_{{L}^2(\O)}^2 + C_\mu \|\theta\|_{{L}^2(\O)}^2.
\end{align}
In the last line, we used interpolation and Young's inequality.  

Furthermore, we have by the Neumann boundary condition for $v$, partial integration and Kadlec's formula \cite[Lem.\ A.1 ($\beta=0$)]{AHHS22}:  
\begin{align}\label{eq:Delta v,v}
\<\Delta v,v\?_{\mathbb{L}^2(\O),\mathbb{H}_{\n}^2(\O)}=(\Delta v,v)_{L^2(\O)}-(\Delta v,\Delta v)_{L^2(\O)}
&=-\|\nabla v\|_{L^2(\O)}^2-\textstyle{\sum_{i,j=1}^3}\|\partial_{ij}^2v\|_{L^2(\O)}^2 \notag\\
&
=-\|v\|_{H^2(\O)}^2+\|v\|_{L^2(\O)}^2. 
\end{align}

Thus,  recalling \eqref{eq:inner prod defs} and the definition of $A_1$, and applying the above,  we have
\begin{align*}
\langle A_1 (v,\theta),(v,\theta)\rangle 
&= -\<\Delta v,v\?_{\mathbb{L}^2(\O),\mathbb{H}_{\n}^2(\O)}
-C_0(\DeltawR\theta)(\theta)\notag\\
&\geq \| v\|_{H^2(\O)}^2-\|v\|_{L^2(\O)}^2+C_0(1 -\mu)\|\nabla \theta\|_{{L}^2(\O)}^2 -C_0C_\mu \|\theta\|_{{L}^2(\O)}^2.
\end{align*} 
where $C_0>0$ is a constant that will be chosen in a moment. Now, arguing as in \cite[(4.19), Rem.\ 4.3]{AHHS22}, we have for all $\sigma>0$,
\begin{align*}
|\langle A_2 (v,\theta),(v,\theta)\rangle |
 \leq   \Big|\int_{\O} \Delta v \cdot \mathcal{J}_\kappa {\theta}\dd x  \Big| 
&\leq  \sigma \| v\|_{H^2(\O)}^2+ C_{\sigma,M} \| \nabla\theta\|_{L^2(\O)}^2,
\end{align*}
where $C_{\sigma,M}>0$ depends only on $\sigma$ and on the constant $M$ defined in Assumption \ref{ass:well_posedness_primitive}. The value of $\sigma$ will be fixed at the end of the proof. 
Next, we estimate the $B_0$-part. Note that
$$
\|B_0(v,\theta)\|_{\mathcal{L}_2(\ell^2,H)}^2
=
\|(\hhp [(\phi_n \cdot \nabla) v])_{n}\|_{H^1(\O)}^2
+
C_0\|( (\psi_n \cdot \nabla) \theta)_{n}\|_{L^2(\O)}^2.
$$
We estimate the two terms on the right-hand side separately. By the pointwise estimate in Assumption \ref{it:well_posedness_primitive_parabolicity},
\begin{equation*}
\|( (\psi_n \cdot \nabla) \theta)_{n}\|_{L^2(\O)}^2
= \int_{\O}  \sum_{n\geq 1}| (\psi_n\cdot\nabla) \theta |^2 \dd x \leq \nu 
\|\nabla \theta\|_{L^2}^2.
\end{equation*}
For the other term, note that by the boundedness of $\hhp$ and $(\phi_n)_{n\geq 1}$ (Remark \ref{rem:assumptionslocal}), we have
\begin{align*}
\|(\hhp [(\phi_n \cdot \nabla) v])_{n}\|_{H^1(\O)}^2
&=\|(\nabla\hhp [(\phi_n \cdot \nabla) v])_{n}\|_{L^2(\O)}^2
+\|(\hhp [(\phi_n \cdot \nabla) v])_{n}\|_{L^2(\O)}^2\\
&\leq \|(\nabla\hhp [(\phi_n \cdot \nabla) v])_{n}\|_{L^2(\O)}^2
+C\| \nabla v\|_{L^2(\O)}^2.
\end{align*}
For the first term in the last line, \cite[(4.26)]{AHHS22} yields for any  $\nu'$ and $\nu''$ with  $\nu<\nu'<\nu''<2$: 
\begin{equation}\label{eq: coerc est}
        \|(\nabla\hhp [(\phi_n \cdot \nabla) v])_{n}\|_{L^2(\O;\ell^2)}^2
\leq \nu'' \|v\|_{H^2(\O)}^2
+ C_{\nu,\nu',\nu''} \|v\|_{H^1(\O)}^2.
\end{equation}
Now fix any $\nu'$ and $\nu''$ as above. Fix a $\sigma\in(0,1-\nu''/2)$, fix an $\mu\in(0,1-\nu/2)$ and fix a $C_0>0$ such that $C_0(1-\mu-{\nu}/{2})-C_{\sigma,M}>0$. Then the claim of the lemma follows from the estimates above (note the factor $1/2$ before the $B_0$-term). 
\end{proof}

Next, we turn to Assumption \ref{it:growth AB}. 

\begin{lemma}[Boundedness and local Lipschitz estimates]\label{lem:lip}
Let $(V,H,V^*)$ and $(A_0,B_0,F,G)$  be as in  \eqref{eq:Gelfand} and  \eqref{eq: defs coeff}. Then $(A_0,B_0, F, G)$ satisfies   Assumption \ref{it:growth AB}. Furthermore, $F(\cdot,0)\in  L^2_\loc(\R_+;V^*) $ and $G(\cdot,0)\in L^2_\loc(\R_+;\UH)$.
\end{lemma}

\begin{proof}
Making use of Assumption \ref{it:sublinearity_Gforce}, one readily checks that $F(\cdot,0)\in  L^2_\loc(\R_+;V^*) $ and $G(\cdot,0)\in L^2_\loc(\R_+;\UH)$. 
Moreover, note that $A_0(t,\cdot) \in\mathcal{L}(V,V^*)$
 (using \eqref{eq:weak robin pairing}), and 
 $B_0 \in\mathcal{L}(V,\mathcal{L}_2(\ell^2;H))$ (using Remark \ref{rem:assumptionslocal} and Assumption \ref{it:well_posedness_primitive_L_infty_bound}), with bounds that are uniform over $t\in[0,T]$.
In particular, the growth and Lipschitz condition for $A_0$ and $B_0$ are satisfied. 

For the terms $F$ and $G$, we use \cite[\S5.1]{AHHS22}.  
Note that in the notation \cite[\S5.1]{AHHS22}, we have $V^*=X_0$, $H=X_{1/2}$, $V=X_1$, $F= F_1+F_2+F_3$, $G=G$,   and in the notation of \cite[Assumption H, Th.\ 4.8]{AV22nonlinear1}, $F=0+F_c+0$, $G=0+G_c+0$. So by \cite[\S5.1]{AHHS22}, we have
\begin{align*}
  &  \|F_1(\cdot,(v,\theta))-F_1(\cdot,(v',\theta'))\|_{V^*}
    \lesssim (\|(v,\theta)\|_{V_{3/4}}+\|(v',\theta')\|_{V_{3/4}})\|(v,\theta)-(v',\theta')\|_{V_{3/4}},\\
   & \|F_2(\cdot,(v,\theta))-F_2(\cdot,(v',\theta'))\|_{V^*}
    \lesssim (1+\|(v,\theta)\|_{V_{3/5}}^4+\|(v',\theta')\|_{V_{3/5}}^4)\|(v,\theta)-(v',\theta')\|_{V_{3/5}}\\
 &   \quad+ (1+\|(v,\theta)\|_{V_{4/5}}^{2/3}+\|(v',\theta')\|_{V_{4/5}}^{2/3})\|(v,\theta)-(v',\theta')\|_{V_{4/5}},
\end{align*}
\begin{align*}
    \|F_3(\cdot,(v,\theta))-F_3(\cdot,(v',\theta'))\|_{V^*}+&\|G(\cdot,(v,\theta))-G(\cdot,(v',\theta'))\|_{\mathcal{L}_2(\ell^2;H)}\\
    &\,\lesssim (1+\|(v,\theta)\|_{V_{3/5}}^4+\|(v',\theta')\|_{V_{3/5}}^4)\|(v,\theta)-(v',\theta')\|_{V_{3/5}}\\
    &\quad+(1+\|(v,\theta)\|_{V_{2/3}}^2+\|(v',\theta')\|_{V_{2/3}}^2)\|(v,\theta)-(v',\theta')\|_{V_{2/3}}, 
\end{align*}  
where for $\beta\in(0,1)$, $V_\beta\ceqq [V^*,V]_\beta$ denotes the complex interpolation space. 
\end{proof}

In Section \ref{sec:skeleton}, we will verify Assumption \ref{ass:coer replace}\ref{it:2} for the stochastic primitive equations, by proving the following proposition. 

\begin{proposition}\label{prop:skeleton}
Let $(v_0,\theta_0)\in \H^1(\O)\times L^2(\O)$, let $T>0$ and let $\varphi\in L^2(0,T;\ell^2)$. Suppose that $(v,\theta)$ is an $L^2$-strong-weak solution to the corresponding skeleton equation \eqref{eq:primitive_skeleton} on $[0,T]$. Then:  
\begin{equation*}
   \|(v,\theta)\|_{\MR(0,{T})}   \leq C_{(v_0,\theta_0)}({T},\|\varphi\|_{L^2(0,{T};U)}),
  \end{equation*}
   with $C_{(v_0,\theta_0)}\col \R_+\times \R_+\to\R_+$ a function that is non-decreasing in both components and with  $\MR(0,T)$ defined by \eqref{eq: def MR space}. 
\end{proposition}

Then, Section \ref{sec:tiltedspde} will contain  the verification of  Assumption \ref{ass:coer replace}\ref{it:3}:  

\begin{proposition}\label{prop:tilt}
Let $(v_0,\theta_0)\in \H^1(\O)\times L^2(\O)$ and let $T>0$. 
Suppose that $(\bphieps)_{\eps>0}$ is a collection of predictable stochastic processes and suppose that $K\geq 0$ is such that for each $\eps>0$, we have 
\begin{equation}\label{eq: as bdd K}
    \|\bphieps\|_{L^2(0,T;U)}\leq K  \text{ a.s.}
\end{equation}
For all $\eps\in(0,1]$, there exists a unique  $L^2$-strong-weak solution   $X^\eps=(v^\eps,\theta^\eps)$ to \eqref{eq:SPDE tilted} corresponding to $\bphieps$, and it holds that 
\[
\lim_{\gamma\to\infty}\sup_{\eps\in(0,\eps_0)}\P(\|X^\eps\|_{\MR(0,T)}>\gamma)=0,
\]
for some $\eps_0\in(0,1)$ that is independent of $T$, $K$ and $(\bphieps)$, and with 
$\MR(0,T)$ defined by \eqref{eq: def MR space}. 
\end{proposition}

\begin{remark}\label{rem: tilt well posed}
  The fact that for all $\eps\in(0,1]$, \eqref{eq:SPDE tilted} has a unique $L^2$-strong-weak solution follows from Theorem \ref{th:SPDE well posed} and Girsanov's theorem. For details, see the argument in \cite[Lem.\ 4.2]{T25L2LDP}. 
\end{remark} 

Now taking for granted the above results,  we are already able to give the proof of our main result: the large deviation principle for the stochastic primitive equations with transport noise.  

\begin{proof}[\textbf{Proof of Theorem \ref{th:PE LDP}}]
Recall the reformulation \eqref{reformulation}.  
We verify that Assumptions \ref{ass:critvarsettinglocal} and \ref{ass:coer replace} are  satisfied for $(A,B)$ defined by \eqref{eq:defAB}, and with $\eps_0\in(0, 1)$ from Proposition \ref{prop:tilt}. Then the result follows directly from   Theorem \ref{th:LDP general}. 

Assumption \ref{ass:critvarsettinglocal} is satisfied by Lemmas \ref{lem:coer},  \ref{lem:lip} and Remark \ref{rem:growth from lip}. 
Furthermore, \ref{ass:coer replace}\ref{it:1} follows from Theorem \ref{th:SPDE well posed}, using that $\eps\in(0,\eps_0)\subset(0,1)$. 
Finally, \ref{ass:coer replace}\ref{it:2} follows from Proposition  \ref{prop:skeleton} and \ref{ass:coer replace}\ref{it:3} follows from Proposition \ref{prop:tilt}. 
Thus  Theorem \ref{th:LDP general} yields the claimed LDP.
\end{proof}

\section{Skeleton equation}\label{sec:skeleton}

In this section, we prove Proposition \ref{prop:skeleton}. That is, we prove an a priori estimate for the skeleton equation \eqref{eq:primitive_skeleton} for the stochastic primitive equations. 

The proof of Proposition \ref{prop:skeleton} will be executed in three substantial steps, which are distributed over the next  three subsections. 
Eventually, we need to obtain an energy estimate in the $\MR(0,T)$-norm (see \eqref{eq: def MR space}), which corresponds to the analytically strong setting for $v$. 
To do so, we use a combination of estimates similar to those in the celebrated work of \cite{CT07}, and we use the proof structure of \cite{HH20,AHHS22}.
\begin{itemize}
    \item Subsection \ref{sub:skeleton L2}: proof of an energy estimate in the analytically weak setting, i.e.\   in the norm of $L^2(0,T;\H^1(\O)\times L^2(\O))\cap C([0,T];\mathbb{L}^2(\O)\times L^2(\O))$. 
    \item Subsection \ref{sub:skeleton intermed}: intermediate estimates for a coupled system $(\bar{v},\tilde{v})$ that  uniquely determines $v$, containing analytically strong estimates for $\bar{v}$ and $L^4$-estimates for $\tilde{v}$. 
    \item Subsection \ref{sub:skeleton H1}: derivation of the stated $\MR(0,T)$-estimate.
\end{itemize}

Throughout the remainder of this section, we assume that  $(v_0,\theta_0)\in \H^1(\O)\times L^2(\O)$,  $T>0$ and   $\varphi\in L^2(0,T;\ell^2)$ are given, and  as in the statement of Proposition \ref{prop:skeleton}, we let $(v,\theta)$ be an $L^2$-strong-weak solution to \eqref{eq:primitive_skeleton} on $[0,T]$.

\subsection{$L^2$-estimates for the horizontal skeleton velocity and temperature}\label{sub:skeleton L2} 
The goal of this subsection is to prove an estimate:
\begin{align}\label{eq:skeleton L^2 est prep}
\begin{split}
\|v(t)\|_{{L}^2(\O)}^2+&\|\theta(t)\|_{{L}^2(\O)}^2
-\|v_0\|_{L^2(\O)}^2-\|\theta_0\|_{L^2(\O)}^2 \\
&\quad\lesssim R_t+\int_0^t \big(1+\|\varphi(s)\|_{\ell^2}^2\big)\Big(\|v(s)\|_{{L}^2(\O)}^2+\|\theta(s)\|_{{L}^2(\O)}^2\Big)\dd s\\
&\qquad-\alpha\Big[\|\nabla v\|_{L^2(0,t;{L}^2(\O))}^2+ \|\nabla \theta\|_{L^2(0,t;{L}^2(\O))}^2\Big],
\end{split}
\end{align}
for a constant $R_t\geq 0$ non-decreasing in $t$ and a constant  $ \alpha>0$. Then, Gr\"onwall's inequality directly yields for all $t\in[0,T]$:
\begin{align}\label{eq:skeleton L^2 est}
\begin{split}
\|v(t)\|_{{L}^2(\O)}^2+\|\theta(t)\|_{{L}^2(\O)}^2+ &\|v\|_{L^2(0,t;H^1(\O))}^2
 +\|\theta \|_{L^2(0,t;H^1(\O;\R))}^2\\&\lesssim R_TC_{v_0,\theta_0}\exp\big(c_1(T+\|\varphi\|_{L^2(0,T;\ell^2)}^2)\big) 
 \eqqc {C}_{1,v_0,\theta_0,T,\varphi}
\end{split}
\end{align} 
with a constant $c_1>0$ independent of $\varphi$, $T$, $v$ and $\theta$.  

Recall \eqref{eq: defs coeff} for the definition of $A$ and $B$. 
By the  chain rule of \cite[Lem.\ 2.2 p.\ 30]{pardoux} for $L^2(0, T ; V^*)+L^1(0, T ; H)$-integrands (see also \cite[ Lem.\ A.2]{TV24}), applied with Gelfand triple $(\mathbb{H}^1(\O)\times H^1(\O),\mathbb{L}^2(\O)\times L^2(\O),(\mathbb{H}^{1}(\O))^*\times H^{-1}(\O))$,  
we have for  all $t\in[0,T]$:  
\begin{align}
    \frac{1}{2}\|v(t)\|_{{L}^2(\O)}^2-\frac{1}{2}\|v_0\|_{{L}^2(\O)}^2 
    &=\int_0^t \int_\O \Delta v(s)\cdot v(s)\dd x\dd s -\int_0^t \int_\O \hhp[(u(s)\cdot \nabla)v(s)]\cdot v(s)\dd x\dd s\notag\\ 
    &\quad+\int_0^t \int_\O \hhp[\PP_\gamma(\cdot,v)]\cdot v(s)\dd x\dd s+\int_0^t \int_\O \hhp[\mathcal{J}_\kappa{\theta}(s)(s,x)]\cdot v(s)\dd x\dd s\notag\\ 
    &\quad+\int_0^t \int_\O \hhp[F_v(\cdot,v,\theta,\nabla v)]\cdot v(s)\dd x\dd s\notag\\
    &\quad+ \int_0^t \int_\O \sum_{n\in\N}\varphi_n(s)\hhp[(\phi_n\cdot\nabla)v(s)]\cdot v(s) \dd x \dd s\notag\\
    &\quad+\int_0^t \int_\O\sum_{n\in\N}\varphi_n(s)\hhp[G_{v,n}(s,v(s))]\cdot v(s)\dd x\dd s\notag\\
    &\eqqc I_1^{\scriptscriptstyle L^2}(t,v,\theta,\varphi), \label{eq:L^2 norm v}
\end{align} 
and
\begin{align}
    \frac{1}{2}\|\theta(t)\|_{{L}^2(\O)}^2-\frac{1}{2}\|\theta_0\|_{{L}^2(\O)}^2 
    &=\int_0^t (\DeltawR\theta)(\theta)\dd s -\int_0^t \int_\O [(u(s)\cdot\nabla)\theta(s)] \theta(s)\dd x\dd s \notag\\  
    &\quad+\int_0^t \int_\O F_{\theta}(s,v(s),\theta(s),\nabla v(s)) \theta(s)\dd x\dd s\notag\\  
    &\quad+ \int_0^t \int_\O \sum_{n\in\N}\varphi_n(s)[(\psi_n\cdot\nabla)\theta(s)] \theta(s) \dd x \dd s\notag\\  
    &\quad+\int_0^t \int_\O\sum_{n\in\N}\varphi_n(s)G_{\theta,n}(s,v(s),\theta(s),\nabla v(s)) \theta(s)\dd x\dd s\notag\\
    &\eqqc I_2^{\scriptscriptstyle L^2}(t,v,\theta,\varphi). \label{eq:L^2 norm theta}
\end{align} 
We estimate the terms constituting \eqref{eq:L^2 norm v} and \eqref{eq:L^2 norm theta} in order of appearance. 
The upcoming estimates hold pointwise in $s$  for $s\in[0,t]$, thus below we omit writing the time input $s$ in the integrands. By the Neumann boundary conditions for $v$, we  have  
\[
\int_\O \Delta v\cdot v\dd x=-\|\nabla v\|_{{L}^2(\O)}^2.
\]
Moreover, $v\in \mathbb{H}^2_{\n}\subset \mathrm{Ran}(\hhp)$ and $\hhp=\hhp^*$, hence
\begin{align*}
\int_\O \hhp[(u \cdot \nabla)v]\cdot v \dd x&= \int_\O [(u \cdot \nabla)v]\cdot \hhp v\dd x= \int_\O [(u\cdot \nabla)v]\cdot  v\dd x=0.
\end{align*}
For the last equality, note that by the boundary conditions, we have $\int_\O u_j(\partial_j v_k)v_k\dd x=\frac{1}{2}  \int_\O u_j\partial_j(v_k^2)\dd x$ for $k\in\{1,2\}$ and  $j\in\{1,2,3\}$, thus  
\begin{align*}
\int_\O[(u\cdot \nabla)v]\cdot v\dd x  
&=\frac{1}{2} \sum_{k=1}^2\int_\O u\cdot(\nabla v_k^2)\dd x
=\frac{1}{2} \sum_{k=1}^2\Big(-\int_\O (\nabla\cdot u)v_k^2\dd x +\int_{\partial\O} (v_k^2u)\cdot\mathbf{n}\dd \Gamma \Big)=0.
\end{align*} 
Here, the last expression is zero since $u$ is divergence free and the  surface integral vanishes on $\partial \mathbb{T}^2\times(-h,0)$ by periodicity of $u=(v,w)$ and it vanishes on $\mathbb{T}^2\times\{-h,0\}$  
by the boundary condition \cite[(1.3)]{AHHS22}. 
Next,  we have by \cite[Rem.\ 4.3]{AHHS22}:
\begin{align*}
    \Big|\int_\O \hhp[\PP_\gamma(\cdot,v)]\cdot v\dd x\Big|&=\Big|\int_\O [\PP_{\gamma,\phi}(\cdot,v)+\PP_{\gamma,G}(\cdot,v)]\cdot v\dd x\Big|\\
    &\leq \eps\|\nabla v\|_{L^2(\O)}^2+C_{\eps}\|v\|_{L^2(\O)}^2+\int_\O|\PP_{\gamma,G}(\cdot,v)||v|\dd x. 
\end{align*}
Let $\rho\in(2,6)$ satisfy $\frac{1}{3+\delta}+\frac{1}{2}+\frac{1}{\rho}=1$. In particular, $H^1(\O;\R^2)\into L^\rho(\O;\R^2)$. Using the latter, together with H\"older's inequality,  Assumption \ref{it:well_posedness_primitive_phi_smoothness}, Assumption \ref{it:sublinearity_Gforce}  and Young's inequality, we can   estimate further: 
\begin{align*}
    \int_\O|\PP_{\gamma,G}(\cdot,v)||v|\dd x &\leq \sum_{l,m=1}^2 \int_\O\|(\gamma_n^{l,m})_n\|_{\ell^2}\|(\Q[G_{v,n}(\cdot,v)])_n\|_{\ell^2}|v|\dd x\\ 
    &\lesssim \sum_{l,m=1}^2 \int_\O\|(\gamma_n^{l,m})_n\|_{\ell^2}(|\Xi|+|v|)|v|\dd x\\
    &\leq \sum_{l,m=1}^2\|(\gamma_n^{l,m})_n\|_{L^{3+\delta}(\O;\ell^2)}(\|\Xi\|_{L^2(\O)}+\|v\|_{L^2(\O)})\|v\|_{L^\rho(\O)}\\  
    &\lesssim C_\eps\|\Xi\|_{L^2(\O)}^2+C_\eps\|v\|_{L^2(\O)}^2+\eps \|\nabla v\|_{L^2(\O)}^2, 
  \end{align*}
  where we recall that $\Q$ (see Subsection \ref{sub:prelim}) is a projection, thus $\|\Q\|= 1$.  
For the next term, we use integration by parts,  Young's inequality, the Cauchy--Schwarz inequality and Assumption \ref{it:well_posedness_primitive_kone_smoothness}: 
\begin{align*}
    \Big|\int_\O \hhp\big[\mathcal{J}_\kappa {\theta}(\cdot,x)\big]\cdot v \dd x\Big| 
    &=\Big|\int_\O \Big(\int_{-h}^{x_3} \kappa(\cdot,\zeta)\theta(\cdot,\zeta) \dd\zeta\Big) \big(\nabla_H \cdot v\big) \dd x \Big|\\
    &\leq \int_\O C_\eps\Big|\int_{-h}^{0} \kappa(\cdot,\zeta)\theta(\cdot,\zeta) \dd\zeta\Big|^2+\eps |\nabla v|^2  \dd x \\
    &\leq hC_\eps \int_{\mathbb{T}^2}\|\kappa(\cdot)\|_{L^2(-h,0)}^2\|\theta(\cdot)\|_{L^2(-h,0)}^2\dd x_H+\eps \|\nabla v\|_{L^2(\O)}^2 \\
    &\lesssim  C_\eps  \|\theta(\cdot)\|_{L^2(\O)}^2 +\eps \|\nabla v\|_{L^2(\O)}^2.
\end{align*}
By Assumption \ref{it:sublinearity_Gforce} and Young's inequality, we have
\begin{align*}
    \Big|\int_\O \hhp[F_v(\cdot,v,\theta,\nabla v)]\cdot v\dd x\Big|&\leq \int_\O |F_v(\cdot,v,\theta,\nabla v)|| v|\dd x\\
    &\lesssim \|\Xi\|_{L^2(\O)}^2+C_\eps\|v\|_{L^2(\O)}^2+\|\theta\|_{L^2(\O)}^2+\eps\|\nabla v\|_{L^2(\O)}^2.
\end{align*}
By Young's inequality and Remark \ref{rem:assumptionslocal}, we have  
\begin{align*}
    \Big|\int_\O \Big(\sum_{n\in\N}\varphi_n \hhp[(\phi_n\cdot\nabla)v ]\Big)\cdot v \dd x\Big| &\leq \int_\O \|\varphi\|_{\ell^2} \|((\phi_n\cdot\nabla)v )_n\|_{\ell^2}|v| \dd x\\
    &\leq  
    C_\eps \|\varphi\|_{\ell^2}^2 \|v\|_{L^2(\O)}^2+ \eps\|((\phi_n\cdot\nabla)v )_n\|_{L^2(\O;\ell^2)}^2\\
    &\lesssim C_\eps \|\varphi\|_{\ell^2}^2 \|v\|_{L^2(\O)}^2+ \eps\|\nabla v\|_{L^2(\O)}^2.
\end{align*}  
Similarly, using Assumption \ref{it:sublinearity_Gforce},
\begin{align*}
\Big|\int_\O\big(\sum_{n\in\N}\varphi_n \hhp[G_{v,n}(\cdot,v)]\big)\cdot v\dd x\Big|
&\leq \frac{1}{2} \|\varphi\|_{\ell^2}^2 \|v\|_{L^2(\O)}^2+C^2\big(\|\Xi\|_{L^2(\O)}^2+ \|v\|_{L^2(\O)}^2\big),
\end{align*} 
concluding the estimates for the terms of \eqref{eq:L^2 norm v}.

For the first term  of \eqref{eq:L^2 norm theta}, we have by \eqref{eq: est theta robin laplace}: 
\begin{align*}
 (\DeltawR\theta)(\theta)
&\leq -\|\nabla \theta\|_{{L}^2(\O)}^2+\eps\|\nabla \theta\|_{{L}^2(\O)}^2+C_\eps \|\theta\|_{{L}^2(\O)}^2
, 
\end{align*}
and for the second term  of \eqref{eq:L^2 norm theta}, we have similarly to that of \eqref{eq:L^2 norm v}:
\[
\int_\O [(u\cdot\nabla)\theta)] \theta\dd x=0.
\]

Mimicking the estimates with $F_v$, $\hhp[(\phi_n\cdot\nabla)v ]$ and $G_{v,n}$ above, and using Assumptions \ref{it:sublinearity_Gforce}, \ref{it:well_posedness_primitive_parabolicity} and Young's inequality, one finds 
\begin{align*}
&\Big|\int_\O F_{\theta}(\cdot,v,\theta,\nabla v) \theta\dd x\Big|  \lesssim \|\Xi\|_{L^2(\O)}^2+\|v\|_{L^2(\O)}^2+C_\eps\|\theta\|_{L^2(\O)}^2+\eps\|\nabla v\|_{L^2(\O)}^2,\\
&\Big|\int_\O \sum_{n\in\N}\varphi_n[(\psi_n\cdot\nabla)\theta] \theta \dd x\Big|\leq \frac{1}{2} \|\varphi\|_{\ell^2}^2 \|\theta\|_{L^2(\O)}^2+ \frac{1}{2}\nu\|\nabla \theta\|_{L^2(\O)}^2,\\
&\Big|\int_\O\sum_{n\in\N}\varphi_n(s)G_{\theta,n}(\cdot,v,\theta,\nabla v) \theta\dd x\Big|\lesssim C_\eps \|\varphi\|_{\ell^2}^2 \|\theta\|_{L^2(\O)}^2\\
&\hspace{5cm}+ \eps\big(\|\Xi\|_{L^2(\O)}^2+ \|v\|_{L^2(\O)}^2+\|\theta\|_{L^2(\O)}^2+\|\nabla v\|_{L^2(\O)}^2\big).
\end{align*} 

Combining all above estimates, recalling that $\nu\in(0,2)$ and fixing $\eps>0$ sufficiently small, we find that \eqref{eq:skeleton L^2 est prep} is satisfied with $R_t\ceqq 1+\|\Xi\|_{L^2(0,t;L^2(\O))}^2$, for some $\alpha>0$.  Consequently, \eqref{eq:skeleton L^2 est} holds.

\subsection{Intermediate estimates for skeleton barotropic and baroclinic modes}\label{sub:skeleton intermed}

Here, we use the $L^2$-estimate \eqref{eq:skeleton L^2 est} to  prove a powerful intermediate estimate, namely \eqref{eq:skeleton intermed est} below. We introduce a  vertical average and difference function: 
\begin{align}\label{eq: def bar tilde}
     \bar{g}(t,x_{\h})\ceqq \frac{1}{h}\int_{-h}^0 g(t,x_{\h},\zeta)\dd\zeta,\quad \tilde{g}(t,x)\ceqq g(t,x)-\bar{g}(t,x_{\h})
\end{align}
for $g\in L^2(\O;\R^2)$, $t\in\R_+$ and  $x=(x_H,x_3)\in\O$.

We will rewrite the $v$-component of the skeleton equation as a coupled PDE for the pair $(\bar{v},\tilde{v})$, which uniquely determines  $v$ and vice versa. The vertical average $\bar{v}$ represents the so-called \emph{barotropic mode}, and $\tilde{v}$ represents the \emph{baroclinic mode}. The strategy of splitting $v$ into these modes for the (deterministic) primitive equations originates from \cite{CT07}. 

Using the $L^2$-estimate \eqref{eq:skeleton L^2 est}, we will prove the following for $(\bar{v},\tilde{v},v)$, for all $t\in[0,T]$: 
\begin{align}\label{eq:skeleton intermed est}
    \begin{split}
    \|\bar{v}(t)\|_{H^1(\Tor^2)}^2&+\|\pz v(t)\|_{L^2(\O)}^2+ \|\tilde{v}(t)\|_{L^4(\O)}^4+\int_0^t\|\bar{v}(s)\|_{H^2(\Tor^2)}^2\dd s\\
    &\quad+\int_0^t \|\nabla\pz v(s)\|_{L^2(\O)}^2\dd s+\int_0^t\||\tilde{v}(s)||\nabla\tilde{v}(s)|\|_{L^2(\O)}^2\dd s \lesssim {C}_{2,v_0,\theta_0,T,\varphi}, 
    \end{split}
\end{align} 
where
\begin{equation}\label{eq: def exp exp constant}
\begin{split}
& {C}_{2,v_0,\theta_0,T,\varphi}=( C_{v_0}+\hat{C}_{2,v_0,\theta_0,T,\varphi})\exp(c_2\hat{C}_{2,v_0,\theta_0,T,\varphi}),\\
&\hat{C}_{2,v_0,\theta_0,T,\varphi}\ceqq T+ \|\Xi\|_{L^2(0,T;L^2(\O))}^2+  {C}_{1,v_0,\theta_0,T,\varphi}(T+1+{C}_{1,v_0,\theta_0,T,\varphi}) +\|\varphi \|_{L^2(0,T;\ell^2)}^2,
    \end{split}
\end{equation}
with constant ${C}_{1,v_0,\theta_0,T,\varphi}$ from \eqref{eq:skeleton L^2 est} and some constant $c_2>0$ independent of $\varphi$, $T$, $v$ and $\theta$.

We note that $\overline{\hhp v}=\hhp_\h \bar{v}$, $\pz  v=\pz  \tilde{v}$ and  $w(v)=w(\tilde{v})$ since $\nabla_{\h}\cdot \bar{v}=0$. Also, for any $f\in L^2(\O;\R^2)$, we have $\widetilde{\hhp f}=\tilde{f}$. Using these facts, similar as in \cite[\S 5.2.2]{AHHS22}, one can show that $(\bar{v},\tilde{v})$ solves  
\begin{align}\label{eq:primitive tilde and bar}
\begin{cases}
\bar{v}' -\Delta_{\h} \bar{v}=\Big(\pr\big[-(\bar{v}\cdot \nabla_{\h})\bar{v}- \force(\tilde{v}) + \bar{\fvt}+ \mathcal{P}_{\gamma,\phi}(\cdot,v)\big]\Big)\hspace{3cm}&\\
\qquad\qquad\qquad\qquad\qquad+\sum_{n\geq 1}\pr \big[(\phi_{n,\h}\cdot\nabla_{\h}) \bar{v}+\overline{\phi^3_n \pz  v }  +\overline{\gvtn}\big] \varphi_n &\text{ on }\Tor^2, \\
\bar{v}(\cdot,0)=\bar{v}_0& \text{ on }\Tor^2,\\
\tilde{v}' -\Delta \tilde{v}=\big[-(\tilde{v}\cdot \nabla_{\h})\tilde{v}+\forcetwo(\tilde{v},\bar{v}) + \tilde{\fvt} \,\big]  
+\sum_{n\geq 1}\big[(\phi_{n}\cdot\nabla) \tilde{v}-\overline{\phi^3_n \pz  v } +\tilde{\gvtn}\,\big] \varphi_n &\text{ on }\O,\\
\pz  \tilde{v}(\cdot,-h)=\pz  \tilde{v}(\cdot,0)=0&\text{ on }\ \Tor^2,\\
\tilde{v}(\cdot,0)=\tilde{v}_0\ceqq v_0-\bar{v}_0&\text{ on }\O,
\end{cases}
\end{align}
where $\phi_{n,\h}\ceqq (\phi^1_n,\phi^2_n)$ and 
\begin{align}\label{eq: abbrevs}
\begin{split}
&\forcetwo(\tilde{v},\bar{v})\ceqq- w(v) \pz  \tilde{v} -(\tilde{v}\cdot \nabla_{\h}) \bar{v} -( \bar{v}\cdot \nabla_{\h} )\tilde{v} +\force(\tilde{v}),\\
&\force(\tilde{v})\ceqq \textstyle{\frac{1}{h}\int_{-h}^0 \big[(\tilde{v}\cdot \nabla_{\h}) \tilde{v}+\tilde{v} (\nabla_{\h}\cdot \tilde{v})\big]\dd\zeta}, \\
&f\ceqq \mathcal{J}_\kappa {\theta}+F_v(\cdot,v,\theta,\nabla v)+\mathcal{P}_{\gamma,G} (t,v),\\
&g_n\ceqq G_{v,n}(\cdot,v).
\end{split}
\end{align} 

Before we turn to the proof of \eqref{eq:skeleton intermed est}, we discuss two estimates that will be applied many times. 

\begin{remark}\label{rem: removing bar or tilde}
For any  $p\in[1,\infty)$ and $s\geq 0$:
    \[
    \|\bar{g}\|_{H^{s,p}(\Tor^2)}, \|\tilde{g}\|_{H^{s,p}(\O)} \lesssim_h \|{g}\|_{H^{s,p}(\O)}.
    \]
Note that $\tilde{g}=g-\bar{g}$, so   by the triangle inequality, it suffices to prove the claim for $\bar{g}$. 
    
The  case $s=0$, i.e.\  the $L^p$-case,  follows (for example) from H\"older's inequality on $(-h,0)$: 
\begin{align*}
        \|\bar{g}\|_{L^p(\Tor^2)}^p=\int_{\Tor^2}\big| h^{-1}\int_{-h}^0  g\dd x_3\big|^p\dd x_H 
        \leq h^{-p}\int_{\Tor^2}h^{p-1} \|g\|_{L^p(-h,0)}^p\dd x_H= h^{-1}\|{g}\|_{L^p(\O)}^p.
    \end{align*}
    For the analogous $H^{k,p}$-estimates with $s=k\in\N$, simply observe from \eqref{eq: def bar tilde} that for $i\in\{1,2\}$, we have  $\partial_i\bar{g}=\overline{\partial_i g}$, so $\|\partial_i\bar{g}\|_{L^p(\Tor^2)}\lesssim \|\partial_i g\|_{L^p(\O)}$ by the $L^p$-case above.  Using this argument, the estimate for $k\in\N$ follows by induction. The case $s\in\R_+$ then follows by exactness of the complex interpolation method. 
\end{remark}

\begin{remark}\label{rem:g_n H1}
    For $G_{v,n}(\cdot,v)$, we have the following estimate. By the chain rule (with $\nabla=\nabla_x$), we have $ 
    |\nabla (G_{v,n}(\cdot,v))|\lesssim | \nabla  G_{v,n}(\cdot,v)|+|\partial_y G_{v,n}(\cdot,v)||\nabla  v|$ for all $n\in\N$. 
    Combining this with 
    Assumption \ref{it:sublinearity_Gforce} gives 
    $
    \|(G_{v,n}(\cdot,v))_n\|_{H^1(\O;\ell^2)}^2\lesssim \|\Xi\|_{L^2(\O)}^2+\|v\|_{L^2(\O)}^2+\|\nabla v\|_{L^2(\O)}^2.
    $ 
\end{remark} 

Let us start the proof of \eqref{eq:skeleton intermed est}.  We subsequently apply It\^o's formula for $\|\bar{v}\|_{H^1(\Tor^2)}^2$, $\|\pz  v\|_{L^2(\O)}^2$ and $\|\tilde{v}\|_{L^4(\O)}^4$. After that, \eqref{eq:skeleton intermed est} will be obtained by Gr\"onwall's inequality applied to the map 
$
t\mapsto  \|\bar{v}(t)\|_{H^1(\Tor^2)}^2+ \|\pz  v(t)\|_{L^2(\O)}^2+c\|\tilde{v}(t)\|_{L^4(\O)}^4$,  for a sufficiently large $c>0$, which is needed to compensate for $\||\tilde{v}(s)||\nabla\tilde{v}|\|_{L^2(\O)}^2$-terms that arise in the estimates for $\|\bar{v}\|_{H^1(\Tor^2)}^2$ and $\|\pz  v\|_{L^2(\O)}^2$. 

For $\bar{v}$, the It\^o formula applied on the Gelfand triple $({H}^2(\Tor^2;\R^2),{H}^1(\Tor^2;\R^2),{L}^2(\Tor^2;\R^2))$ gives 
\begin{align}
 \| \bar{v}(t)\|_{H^1(\Tor^2)}^2-&\|\bar{v}_0\|_{H^1(\Tor^2)}^2  \notag \\
 &= 2\int_0^t\<\Delta_{\h} \bar{v},\bar{v}\?\dd s  +
   2\int_0^t\Big\<\pr\big[-(\bar{v}\cdot \nabla_{\h})\bar{v}- \force(\tilde{v}) + \bar{\fvt}+ \mathcal{P}_{\gamma,\phi}(\cdot,v)\big],\bar{v}\Big\?\dd s\notag\\
  &\quad+
   2\int_0^t\Big\<\sum_{n\geq 1}\pr \big[(\phi_{n,\h}\cdot\nabla_{\h}) \bar{v}+\overline{\phi^3_n \pz  v }  +\overline{\gvtn}\big] \varphi_n,\bar{v}\Big\?_{H^1(\Tor^2)}\dd s\label{eq: Ito bar v}\\
   &\eqqc \bar{I}_{1}^{\scriptscriptstyle H^1}(t,v,\theta,\varphi).\label{eq: Ito bar v def}
\end{align}
Since the domain is the torus, we have
\begin{align*}
\<\Delta_{\h}\bar{v},\bar{v}\?=\<\Delta_{\h}\bar{v},\bar{v}\?_{L^2(\Tor^2)}-\<\Delta_{\h}\bar{v},\Delta_{\h}\bar{v}\?_{L^2(\Tor^2)} 
&=-\|\bar{v}\|_{H^2(\Tor^2)}^2+\|\bar{v}\|_{L^2(\Tor^2)}^2.
\end{align*} 
For the second term of \eqref{eq: Ito bar v},  Young's inequality gives
\begin{align*}
     \Big\<\pr\big[-(\bar{v}\cdot \nabla_{\h})\bar{v} &- \force(\tilde{v}) + \bar{\fvt}+ \mathcal{P}_{\gamma,\phi}(\cdot,v)\big],\bar{v}\Big\?  \\
    &\leq  C_\eps\big\|-(\bar{v}\cdot \nabla_{\h})\bar{v}- \force(\tilde{v}) + \bar{\fvt}+ \mathcal{P}_{\gamma,\phi}(\cdot,v)\big\|_{L^2(\Tor^2)}^2+\eps\|\bar{v}\|_{H^2(\Tor^2)}^2.
\end{align*}
We estimate the first term of the right-hand side in parts. 
Note that $[L^2(\Tor^2),H^1(\Tor^2)]_{1/2}=H^{1/2}(\Tor^2)\into L^{4}(\Tor^2)$. Using interpolation, Young's inequality and Remark \ref{rem: removing bar or tilde}, we thus have  
\begin{align*}
\|(\bar{v}\cdot \nabla_{\h})\bar{v}\|_{L^2(\Tor^2)}^2
    &\lesssim
    \|\bar{v}\|_{L^{4}(\Tor^2)}^2\|\nabla_{\h}\bar{v}\|_{L^{4}(\Tor^2)}^2 \\
    &\lesssim \|\bar{v}\|_{L^{2}(\Tor^2)}\|\bar{v}\|_{H^{1}(\Tor^2)}\|\nabla_{\h}\bar{v}\|_{L^2(\Tor^2)} \|\nabla_{\h}\bar{v}\|_{H^{1}(\Tor^2)}\\
    &\lesssim C_\eps (\|{v}\|_{L^{2}(\O)}^2\|\nabla v\|_{L^2(\O)}^2)\|\bar{v}\|_{H^{1}(\Tor^2)}^2+\eps\|\bar{v}\|_{H^{2}(\Tor^2)}^2. 
\end{align*}  
Turning to the next terms, we have 
\begin{equation}\label{eq:origin c}
  \|\force(\tilde{v})\|_{L^2(\Tor^2)}^2\lesssim \||\tilde{v}||\nabla\tilde{v}|\|_{L^2(\O)}^2, 
\end{equation}
where the constant comes from  Remark \ref{rem: removing bar or tilde} and depends only on the height $h$ of $\O$. 
Moreover, by Remark \ref{rem: removing bar or tilde} and the definition of $f$ (see \eqref{eq: abbrevs}), 
\begin{align*}
\|\bar{\fvt}\|_{L^2(\Tor^2)}^2&\lesssim\|{\fvt}\|_{L^2(\O)}^2\lesssim \|\mathcal{J}_\kappa {\theta}\|_{L^2(\O)}^2+\|F_v(\cdot,v,\theta,\nabla v)\|_{L^2(\O)}^2+\|\mathcal{P}_{\gamma,G} (t,v)\|_{L^2(\O)}^2.
\end{align*}
For the terms on the right-hand side, one can use Assumption  \ref{it:well_posedness_primitive_kone_smoothness},  H\"older's inequality and the embedding $H^1(\O)\into H^1(\Tor^2;L^2(-h,0))\into L^\rho(\Tor^2;L^2(-h,0))$ with $\frac{1}{2+\delta}+\frac{1}{\rho}=\frac{1}{2}$ to obtain
\begin{equation}\label{eq:Jkappa est}
    \|\mathcal{J}_\kappa {\theta}\|_{L^2(\O)}^2\lesssim \|\theta\|_{L^2(\O)}^2+ \|\nabla\theta\|_{L^2(\O)}^2,
\end{equation}
Assumption \ref{it:sublinearity_Gforce} to find that
\[
\|F_v(\cdot,v,\theta,\nabla v)\|_{L^2(\O)}^2\lesssim \|\Xi\|_{L^2(\O)}^2+\|\theta\|_{L^2(\O)}^2+ \|v\|_{L^2(\O)}^2+\|\nabla v\|_{L^2(\O)}^2,
\]
and H\"older's inequality, the embedding $H^1(\O)\into L^6(\O)$ and Remark \ref{rem:g_n H1} to obtain 
\[
\|\mathcal{P}_{\gamma,G} (t,v)\|_{L^2(\O)}^2\lesssim \|\Xi\|_{L^2(\O)}^2+\|v\|_{L^2(\O)}^2+\|\nabla v\|_{L^2(\O)}^2.
\] 
We  conclude that 
  \begin{align}
 \|\bar{\fvt}\|_{L^2(\Tor^2)}^2\lesssim \|f\|_{L^2(\O)}^2  
 &\lesssim \|\Xi\|_{L^2(\O)}^2+\|\theta\|_{L^2(\O)}^2+  \|v\|_{L^2(\O)}^2+ \| \nabla\theta\|_{L^2(\O)}^2 +\|\nabla v\|_{L^2(\O)}^2. \label{eq: f est L2} 
 \end{align} 
For $\mathcal{P}_{\gamma,\phi} (\cdot,v)$, we  note that $\Q f=\Q \bar{f}$, and use Assumption  \ref{it:independence_z_variable}, Remark \ref{rem:assumptionslocal}, the Sobolev embedding $[L^2(\Tor^2),H^1(\Tor^2)]_{2/3}=H^{2/3}(\Tor^2)\into L^6(\Tor^2)$, interpolation and Young's inequality to find
\begin{align*}
\|\mathcal{P}_{\gamma,\phi} (\cdot,v)\|_{L^2(\O)}^2
&\lesssim \max_{1\leq l,m\leq 2}\big\| \|(\gamma_n^{l,m})_n\|_{\ell^2}\|(\Q[\overline{(\phi_n\cdot \nabla)v}])_n\|_{\ell^2} \big\|_{L^2(\O)}^2\\
&\leq\max_{1\leq l,m\leq 2}\|(\gamma_n^{l,m})_n\|_{L^3(\O;\ell^2)}^2\|(\overline{(\phi_n\cdot \nabla)v})_n\|_{L^6(\Tor^2;\ell^2)}^2\\
&\lesssim \|((\phi_{n,\h}\cdot \nabla)\bar{v})_n\|_{L^6(\Tor^2;\ell^2)}^2+\|(\overline{\phi_{n}^3\pz v})_n\|_{L^6(\Tor^2;\ell^2)}^2\\
&\lesssim \|\nabla_{\h}\bar{v}\|_{L^6(\Tor^2))}^2+\| \overline{\pz v}  \|_{L^6(\Tor^2)}^2\\
&\lesssim \eps\|\bar{v}\|_{H^2(\Tor^2)}^2+ {C}_{\eps}\|\bar{v}\|_{H^1(\Tor^2)}^2+\eps \|\overline{\pz v}\|_{H^1(\Tor^2)}^2+C_\eps \|\overline{\pz v}\|_{L^2(\Tor^2)}^2\\ 
&\lesssim \eps\|\bar{v}\|_{H^2(\Tor^2)}^2+ {C}_{\eps}\|\bar{v}\|_{H^1(\Tor^2)}^2+\eps \|\nabla{\pz v}\|_{L^2(\O)}^2+C_\eps \|{\pz v}\|_{L^2(\O)}^2.
\end{align*}

The last term of \eqref{eq: Ito bar v} can be treated with  integration by parts and similar estimates as before, and with  Remark   \ref{rem:g_n H1}:
\begin{align*}
      &\Big\<\sum_{n\geq 1}\pr \big[(\phi_{n,\h}\cdot\nabla_{\h}) \bar{v}+\overline{\phi^3_n \pz  v }  +\overline{\gvtn}\big] \varphi_n,\bar{v}\Big\?_{H^1(\Tor^2)} \\
     &\lesssim \|(\overline{\gvtn})\|_{H^1(\Tor^2;\ell^2)}\|\varphi\|_{\ell^2} \|\bar{v}\|_{H^1(\Tor^2)}+ \Big\|\sum_{n\geq 1}\pr \big[(\phi_{n,\h}\cdot\nabla_{\h}) \bar{v}+\overline{\phi^3_n \pz  v }  \big] \varphi_n\Big\|_{L^2(\Tor^2)}^2+\|\bar{v}\|_{L^2(\Tor^2)}^2\\
     &\qquad\qquad+\Big\<\nabla_H\sum_{n\geq 1}\pr \big[(\phi_{n,\h}\cdot\nabla_{\h}) \bar{v}+\overline{\phi^3_n \pz  v }   \big] \varphi_n,\nabla_H\bar{v}\Big\?_{L^2(\Tor^2)}\\
     &\leq \|({\gvtn})\|_{H^1(\O;\ell^2)}^2+\|\varphi\|_{\ell^2}^2 \|\bar{v}\|_{H^1(\Tor^2)}^2+  \|\bar{v}\|_{L^2(\Tor^2)}^2\\
     &\qquad\qquad +C_\eps\Big\|\sum_{n\geq 1}\pr \big[(\phi_{n,\h}\cdot\nabla_{\h}) \bar{v}+\overline{\phi^3_n \pz  v }  +\overline{\gvtn}\big] \varphi_n\Big\|_{L^2(\Tor^2)}^2+\eps\|\bar{v}\|_{H^2(\Tor^2)}^2\\
     &\lesssim \|\Xi\|_{L^2(\O )}^2+\|v\|_{L^2(\O)}^2+\|\nabla v\|_{L^2(\O)}^2+\|\varphi\|_{\ell^2}^2 \|\bar{v}\|_{H^1(\Tor^2)}^2 +\|\bar{v}\|_{L^2(\Tor^2)}^2 \\
     &\qquad\qquad +C_\eps\|\varphi\|_{\ell^2}^2\Big(\|((\phi_{n,\h}\cdot\nabla_{\h}) \bar{v})_n\|_{L^2(\Tor^2;\ell^2)}^2+\|(\overline{\phi^3_n \pz  v })_n\|_{L^2(\Tor^2;\ell^2)}^2    \Big)+\eps\|\bar{v}\|_{H^2(\Tor^2)}^2\\
     &\lesssim 
      \|\Xi\|_{L^2(\O)}^2+\|v\|_{L^2(\O)}^2+\|\nabla v\|_{L^2(\O)}^2+\|\varphi\|_{\ell^2}^2 \|\bar{v}\|_{H^1(\Tor^2)}^2 + \|\bar{v}\|_{L^2(\Tor^2)}^2\\
     &\qquad\qquad+
     C_\eps\|\varphi\|_{\ell^2}^2\Big(\|\bar{v}\|_{H^1(\Tor^2)}^2+\| \pz v\|_{L^2(\O)}^2  \Big)+\eps\|\bar{v}\|_{H^2(\Tor^2)}^2. 
\end{align*}

Next, let us turn to $\pz v$. Applying $\pz \int_0^t(\cdot)\dd s$ to the equation for $v$ in \eqref{eq:primitive_skeleton}, applying the chain rule for $\pz v$ with Gelfand triple $({H}^1(\O;\R^2),{L}^2(\O;\R^2),{H}^{-1}(\O;\R^2))$, recalling \eqref{eq: abbrevs}  and integrating by parts, we obtain
\begin{align}
   \frac{1}{2}&\|\pz v(t)\|_{L^2(\O)}^2 -\frac{1}{2} \|\pz v_0\|_{L^2(\O)}^2\notag\\
   &= \int_0^t \<\pz \Delta v,\pz v\?\dd s\,+\int_0^t \Big\<\pz \hhp\big[ -(v\cdot \nabla_{\h}) v- w(v)\pz  v+f+\mathcal{P}_{\gamma,\phi} (s,v)  \big],\pz v\Big\?\dd s\notag\\
   &\quad+\int_0^t \Big\<\pz \sum_{n\geq 1}\hhp \big[(\phi_{n}\cdot\nabla) v  +g_n \big]\varphi_n ,\pz v\Big\?\dd s \notag\\
   &=-\int_0^t \int_\O|\nabla \pz v|^2\dd x\dd s
  -\int_0^t\int_\O \pz[(u\cdot \nabla) v] \cdot\pz   v\dd x\dd s-\int_0^t\int_\O \hhp[f]\cdot\pzz v\dd x\dd s\notag\\
   &\quad+\int_0^t\int_\O \sum_{n\geq 1} [ (\phi_{n,\h}\cdot\nabla_{\h}) \pz v+\pz(\phi_n^3\pz v)+ \pz g_n]\varphi_n \cdot\pz v\dd x\dd s \notag\\
   &\eqqc I_{1}^{\scriptscriptstyle \pz}(t,v,\theta,\varphi). \label{eq: Ito pz v}
\end{align}
Here, for the terms with $(u\cdot \nabla) v$, $\mathcal{P}_{\gamma,\phi}$, $(\phi_{n}\cdot\nabla) v$ and $g_n$, we used that $\pz \hhp g= \pz g$ in the distributional sense, whenever $g\in  L^2(\O;\R^2)$ and  $\pz g$ exists in the distributional sense, as follows instantly from  \eqref{eq:def hhp}. We also used Assumption \ref{it:independence_z_variable} to compute $\pz((\phi_{n,\h}\cdot\nabla_{\h}) v)$, and we used that $\mathcal{P}_{\gamma,\phi} =\overline{\mathcal{P}_{\gamma,\phi}}$ by \eqref{eq:def Pgamma} and the fact that $\Q[f]=\Q_{\h}[\bar{f}]$. In particular,  $\pz\hhp[\mathcal{P}_{\gamma,\phi}]=\pz\mathcal{P}_{\gamma,\phi} =\pz\overline{\mathcal{P}_{\gamma,\phi}}=0$. 

Recall that $\pz w(v)=-\nabla_{\h}\cdot v$,  so we can rewrite 
\begin{align*}
\int_\O \pz [ (u\cdot \nabla) v] \cdot\pz   v\dd x&= \int_\O  [(\pz u\cdot \nabla) v]\cdot\pz v\dd x+\int_\O  [(u\cdot \nabla) \pz v] \cdot\pz   v\dd x\\
&=  \int_\O  [(\pz u\cdot \nabla) v]\cdot\pz v\dd x\\
&= \int_\O [(\pz v\cdot \nabla_{\h}) v]\cdot\pz v\dd x-\int_\O [( \nabla_{\h}\cdot v)\pz  v]\cdot\pz v\dd x,
\end{align*} 
using the divergence-free condition and boundary conditions in the second line. The last line can now be estimated using  \cite[p.\ 24, 25]{HH20} (see $I_5$ and $I_6$ therein), 
leading to 
\begin{align}\label{eq:origin c 2}
\big|\int_\O \pz [ (u\cdot \nabla) v] \cdot\pz   v\dd x\big|&\lesssim \eps \|\nabla\pz v\|_{L^2(\O)}^2+C_\eps\|\nabla v\|_{L^2(\O)}^2\|\pz v\|_{L^2(\O)}^2+C_\eps \||\tilde{v}||\nabla \tilde{v}|\|_{L^2(\O)}^2,
\end{align} 
where we also used Remark \ref{rem: removing bar or tilde}. 
Next, by \eqref{eq: f est L2}  we have
\begin{align*}
\int_\O &\hhp[f]\cdot\pzz v\dd x \leq C_\eps \|f\|_{L^2(\O)}^2+\eps\|\nabla\partial_3v\|_{L^2(\O)}^2 \\
&\lesssim C_\eps\big(\|\Xi\|_{L^2(\O)}^2+\|\theta\|_{L^2(\O)}+  \|v\|_{L^2(\O)}^2+ \| \nabla\theta\|_{L^2(\O)}^2 +\|\nabla v\|_{L^2(\O)}^2 \big)+\eps\|\nabla\partial_3v\|_{L^2(\O)}^2.
\end{align*}
By Remarks \ref{rem:assumptionslocal} and   \ref{rem:g_n H1}, 
\begin{align*}
   \int_\O \sum_{n\geq 1} &[(\phi_{n,\h}\cdot\nabla_{\h}) \pz v+\pz g_n]\varphi_n \cdot\pz v\dd x \leq   \big(\|\nabla \partial_3 v\|_{L^2(\O)} +\|\nabla g_n\|_{L^2(\O;\ell^2)}\big)\|\varphi\|_{\ell^2}\|\partial_3 v\|_{L^2(\O)}\\
   &\qquad\lesssim \eps\|\nabla\partial_3 v\|_{L^2(\O)}^2+\|\Xi\|_{L^2(\O)}^2+  \|v\|_{L^2(\O)}^2+\|\nabla v\|_{L^2(\O)}^2+C_\eps \|\varphi\|_{\ell^2}^2 \|\partial_3 v\|_{L^2(\O)}^2. 
\end{align*} 
Moreover, we can estimate  $\phi^3_n \pz  v$ in $H^1(\O;\ell^2)$ as follows. Let  $\rho\in(0,6)$ satisfy $\frac{1}{2}=\frac{1}{3+\delta}+\frac{1}{\rho}$. Then,  using integration by parts, Assumption \ref{it:well_posedness_primitive_phi_smoothness} and Remark \ref{rem:assumptionslocal}, we find that
\begin{align}
    \|(\phi^3_n \pz  v)_n\|_{H^1(\O;\ell^2)}^2&\lesssim \| (\phi_n^3 \pz v)_n\|_{L^2(\O;\ell^2)}^2+\| (\phi_n^3\nabla\pz v)_n\|_{L^2(\O;\ell^2)}^2 +\sum_{j=1}^3\|\big((\partial_j\phi_n^3)\pz v\big)_n\|_{L^2(\O;\ell^2)}^2 \notag\\
    &\lesssim  \| \pz v\|_{L^2(\O)}^2+\| \nabla\pz v\|_{L^2(\O)}^2 +\sum_{j=1}^3\|(\partial_j\phi_n^3)_n\|_{L^{3+\delta}(\O;\ell^2)}^2\|\pz v\|_{L^\rho(\O)}^2\notag \\
    &\lesssim  \| \pz v\|_{L^2(\O)}^2+\| \nabla\pz v\|_{L^2(\O)}^2,\label{eq:pzz est}
\end{align} 
where we used the  Sobolev embedding $H^1(\O)\into L^{\rho}(\O)$. In particular,
\begin{align*}
   \int_\O \sum_{n\geq 1} \pz(\phi_n^3\pz v) \varphi_n \cdot\pz v\dd x &\leq   \|\big(\nabla (\phi^3_n\partial_3 v)\big)_n\|_{L^2(\O;\ell^2)}  \|\varphi\|_{\ell^2}\|\partial_3 v\|_{L^2(\O)}\\
   &\lesssim \| \pz v\|_{L^2(\O)}^2+ \eps\|\nabla\partial_3 v\|_{L^2(\O)}^2+ C_\eps \|\varphi\|_{\ell^2}^2 \|\partial_3 v\|_{L^2(\O)}^2. 
\end{align*} 

For $\|\tilde{v}(t)\|_{L^4(\O)}^4$, we have by the chain rule/It\^o formula (similar to \cite[Lem.\ A.6]{AV25survey}, \cite[\S 3]{dareiotis}): 
\begin{align}
\frac{1}{4}\|\tilde{v}(t)\|_{L^4(\O)}^4-\frac{1}{4}\|\tilde{v}_0\|_{L^4(\O)}^4&=\int_0^t\int_\O\Delta \tilde{v}\cdot \tilde{v}|\tilde{v}|^2\dd x\dd s\notag\\
&\quad+\int_0^t\int_\O \Big[-(\tilde{v}\cdot \nabla_{\h})\tilde{v}+\forcetwo(\tilde{v},\bar{v}) + \tilde{\fvt} \,\Big]\cdot \tilde{v}|\tilde{v}|^2  \dd x\dd s\notag\\
&\qquad+\int_0^t\int_\O\sum_{n\geq 1}\Big[(\phi_{n}\cdot\nabla) \tilde{v}-\overline{\phi^3_n \pz  v } +\tilde{\gvtn}\,\Big] \varphi_n\cdot \tilde{v}|\tilde{v}|^2 \dd x\dd s\notag\\
&\eqqc \tilde{I}_{1}^{\scriptscriptstyle L^4}(t,v,\theta,\varphi). \label{eq: Ito tilde v L4}
\end{align}
Now, integrating by parts, 
\begin{align*}
\int_\O\Delta \tilde{v}\cdot \tilde{v}|\tilde{v}|^2\dd x&=-\sum_{i=1}^2\Big(\int_\O\nabla \tilde{v}_i\cdot \nabla\tilde{v}_i|\tilde{v}|^2\dd x+\int_\O\nabla \tilde{v}_i\cdot \tilde{v}_i\nabla|\tilde{v}|^2\dd x\Big)\\
&=-\||\tilde{v}||\nabla \tilde{v}|\|_{L^2(\O)}^2-\frac{1}{2}\int_\O \nabla|\tilde{v}|^2\cdot \nabla|\tilde{v}|^2\dd x \leq -\||\tilde{v}||\nabla \tilde{v}|\|_{L^2(\O)}^2.
\end{align*}
Moreover, recalling \eqref{eq: abbrevs} and using that $w(v)=w(\tilde{v})$, we find  
\begin{align*}
\int_\O \big[-(\tilde{v}\cdot \nabla_{\h})\tilde{v}+\forcetwo(\tilde{v},\bar{v})\big]\cdot \tilde{v}|\tilde{v}|^2  \dd x&= \int_\O \big[-(u\cdot \nabla)\tilde{v}+(\bar{v}\cdot \nabla_{\h})\tilde{v}\big]\cdot \tilde{v}|\tilde{v}|^2  \dd x\\
&\quad +\int_\O \big[-(\tilde{v}\cdot \nabla_{\h}) \bar{v} -( \bar{v}\cdot \nabla_{\h} )\tilde{v} +\force(\tilde{v})\big]\cdot \tilde{v}|\tilde{v}|^2  \dd x\\
&=  \int_\O \big[-(\tilde{v}\cdot \nabla_{\h}) \bar{v} +\force(\tilde{v})\big]\cdot \tilde{v}|\tilde{v}|^2  \dd x\\
&\leq C_\eps \|\nabla v\|_{L^2(\O)}^2\|\tilde{v}\|_{L^4(\O)}^4+\eps \||\tilde{v}||\nabla \tilde{v}|\|_{L^2(\O)}^2,
\end{align*}
where we used the divergence-free condition and boundary conditions for the second equality, and the last line follows from \cite[p.\ 25, 26]{HH20} (see $I_7$ and $I_8$ therein) and by Remark \ref{rem: removing bar or tilde}.  
  Next, since $H^1(\O)\into L^4(\O)$ and by Remark \ref{rem: removing bar or tilde}, H\"older's inequality, Young's inequality and \eqref{eq: f est L2},
\begin{align*}
\int_\O  \tilde{\fvt}  \cdot \tilde{v}|\tilde{v}|^2  \dd x  &\lesssim \|f\|_{L^2(\O)} \|\tilde{v}\|_{L^4(\O)}\||\tilde{v}|^2\|_{L^4(\O)} \\
&\lesssim C_\eps\|f\|_{L^2(\O)}^2 \|\tilde{v}\|_{L^4(\O)}^2+\eps(\||\tilde{v}|^2\|_{L^2(\O)}^2+\|\nabla|\tilde{v}|^2\|_{L^2(\O)}^2) \\
&\lesssim C_\eps(\|\Xi\|_{L^2(\O)}^2+\|\theta\|_{L^2(\O)}+  \|v\|_{L^2(\O)}^2+ \| \nabla\theta\|_{L^2(\O)}^2 +\|\nabla v\|_{L^2(\O)}^2)(1+ \|\tilde{v}\|_{L^4(\O)}^4) \\
&\qquad +\eps \|\tilde{v}\|_{L^4(\O)}^4+\eps\||\tilde{v}||\nabla \tilde{v}|\|_{L^2(\O)}^2. 
\end{align*}
We have, using Remark \ref{rem:assumptionslocal} and Young's inequality,
\begin{align*}
\int_\O\sum_{n\geq 1}[(\phi_{n}\cdot\nabla) \tilde{v}] \varphi_n\cdot \tilde{v}|\tilde{v}|^2 \dd x \lesssim \int_\O|\nabla \tilde{v}|\|\varphi\|_{\ell^2} |\tilde{v}|^3 \dd x &\leq \eps\||\tilde{v}||\nabla\tilde{v}|\|_{L^2(\O)}^2+C_\eps\|\varphi\|_{\ell^2}^2\|\tilde{v}\|_{L^4(\O)}^4. 
\end{align*}
Furthermore, by H\"older's inequality,  Minkowski's inequality and Remark \ref{rem:assumptionslocal},   
\begin{align*}
\big|\int_\O\sum_{n\geq 1} -\overline{\phi^3_n \pz  v }  \varphi_n\cdot \tilde{v}|\tilde{v}|^2 \dd x\big|
&\leq \|\varphi\|_{\ell^2}
\int_{\Tor^2}\|(\overline{\phi^3_n \pz  v })_n\|_{\ell^2}\big(\int_{-h}^0|\tilde{v}|^3 \dd x_3\big)\dd x_{\h}\\
&\leq \|\varphi\|_{\ell^2}
\|(\overline{\phi^3_n \pz  v })_n\|_{L^2(\Tor^2;\ell^2)}\Big\|\int_{-h}^0|\tilde{v}|^3 \dd x_3\Big\|_{L^2(\Tor^2)}\\
&\lesssim \|\varphi\|_{\ell^2}
\|({\phi^3_n \pz  v })_n\|_{L^2(\O;\ell^2)}\int_{-h}^0\||\tilde{v}|^3\|_{L^2(\Tor^2)} \dd x_3\\ 
&\lesssim \|\varphi\|_{\ell^2}
\|\pz  v\|_{L^2(\O)}\int_{-h}^0\||\tilde{v}|^2\|_{L^3(\Tor^2)}^{3/2} \dd x_3. 
\end{align*}
We estimate the last term further  using that $[L^2(\Tor^2),H^1(\Tor^2)]_{1/3}=H^{1/3}(\Tor^2)\into L^3(\Tor^2)$ and using H\"older's inequality for the variable $x_3$: 
\begin{align*}
\int_{-h}^0\||\tilde{v}|^2\|_{L^3(\Tor^2)}^{3/2} \dd x_3
&\lesssim \int_{-h}^0\||\tilde{v}|^2\|_{L^2(\Tor^2)}\big(\||\tilde{v}|^2\|_{L^2(\Tor^2)}^{1/2}+\|\nabla_{\h}|\tilde{v}|^2\|_{L^2(\Tor^2)}^{1/2}\big) \dd x_3\\
&\lesssim  \||\tilde{v}|^2\|_{L^2(\O)} \Big(\int_{-h}^0 \||\tilde{v}|^2\|_{L^2(\Tor^2)}+\|\nabla_{\h}|\tilde{v}|^2\|_{L^2(\Tor^2)} \dd x_3\Big)^{1/2}\\
&\lesssim \|\tilde{v}\|_{L^4(\O)}^2 \Big(  \|\tilde{v}\|_{L^4(\O)}+\|\nabla|\tilde{v}|^2\|_{L^2(\O;\R^3)}^{1/2}\Big).
\end{align*}
Combining and applying Young's inequality gives
\begin{align*}
    \big|\int_\O\sum_{n\geq 1} -\overline{\phi^3_n \pz  v }  \varphi_n\cdot &\tilde{v}|\tilde{v}|^2 \dd x\big|\lesssim 
C_\eps\|\varphi\|_{\ell^2}^2
\|\tilde{v}\|_{L^4(\O)}^4+\eps \|\pz  v\|_{L^2(\O)}^2 \Big(\|\tilde{v}\|_{L^4(\O)}^2+\|\nabla|\tilde{v}|^2\|_{L^2(\O;\R^3)}\Big)\\
&\lesssim
C_\eps\|\varphi\|_{\ell^2}^2
\|\tilde{v}\|_{L^4(\O)}^4+\eps \|\nabla v\|_{L^2(\O)}^2\|\pz  v\|_{L^2(\O)}^2  +\eps\|\tilde{v}\|_{L^4(\O)}^4+\eps\||\tilde{v}||\nabla\tilde{v}|\|_{L^2(\O)}^2.
\end{align*}
Lastly, by Remark \ref{rem: removing bar or tilde},  Remark \ref{rem:g_n H1} and since $H^1(\O)\into L^4(\O)$, 
\begin{align*}
\big|\int_\O\sum_{n\geq 1} \tilde{\gvtn} \varphi_n\cdot \tilde{v}|\tilde{v}|^2 \dd x\big|&\lesssim \|  \varphi\|_{\ell^2}\|( {g_n})_n\|_{L^4(\O;\ell^2)}\|\tilde{v}\|_{L^4(\O)}\||\tilde{v}|^2\|_{L^2(\O)}\\
&\lesssim  \|({g_n})_n\|_{H^1(\O;\ell^2)}^2\|\tilde{v}\|_{L^4(\O)}^2+ \|\varphi\|_{\ell^2}^2\|\tilde{v}\|_{L^4(\O)}^4  \\
&\lesssim   (\|\Xi\|_{L^2(\O)}^2 +  \|v\|_{L^2(\O)}^2+ \|\nabla v\|_{L^2(\O)}^2+\|\varphi\|_{\ell^2}^2)(1+ \|\tilde{v}\|_{L^4(\O)}^4).
\end{align*} 

For Gr\"onwall's inequality, define 
$$
H_c(t)\ceqq \|\bar{v}(t)\|_{H^1(\Tor^2)}^2+\|\pz v(t)\|_{L^2(\O)}^2+c\|\tilde{v}(t)\|_{L^4(\O)}^4. 
$$ 
Adding the above estimates and fixing first $\eps>0$ small enough, and then $c>0$ large enough, we conclude that  
\begin{align}\label{eq: Gronwall prep intermed}
\begin{split}
H_c (t)&\lesssim_c C_{v_0}-\int_0^t\|\bar{v}\|_{H^2(\Tor^2)}^2\dd s -\int_0^t \|\nabla\pz v\|_{L^2(\O)}^2\dd s-\int_0^t\||\tilde{v}||\nabla\tilde{v}|\|_{L^2(\O)}^2\dd s\\
&\quad+\int_0^t (1+H_c)\bigg[ 1+\|\Xi\|_{L^2(\O)}^2+\|\theta\|_{L^2(\O)}+\|\nabla \theta\|_{L^2(\O)}^2+  \|v\|_{L^2(\O)}^2+\|\nabla v\|_{L^2(\O)}^2\\
&\qquad\qquad\qquad\qquad\qquad\qquad\qquad\qquad\qquad\qquad\qquad+ \|v\|_{L^2(\O)}^2 \|\nabla v\|_{L^2(\O)}^2+\|\varphi \|_{\ell^2}^2\bigg]\dd s.
\end{split}
\end{align}
Note that $\eps>0$ can be chosen independently of $\varphi$, $T$, $v_0$, $\theta_0$, $v$ and $\theta$ (it only depends on norms of embeddings for spaces of functions on $\O$ or $\Tor^2$). 
Moreover, the constant $c$ is merely used to compensate for \eqref{eq:origin c} and \eqref{eq:origin c 2}. The constant in \eqref{eq:origin c} originates from Remark \ref{rem: removing bar or tilde} and depends only on $h$, and $C_\eps$ in \eqref{eq:origin c 2} is the constant from Young's inequality with $\eps$. Thus $c>0$ can also be chosen independently of $\varphi$, $T$, $v_0$, $\theta_0$, $v$ and $\theta$. 

Gr\"onwall's inequality and  an application of  \eqref{eq:skeleton L^2 est} now yield \eqref{eq:skeleton intermed est}, with constant $\hat{C}_{2,v_0,\theta_0,T,\varphi}$ defined by \eqref{eq: def exp exp constant} as we claimed.

\subsection{$H^1$-estimate for the horizontal skeleton velocity}\label{sub:skeleton H1} 
We will now derive the claimed estimate of Proposition \ref{prop:skeleton}  from the intermediate estimate \eqref{eq:skeleton intermed est} and the $L^2$-estimate \eqref{eq:skeleton L^2 est}. 

Note that \eqref{eq:skeleton L^2 est} already yielded the needed estimates for $\theta$. Thus it remains to consider $v$ in the analytically strong setting: we now use the Gelfand triple $(\mathbb{H}_{\n}^2(\O),\mathbb{H}^1(\O),\mathbb{L}^2(\O))$, for which the duality is defined through $\mathbb{H}^1$. In particular, we have by \eqref{eq:Delta v,v}: 
\begin{equation*}
\<\Delta v,v\?=-\|v\|_{H^2}^2+\|v\|_{L^2}^2. 
\end{equation*}
We will work towards a Gr\"onwall argument for $t\mapsto \sup_{r\in[0,t]} \|v(r)\|_{H^1(\O)}^2+\|v\|_{L^2(0,t;H^2(\O;\R^2))}^2$. 
Applying the chain rule  \cite[Lem.\ 2.2 p.\ 30]{pardoux} on   $(\mathbb{H}^2(\O),\mathbb{H}^1(\O),\mathbb{L}^2(\O))$ and using  the Cauchy--Schwarz inequality and Young's inequality yields for any $\sigma>0$ and $0\leq r\leq t\leq T$: 
\begin{align}
   \frac{1}{2} \|v(r)&\|_{H^1(\O)}^2 -\frac{1}{2}\|{v}_0\|_{H^1(\O)}^2 = \int_0^r\<\Delta v,v\?\dd s \notag\\
    &\quad+
   \int_0^r\Big\<\hhp\big[ -(v\cdot \nabla_{\h}) v- w(v)\pz  v+\mathcal{P}_{\gamma} (\cdot,v) +\mathcal{J}_\kappa {\theta}+ \fv (\cdot,v,\theta,\nabla v)\big], v\Big\?\dd s\notag\\
   &\quad+\int_0^r\Big\<\sum_{n\geq 1} \hhp\big[(\phi_{n}\cdot\nabla) v  +\gvn(\cdot,v)\big]\varphi_n,v\Big\?_{H^1(\O)}\dd s\notag\\
   &\eqqc I_1^{H^1}(r,v,\theta,\varphi)\label{eq: Ito v H1 def}\\ 
&\lesssim
    -\|v\|_{L^2(0,r;H^2(\O))}^2+\|v\|_{L^2(0,T;L^2(\O))}^2 \notag\\
    &\quad+\int_0^tC_\sigma\Big(\big\|(v\cdot \nabla_{\h}) v\|_{L^2(\O)}^2+\|w(v)\pz  v\|_{L^2(\O)}^2+\|\mathcal{P}_{\gamma,\phi} (\cdot,v)\|_{L^2(\O)}^2 \notag\\
    &\quad+\|\mathcal{P}_{\gamma,G} (\cdot,v)\|_{L^2(\O)}^2+\|\mathcal{J}_\kappa {\theta}\|_{L^2(\O)}^2+ \|\fv (\cdot,v,\theta,\nabla v)\|_{L^2(\O)}^2\Big)+\sigma\|v\|_{H^2(\O)}^2\dd s\notag\\ 
&\quad+\int_0^t \|\big(\hhp[(\phi_{n}\cdot\nabla) v   +\gvn(\cdot,v)]\big)_n\|_{H^1(\O;\ell^2)}\|\varphi\|_{\ell^2}\|v\|_{H^1(\O)}\dd s. \label{eq:skeleton H1}  
\end{align}
We estimate the nonnegative terms from the second line of  \eqref{eq:skeleton H1} on, in order of appearance. 
Recall that $v=\tilde{v}+\bar{v}$, so
$(v\cdot \nabla_{\h}) v=(\tilde{v}\cdot \nabla_{\h}) \tilde{v}+(\tilde{v}\cdot \nabla_{\h}) \bar{v}+(\bar{v}\cdot \nabla_{\h}) \tilde{v}+(\bar{v}\cdot \nabla_{\h}) \bar{v}$. We have
\begin{align*}
    \int_0^t\|(\tilde{v}\cdot \nabla_{\h}) \tilde{v}\|_{L^2(\O)}^2\dd s
    \lesssim\int_0^t\||\tilde{v}||\nabla  \tilde{v}|\|_{L^2(\O)}^2 \dd s. 
\end{align*} 
Due to H\"older's inequality and since $H^2(\Tor^2)\into H^{1,4}(\Tor^2)$,
\begin{align*}
    \int_0^t\|(\tilde{v}\cdot \nabla_{\h}) \bar{v}\|_{L^2(\O)}^2\dd s
    \lesssim 
    \int_0^t\|\tilde{v}\|_{L^4(\O)}^2\|\bar{v}\|_{H^{1,4}(\Tor^2)}^2\dd s 
    &\lesssim (1+ \sup_{r\in[0,T]}\|\tilde{v}(r)\|_{L^4(\O)}^4)\int_0^T\|\bar{v}\|_{H^2(\Tor^2)}^2\dd s. 
\end{align*} 
Due to  H\"older's inequality,  the Sobolev embeddings  $[H^{1}(\O),H^{2}(\O)]_{\frac{1}{2}}=H^{\frac32}(\O)\into H^{1,3}(\O)$  and $H^1(\O)\into L^6(\O)$, interpolation,  Young's inequality and Remark \ref{rem: removing bar or tilde}:
\begin{align*}
    \int_0^t\|(\bar{v}\cdot \nabla_{\h}) \tilde{v}\|_{L^2(\O)}^2\dd s 
    & \lesssim 
    \int_0^t\|\bar{v}\|_{L^6(\Tor^2)}^2\|\tilde{v}\|_{H^{1,3}(\O)}^2\dd s\\
    &\lesssim \int_0^t\|\bar{v}\|_{H^1(\Tor^2)}^2\|\tilde{v}\|_{H^{1}(\O)}\|\tilde{v}\|_{H^{2}(\O)}\dd s\\
     &  \lesssim C_\eps  \int_0^t  \|\bar{v}(s)\|_{H^1(\Tor^2)}^4\sup_{r\in[0,s]}\|{v}(r)\|_{H^1(\O)}^2\dd s+\eps\|{v}\|_{L^2(0,t;H^2(\O))}^2. 
\end{align*} 
Due to H\"older's inequality and two Sobolev embeddings,
\begin{align*}
\int_0^t\|(\bar{v}\cdot \nabla_{\h}) \bar{v}\|_{L^2(\Tor^2)}^2\dd s
\lesssim \int_0^t \|\bar{v}\|_{L^4(\Tor^2)}^2\|\bar{v}\|_{H^{1,4}(\Tor^2)}^2\dd s  
&\lesssim \sup_{r\in[0,T]}\|\bar{v}(r)\|_{H^1(\Tor^2)}^2\int_0^T\|\bar{v}\|_{H^{2}(\Tor^2)}^2\dd s. 
\end{align*} 
For the next term, first observe that \eqref{eq:def_w_v} and  Minkowski's integral inequality give 
\begin{align*}
\|w(v)\|_{L^\infty(-h,0;L^4(\Tor^2))}\leq \big\|\int_{-h}^\cdot \|\nabla\cdot v\|_{L^4(\Tor^2)} \dd \zeta\big\|_{L^\infty(-h,0)}&=\|\nabla\cdot v\|_{L^1(-h,0;L^4(\Tor^2))}\\
&\lesssim\|\nabla v\|_{L^2(-h,0;L^4(\Tor^2))}.
\end{align*} 
Combining with $[L^2(\Tor^2),H^1(\Tor^2)]_{\frac{1}{2}}=H^{\frac{1}{2}}(\Tor^2)\into L^4(\Tor^2)$, interpolation and Young's inequality gives
\begin{align*}
  \int_0^t  \|&w(v)\pz  v\|_{L^2(\O)}^2\dd s 
   \leq \int_0^t\int_{-h}^0 \|w(v)\|_{L^4(\Tor^2)}^2\|\pz  v\|_{L^4(\Tor^2)}^2\dd x_3\dd s\\
  &\lesssim \int_0^t \|\nabla v\|_{L^2(-h,0;L^4(\Tor^2))}^2\|\pz  v\|_{L^2(-h,0;L^4(\Tor^2))}^2\dd s\\
  &\lesssim \int_0^t \|\nabla v\|_{L^2(-h,0;L^2(\Tor^2))}\|\nabla v\|_{L^2(-h,0;H^1(\Tor^2))}\|\pz  v\|_{L^2(-h,0;L^2(\Tor^2))}\|\pz  v\|_{L^2(-h,0;H^1(\Tor^2))}\dd s \\
  &\leq \int_0^t \eps \|v\|_{L^2(-h,0;H^2(\Tor^2))}^2+C_\eps\|v\|_{L^2(-h,0;H^1(\Tor^2))}^2\|\pz  v\|_{L^2(-h,0;L^2(\Tor^2))}^2\|\pz  v\|_{L^2(-h,0;H^1(\Tor^2))}^2\dd s \\
&\leq \eps \|v\|_{L^2(0,t;H^2(\O))}^2+ \int_0^t C_\eps\sup_{r\in[0,s]}\|v(r)\|_{H^1(\O)}^2\|\pz  v(s)\|_{L^2(\O)}^2\|\pz  v(s)\|_{H^1(\O)}^2\dd s.
\end{align*} 
For the next two estimates, recall Assumption  \ref{it:well_posedness_primitive_phi_smoothness} and let $\rho\in(2,6)$ be such that $\frac{1}{3+\delta}+\frac{1}{\rho}=\frac{1}{2}$. Then  $H^\tau(\O)\into L^\rho(\O)$ for $\tau\ceqq \frac{3}{2}-\frac{3}{\rho}\in(0,1)$, and using also  Remark  \ref{rem:assumptionslocal}, interpolation and Young's inequality,  we obtain: 
\begin{align*}
\int_0^t\|\mathcal{P}_{\gamma,\phi} (\cdot,v)\|_{L^2(\O)}^2\dd s
 &\lesssim \max_{1\leq l,m\leq 2}\int_0^t \|(\gamma_n^{l,m})_n\|_{L^{3+\delta}(\O;\ell^2)}^2\|({(\phi_n\cdot \nabla)v})_n\|_{L^\rho(\O;\ell^2)}^2\dd s\\ 
&\lesssim \int_0^t \| \nabla {v} \|_{H^\tau(\O)}^2\dd s \\ 
&\lesssim \eps\| {v}\|_{L^2(0,t;H^2(\O))}^2 + {C}_{\eps} \| {v}\|_{L^2(0,T;H^1(\O))}^2. 
\end{align*}
and  with Remark \ref{rem:g_n H1}, 
\begin{align*}
\int_0^t\|\mathcal{P}_{\gamma,G} (\cdot,v)\|_{L^2(\O)}^2\dd s
 &\lesssim \max_{1\leq l,m\leq 2}\int_0^t\|(\gamma_n^{l,m})_n\|_{L^{3+\delta}(\O;\ell^2)}^2\|(G_{v,n}(\cdot,v))_n\|_{L^\rho(\O;\ell^2)}^2\dd s\\
&\lesssim  \int_0^t\|(G_{v,n}(\cdot,v))_n\|_{H^1(\O;\ell^2)}^2\dd s\\
&\lesssim  \|\Xi\|_{L^2(0,T;L^2(\O))}^2+\|v\|_{L^2(0,T;H^1(\O))}^2. 
\end{align*}
By \eqref{eq:Jkappa est}, 
\begin{align*}
\int_0^t\|\mathcal{J}_\kappa {\theta}\|_{L^2(\O)}^2
\dd s\lesssim \|\theta\|_{L^2(0,T;H^1(\O))}^2. 
\end{align*} 
By Assumption \ref{it:sublinearity_Gforce}, we have 
\begin{align*}
\int_0^t\|\fv (\cdot,v,\theta,\nabla v)\|_{L^2(\O)}^2\dd s&\lesssim \|\Xi\|_{L^2(0,T;L^2(\O))}^2+\|\theta \|_{L^2(0,T;L^2(\O))}^2+ \|v \|_{L^2(0,T;H^1(\O))}^2.
\end{align*} 
Lastly, by \eqref{eq: coerc est}, together with Remarks \ref{rem:assumptionslocal} and  \ref{rem:g_n H1}, we have  
\begin{align*}
\int_0^t  \|&\big(\hhp[(\phi_{n}\cdot\nabla) v +\gvn(\cdot,v)]\big)_n\|_{H^1(\O;\ell^2)}\|\varphi\|_{\ell^2}\|v\|_{H^1(\O)}\dd s\\
&\lesssim \int_0^t\big(\| v\|_{H^2(\O)}+\|v\|_{L^2(\O)}+\|\Xi\|_{L^2(\O)}+\|v\|_{H^1(\O)}\big)\|\varphi\|_{\ell^2}\|v\|_{H^1(\O)}\dd s\\
&\lesssim \eps\|v\|_{L^2(0,t;H^2(\O))}^2+\|\Xi\|_{L^2(0,T;L^2(\O))}^2+ \int_0^tC_{\eps}\|\varphi(s)\|_{\ell^2}^2\sup_{r\in[0,s]}\|v(r)\|_{H^1(\O)}^2\dd s.
\end{align*}

Plugging all above estimates into \eqref{eq:skeleton H1} and taking a supremum over $r\in[0,t]$, we obtain the following estimate for all $t\in[0,T]$: 
\begin{align}\label{eq:H1gronwallprep}
\begin{split}
    & \sup_{r\in[0,t]} \|v(r)\|_{H^1(\O)}^2+\|v\|_{L^2(0,t;H^2(\O))}^2  \leq 2\sup_{r\in[0,t]}\Big[\|v(r)\|_{H^1(\O)}^2+\|v\|_{L^2(0,r;H^2(\O))}^2\Big] \\
     &\lesssim \hat{C}_{v_0}+  (\sigma+C_\sigma\eps+\eps)\|v\|_{L^2(0,t;H^2(\O))}^2 + C_{\sigma,\eps}\big(Y_{v,\theta}(T)+Z_v(T)+ Z_v(T)^2+\|\Xi\|_{L^2(0,T;L^2(\O))}^2\big)\\
     &\qquad+\int_0^t C_{\sigma,\eps}(1+\hat{Z}_v(s))\big(\|\bar{v}(s)\|_{H^1(\Tor^2)}^2+ \|\pz v(s)\|_{H^1(\O)}^2+\|\varphi(s)\|_{\ell^2}^2\big)\sup_{r\in[0,s]}\|v(r)\|_{H^1(\O)}^2\dd s,
     \end{split}
\end{align}
where $C_\sigma, C_{\sigma,\eps}>0$ are independent of $\varphi$ and $T$, and 
\begin{align}\label{eq:defsgronwallfunc}
\begin{split}
    &Y_{v,\theta}(t) \ceqq \|v\|_{L^2(0,t;H^1(\O))}^2+\|\theta\|_{L^2(0,t;H^1(\O))}^2,\\
    &Z_v(t) \ceqq \sup_{r\in[0,t]}\|\bar{v}(r)\|_{H^1(\Tor^2)}^2+\sup_{r\in[0,t]}\|\tilde{v}(r)\|_{L^4(\O)}^4+ \|\bar{v}\|_{L^2(0,t;H^2(\Tor^2))}^2+\||\tilde{v}||\nabla\tilde{v}|\|_{L^2(0,t;L^2(\O))}^2,\\
    &\hat{Z}_v(t) \ceqq \|\bar{v}(t)\|_{H^1(\Tor^2)}^2+\|\pz{v}(t)\|_{L^2(\O)}^2.
    \end{split}
\end{align} 
Observe that due to   \eqref{eq:skeleton L^2 est} and \eqref{eq:skeleton intermed est}, we have $Y_{v,\theta}(T)\lesssim {C}_{1,v_0,\theta_0,T,\varphi}$, $
Z_v(T)\lesssim {C}_{2,v_0,\theta_0,T,\varphi}$ and
\begin{align*} 
    \|(1+\hat{Z}_v)\big(\|\bar{v}\|_{H^1(\Tor^2)}^2+ &\|\pz v\|_{H^1(\O)}^2+ \|\varphi\|_{\ell^2}^2\big)\|_{L^1(0,T)}\\
   &\lesssim (1+{C}_{2,v_0,\theta_0,T,\varphi})({C}_{2,v_0,\theta_0,T,\varphi}(T+1)+\|\varphi\|_{L^2(0,T;\ell^2)}^2).
\end{align*}
Therefore, by first fixing a small $\sigma>0$ and then a small $\eps>0$,  and then applying Gr\"onwall's inequality,  and using the above estimates for $Y_{v,\theta},Z_v$ and $\hat{Z}_v$, we find that 
\begin{align}
\begin{split}
     \sup_{t\in[0,T]}\|v(t)&\|_{H^1(\O)}^2 +\|v\|_{L^2(0,T;H^2(\O;\R^2))}^2 \\[-6pt]
     &\lesssim  \hat{C}_{3,v_0,\theta_0,T,\varphi}\exp\big(c_3(1+{C}_{2,v_0,\theta_0,T,\varphi})\big({C}_{2,v_0,\theta_0,T,\varphi}(T+1)+\|\varphi\|_{L^2(0,T;\ell^2)}^2\big)\big) \\
     &\eqqc {C}_{3,v_0,\theta_0,T,\varphi},
\end{split}\label{eq:H1aftergronwall}
\end{align}
for a constant $c_3>0$ independent of $\varphi$, $T$, $v$ and $\theta$, and  
\begin{equation}\label{eq:def C3 hat}
    \hat{C}_{3,v_0,\theta_0,T,\varphi}\ceqq \hat{C}_{v_0}+  {C}_{1,v_0,\theta_0,T,\varphi}+{C}_{2,v_0,\theta_0,T,\varphi} +{C}_{2,v_0,\theta_0,T,\varphi}^2+\|\Xi\|_{L^2(0,T;L^2(\O))}^2. 
\end{equation}
Recalling the definitions of ${C}_{1,v_0,\theta_0,T,\varphi}$, ${C}_{2,v_0,\theta_0,T,\varphi}$ and $\hat{C}_{3,v_0,\theta_0,T,\varphi}$ (see \eqref{eq:skeleton L^2 est}, \eqref{eq: def exp exp constant}, \eqref{eq:def C3 hat}), we note that ${C}_{3,v_0,\theta_0,T,\varphi}$ is non-decreasing in $\|\varphi\|_{L^2(0,T;\ell^2)}$ and $T$, thus one can find  a function $C_{(v_0,\theta_0)}\col\R_+\times\R_+\to\R_+$ satisfying $C_{(v_0,\theta_0)}(T,\|\varphi\|_{L^2(0,T;\ell^2)})={C}_{3,v_0,\theta_0,T,\varphi}$, with $C_{(v_0,\theta_0)}$ non-decreasing in both components. Together with \eqref{eq:H1aftergronwall}, this completes the proof of Proposition \ref{prop:skeleton}.  \qed

\section{Controlled stochastic equation}\label{sec:tiltedspde}
 
In this section, we prove Proposition \ref{prop:tilt}, containing the a priori estimate for the stochastic equation \eqref{eq:primitive tilde and bar tilt}. The proof structure will be analogous to Section \ref{sec:skeleton}: the proof is given in three steps of estimates, distributed over Subsections \ref{sub:tilt L2}--\ref{sub:tilt H1}.   
But in contrast, the estimates here are all stochastic due to the stochastic processes $(\bphieps)$, and rather than e.g.\  applying the deterministic Gr\"onwall lemma to the expectation, we will need to exploit a stochastic Gr\"onwall lemma in each step, and work with delicate estimates that hold only in probability.  

Throughout this section, we assume that $(\bphieps)$ and $(v,\theta)$ are as in Proposition \ref{prop:tilt}: we assume that  $(v_0,\theta_0)\in \H^1(\O)\times L^2(\O)$ and   $T>0$ are given. Moreover,  $(\bphieps)_{\eps>0}$ is a collection of predictable stochastic processes bounded by $K$ in the sense of \eqref{eq: as bdd K}. 
Moreover, we assume that for each $\eps\in(0,1]$,  $(v^\eps,\theta^\eps)$ is an $L^2$-strong-weak solution to \eqref{eq:primitive tilde and bar tilt} on $[0,T]$ corresponding to $\bphieps$, which exists uniquely by virtue of Remark \ref{rem: tilt well posed}.

\subsection{$L^2$-estimates for the horizontal velocity and temperature}\label{sub:tilt L2}
We start by proving $L^2$-estimates for the velocities $v^\eps$ and temperatures $\theta^\eps$, corresponding to the analytically weak setting. 

By the It\^o formula in \cite[Th.\ 3.1 p.\ 57, Th.\ 3.3 p.\ 59]{pardoux} (see also \cite[Lem.\ A.3]{TV24}) on the Gelfand triple $(\mathbb{H}^1(\O)\times H^1(\O),\mathbb{L}^2(\O)\times L^2(\O),(\mathbb{H}^1(\O))^*\times H^{-1}(\O))$, and 
with the notations $I_1^{\scriptscriptstyle L^2}$, $I_2^{\scriptscriptstyle L^2}$ defined in \eqref{eq:L^2 norm v} and \eqref{eq:L^2 norm theta}, we have a.s.:
\begin{align}
     \|v^\eps(t)\|_{{L}^2(\O)}^2- \|v_0\|_{{L}^2(\O)}^2   
    &= 2I_1^{\scriptscriptstyle L^2}(t,v^\eps,\theta^\eps,\bphieps)\notag\\
     &+2\sqrt{\eps}\sum_{n\geq 1} \int_0^t \big\<v^\eps, \hhp \big[(\phi_{n}\cdot\nabla) v^\eps  +\gvn(\cdot,v^\eps)\big]  \dd \beta^n_s\big\?_{L^2(\O)}\notag\\ 
      &+\eps \int_0^t \big\|\big(\hhp \big[(\phi_{n}\cdot\nabla) v^\eps  +\gvn(\cdot,v^\eps)\big]\big)_n\big\|_{L^2(\O;\ell^2)}^2\dd s,\label{eq:J1}\\
     \|\theta^\eps(t,\om)\|_{{L}^2(\O)}^2- \|\theta_0\|_{{L}^2(\O)}^2   
    &= 2 I_2^{\scriptscriptstyle L^2}(t,v^\eps,\theta^\eps,\bphieps)\notag\\
    &+ 2\sqrt{\eps}\sum_{n\geq 1} \int_0^t \big\<\theta^\eps, (\psi_n\cdot \nabla) \theta^\eps+\gtn(\cdot,v^\eps,\theta^\eps,\nabla v^\eps) \dd \beta^n_s\big\?_{L^2(\O)}\notag\\
    &+\eps \int_0^t \big\| \big((\psi_n\cdot \nabla) \theta^\eps+\gtn(\cdot,v^\eps,\theta^\eps,\nabla v^\eps)\big)_n \big\|_{L^2(\O;\ell^2)}^2\dd s. \label{eq:J2}
\end{align}
Note that the functions $I_1^{\scriptscriptstyle L^2}$, $I_2^{\scriptscriptstyle L^2}$ were introduced for the deterministic skeleton equation. But now we apply these functions pointwise in $\om\in\Om$  ($v^\eps$, $\theta^\eps$ and $\bphieps$  all depend on $\om$). 

Estimates for $I_1^{L^2}$ and $I_2^{L^2}$ were already derived in Subsection \ref{sub:skeleton L2} and were summarized by \eqref{eq:skeleton L^2 est prep}. 
Applying these estimates for each fixed $\om\in\Om$, we find that pointwise on $\Om$:
\begin{align*}
I_1^{\scriptscriptstyle L^2}(t,v^\eps,\theta^\eps,\bphieps)+I_2^{\scriptscriptstyle L^2}(t,v^\eps,\theta^\eps,\bphieps)
&\lesssim
R_T+\int_0^t \big(1+\|\bphieps(s)\|_{\ell^2}^2\big) \Big(\|v^\eps(s)\|_{{L}^2(\O)}^2+\|\theta^\eps(s)\|_{{L}^2(\O)}^2\Big)\dd s\\
&\qquad-\alpha\Big[\|\nabla v^\eps\|_{L^2(0,t;{L}^2(\O))}^2+ \|\nabla \theta^\eps\|_{L^2(0,t;{L}^2(\O))}^2\Big],
\end{align*}
for some constants $\alpha>0$ and $R_T\geq 0$. 
Furthermore, the quadratic variation terms in the equations for $v^\eps$ and $\theta^\eps$ can be estimated using Remark \ref{rem:assumptionslocal} and Assumption  \ref{it:sublinearity_Gforce}:   
\begin{align*}
&\eps\big\|\big(\hhp \big[(\phi_{n}\cdot\nabla) v^\eps  +\gvn(\cdot,v^\eps)\big]\big)_n\big\|_{L^2(\O;\ell^2)}^2\lesssim
\eps \|\nabla v^\eps\|_{L^2(\O)}^2+\eps \big(\|\Xi\|_{L^2(\O)}^2+ \|v^\eps\|_{L^2(\O)}^2\big),\\ 
& \eps\big\| \big((\psi_n\cdot \nabla) \theta^\eps +\gtn(\cdot,v^\eps,\theta^\eps,\nabla v^\eps)\big)_n \big\|_{L^2(\O;\ell^2)}^2\dd s \lesssim  
 \eps \|\nabla \theta^\eps\|_{L^2(\O;\R^{3})}^2+\eps \big(\|\Xi\|_{L^2(\O)}^2+\\
&\hspace{8.2cm} \|v^\eps\|_{L^2(\O)}^2+\|\theta^\eps\|_{L^2(\O)}^2+\|\nabla v^\eps\|_{L^2(\O)}^2\big)
\end{align*}  
The remaining terms in the equations for $v^\eps$ and $\theta^\eps$ are martingales. Thus, we can pick $\eps_0'' \in(0,1)$ sufficiently small so that for all $\eps\in (0,\eps_0'')$ and $\gamma>0$, we have by the stochastic Gr\"onwall lemma \cite[Cor.\ 5.4b), (50)]{geiss24} and by \eqref{eq: as bdd K}:  
\begin{align}\label{eq: tilt SPDE L^2}
\P\Big(\sup_{t\in[0,T]}y_\eps(t)>\gamma\Big) \leq C_0\frac{\exp(C_1(T+K))}{\gamma}\big(R_T+\|v_0\|_{{L}^2(\O)}^2+\|\theta_0\|_{{L}^2(\O)}^2\big)\eqqc \frac{C_{v_0,\theta_0,K,T}}{\gamma},
\end{align} 
where 
\[
y_\eps(t)\ceqq \|v^\eps(t)\|_{{L}^2(\O)}^2+\|\theta^\eps(t,\om)\|_{{L}^2(\O)}^2+\|\nabla v^\eps\|_{L^2(0,t;{L}^2(\O))}^2+ \|\nabla \theta^\eps\|_{L^2(0,t;{L}^2(\O))}^2.
\]

\subsection{Intermediate estimates for barotropic and baroclinic modes}\label{sub:tilt intermed}
 
Now we will establish intermediate estimates for the horizontal velocities $v^\eps$. 
Using the notation introduced in \eqref{eq: def bar tilde}, let us write $\bar{v}^\eps\ceqq \overline{v^\eps}$ for the barotropic mode and $\tilde{v}^\eps\ceqq \wtveps={v}^\eps-\bar{v}^\eps$ for the baroclinic mode. Then we have the decomposition $v^\eps=\bar{v}^\eps+\tilde{v}^\eps$, and similarly as in Subsection \ref{sub:skeleton intermed}, we will derive  $H^1$-estimates for $\bar{v}^\eps$, $L^2$-estimates for $\pz v^\eps$ and $L^4$-estimates for $\tilde{v}^\eps$, but now all in probability.

Similar to \eqref{eq:primitive tilde and bar}  and to \cite[(5.21)]{AHHS22}, one can show that the pair $(\bar{v}^\eps,\tilde{v}^\eps)$ is a solution to the following system:
\begin{align}\label{eq:primitive tilde and bar tilt}
\begin{cases}
 \dd \bar{v}^\eps  -\Delta_{\h} \bar{v}^\eps\dd t= \pr\big[-(\bar{v}^\eps\cdot \nabla_{\h})\bar{v}^\eps- \force(\tilde{v}^\eps) + \overline{\fvt^\eps}+ \mathcal{P}_{\gamma,\phi}(\cdot,v^\eps)\big]\dd t \hspace{3cm}&\\
\qquad\qquad\qquad\qquad\qquad+\sum_{n\geq 1}\pr \big[(\phi_{n,\h}\cdot\nabla_{\h}) \bar{v}^\eps+\overline{\phi^3_n \pz  v^\eps }  +\overline{\gvtn^\eps}\big] \bphieps_n \dd t\\
\qquad\qquad\qquad\qquad\qquad\qquad+\sqrt{\eps}\sum_{n\geq 1}\pr \big[ (\phi_{n,\h}\cdot\nabla_{\h}) \bar{v}^\eps+\overline{\phi^3_n \pz  v^\eps }  +\overline{\gvtn^\eps}\big]\dd \beta^n_t &\text{ on }\Tor^2, \\
\bar{v}^\eps(\cdot,0)=\bar{v}_0& \text{ on }\Tor^2,\\
\dd \tilde{v}^\eps  -\Delta \tilde{v}^\eps\dd t=\big[-(\tilde{v}^\eps\cdot \nabla_{\h})\tilde{v}^\eps+\forcetwo(\tilde{v}^\eps,\bar{v}^\eps) + \widetilde{\fvt^\eps} \,\big]  \dd t
+\sum_{n\geq 1}\big[(\phi_{n}\cdot\nabla) \tilde{v}^\eps-\overline{\phi^3_n \pz  v^\eps } +\widetilde{\gvtn^\eps}\,\big] \bphieps_n \dd t\\
\qquad\qquad\qquad\qquad\qquad\qquad+\sqrt{\eps}\sum_{n\geq 1}\big[(\phi_{n}\cdot\nabla) \tilde{v}^\eps-\overline{\phi^3_n \pz  v^\eps } +\widetilde{\gvtn^\eps}\,\big]\dd \beta^n_t&\text{ on }\O,\\
\pz  \tilde{v}^\eps(\cdot,-h)=\pz  \tilde{v}^\eps(\cdot,0)=0&\text{ on }\ \Tor^2,\\
\tilde{v}^\eps(\cdot,0)=\tilde{v}_0\ceqq v_0-\bar{v}_0&\text{ on }\O,
\end{cases}
\end{align} 
with $\force$, $\forcetwo$ as defined in \eqref{eq: abbrevs} and 
\begin{align*}
&f^\eps\ceqq \mathcal{J}_\kappa{\theta^\eps}+F_v(\cdot,v^\eps,\theta^\eps,\nabla v^\eps)+\mathcal{P}_{\gamma,G} (t,v^\eps),
\qquad 
g_n^\eps\ceqq G_{v,n}(\cdot,v^\eps). 
\end{align*}
Using the notations of \eqref{eq: Ito bar v def}, \eqref{eq: Ito pz v} and  \eqref{eq: Ito tilde v L4}, we have by the It\^o  formula (see the proof of \cite[Lem.\ A.6]{AV25survey}, \cite[\S3]{dareiotis}): 
\begin{align}
\|\bar{v}^\eps(t)\|_{{H}^1(\Tor^2)}^2- \|\bar{v}_0\|_{{H}^1(\Tor^2)}^2   &= 2\bar{I}_{1}^{\scriptscriptstyle H^1}(t,v^\eps,\theta^\eps,\bphieps)\notag\\
&\qquad+2\sqrt{\eps}\sum_{n\geq 1}\int_0^t \big\<\bar{v}^\eps,\pr\big[  (\phi_{n,\h}\cdot\nabla_{\h}) \bar{v}^\eps+\overline{\phi^3_n \pz  v^\eps }  +\overline{\gvtn^\eps}\big]\dd \beta^n_s\big\?_{H^1(\Tor^2)}\notag\\
&\qquad +\eps\int_0^t \Big\|\Big(\pr \big[(\phi_{n,\h}\cdot\nabla_{\h}) \bar{v}^\eps+\overline{\phi^3_n \pz  v^\eps }  +\overline{\gvtn^\eps}\big]\Big)_n \Big\|_{H^1(\Tor^2;\ell^2)}^2\dd s, \label{eq: quadr bar v}
\\
\|\pz v^\eps(t)\|_{L^2(\O)}^2  - \|\pz v_0\|_{L^2(\O)}^2  & = 2I_{1}^{\scriptscriptstyle \pz}(t,v^\eps,\theta^\eps,\bphieps)\notag\\
&+2\sqrt{\eps}\sum_{n\geq 1} \int_0^t \big\<\pz v^\eps, \pz\hhp \big[(\phi_{n}\cdot\nabla) v^\eps  +\gvn(\cdot,v^\eps)\big]  \dd \beta^n_s\big\?_{L^2(\O)}\notag\\
&+\eps \int_0^t \big\|\big(\pz\hhp \big[(\phi_{n}\cdot\nabla) v^\eps  +\gvn(\cdot,v^\eps)\big]\big)_n\big\|_{L^2(\O;\ell^2)}^2\dd s, \label{eq: quadr pz v}
\\
\|\tilde{v}^\eps(t)\|_{L^4(\O)}^4- \|v_0-\bar{v}_0\|_{L^4(\O)}^4 &= 4\tilde{I}_{1}^{\scriptscriptstyle L^4}(t,v^\eps,\theta^\eps,\bphieps) \notag\\
&+4\sqrt{\eps}\sum_{n\geq 1}\int_0^t \big\<|\tilde{v}^\eps|^2\tilde{v}^\eps,\big((\phi_{n}\cdot\nabla) \tilde{v}^\eps-\overline{\phi^3_n \pz  v^\eps } +\widetilde{\gvtn^\eps}\,\big)\dd \beta^n_s\big\?_{L^2(\O)}\notag\\
&+2\eps \int_0^t\int_\O |\tilde{v}^\eps|^2\Big\| \Big((\phi_{n}\cdot\nabla) \tilde{v}^\eps-\overline{\phi^3_n \pz  v^\eps } +\widetilde{\gvtn^\eps}\,\Big)_n\Big\|_{\ell^2}^2\dd x\dd s \label{eq: quadr tilde v}\\
&+4\eps \int_0^t\int_\O \sum_{n\geq 1}\big|\tilde{v}^\eps \cdot \big((\phi_{n}\cdot\nabla) \tilde{v}^\eps-\overline{\phi^3_n \pz  v^\eps } +\widetilde{\gvtn^\eps}\big)\big|^2\dd x\dd s. \label{eq: quadr tilde v 2nd part}
\end{align} 
Now, let us define 
\[
H_c^\eps(t)\ceqq  \|\bar{v}^\eps(t)\|_{H^1(\Tor^2)}^2+\|\pz v^\eps(t)\|_{L^2(\O)}^2+c\|\tilde{v}^\eps(t)\|_{L^4(\O)}^4.
\] 
Applying the  estimates from Subsection \ref{sub:skeleton intermed} (summarized by \eqref{eq: Gronwall prep intermed}) for each fixed $\om\in\Om$, we find that pointwise on $\Om$, we have 
\begin{align*}
|\bar{I}_{1}^{\scriptscriptstyle H^1}(t&,v^\eps,\theta^\eps,\bphieps)|+|I_{1}^{\scriptscriptstyle \pz}(t,v^\eps,\theta^\eps,\bphieps)|+c|\tilde{I}_{1}^{\scriptscriptstyle L^4}(t,v^\eps,\theta^\eps,\bphieps)|\\
&\lesssim_c  -\int_0^t\|\bar{v}^\eps\|_{H^2(\Tor^2)}^2\dd s-\int_0^t \|\nabla\pz v^\eps\|_{L^2(\O)}^2\dd s -\int_0^t\||\tilde{v}^\eps||\nabla\tilde{v}^\eps|\|_{L^2(\O)}^2\dd s\\
&+\int_0^t (1+H^\eps_c)\Big[1+\|\Xi\|_{L^2(\O)}^2+\|\theta^\eps\|_{L^2(\O)}^2+\|\nabla \theta^\eps\|_{L^2(\O)}^2 +\|v^\eps\|_{L^2(\O)}^2+\|\nabla v^\eps\|_{L^2(\O)}^2\\
&\qquad\qquad\qquad\qquad\qquad\qquad\qquad\qquad\qquad\qquad\qquad+\|v^\eps\|_{L^2(\O)}^2 \|\nabla v^\eps\|_{L^2(\O)}^2+\|\bphieps\|_{\ell^2}^2)\Big]\dd s,
\end{align*} 
where we use that $c>0$ in \eqref{eq: Gronwall prep intermed} was independent of $\varphi$, $T$, $v_0$, $\theta_0$, $v$ and $\theta$ (see the lines below \eqref{eq: Gronwall prep intermed}). 

Now we turn to the quadratic variation terms for $\bar{v}^\eps$, $\pz v^\eps$ and $\tilde{v}^\eps$, starting with \eqref{eq: quadr bar v}. Recall \eqref{eq: grad hhp} and note that $\hhp_{\h} f=\hhp f$ if $f$ is $x_3$-independent. Therefore, by the $x_3$-independence of $\phi_{n,\h}$ (Assumption \ref{it:independence_z_variable}) and by  Remark \ref{rem: removing bar or tilde}, we have 
\begin{align*}
     \Big\|\Big(\pr \big[(\phi_{n,\h}\cdot\nabla_{\h}) \bar{v}^\eps+\overline{\phi^3_n \pz  v^\eps }+&\overline{\gvtn^\eps} \big]\Big)_n \Big\|_{H^1(\Tor^2;\ell^2)}^2
     =\Big\|\Big(\hhp \big[(\phi_{n,\h}\cdot\nabla_{\h})  \bar{v}^\eps+\overline{\phi^3_n \pz  v^\eps }  +\overline{\gvtn^\eps}\big]\Big)_n \Big\|_{H^1(\Tor^2;\ell^2)}^2 \\ 
     &\lesssim \|\big( (\phi_{n,\h}\cdot\nabla_{\h}) \bar{v}^\eps\big)_n \|_{L^2(\Tor^2;\ell^2)}^2 +\|\big(\nabla_{\h} (\phi_{n,\h}\cdot\nabla_{\h}) \bar{v}^\eps\big)_n \|_{L^2(\Tor^2;\ell^2)}^2\\
     &\qquad\qquad+ \|(\phi^3_n \pz  v^\eps)_n\|_{H^1(\O;\ell^2)}^2+\|(\gvtn^\eps)_n\|_{H^1(\O;\ell^2)}^2.
\end{align*} 
Now, by Remark \ref{rem:assumptionslocal},
\[
\|\big((\phi_{n,\h}\cdot\nabla_{\h}) \bar{v}^\eps\big)_n \|_{L^2(\Tor^2;\ell^2)}^2\lesssim \|\bar{v}^\eps\|_{H^1(\Tor^2)}^2,
\]
and by integration by parts, Remark \ref{rem:assumptionslocal} and Assumption  \ref{it:well_posedness_primitive_phi_smoothness}, we have 
\begin{align*}
\|\big(\nabla_{\h} (\phi_{n,\h}\cdot\nabla_{\h}) \bar{v}^\eps\big)_n \|_{L^2(\Tor^2;\ell^2)}^2
&\lesssim
\sum_{j=1}^2\| \big((\partial_j\phi_{n,\h}\cdot\nabla_{\h}) \bar{v}^\eps \big)_n\|_{L^2(\Tor^2;\ell^2)}^2+\| \big((\phi_{n,\h}\cdot\partial_j\nabla_{\h}) \bar{v}^\eps\big)_n \|_{L^2(\Tor^2;\ell^2)}^2\\
&\lesssim 
\sum_{j=1}^2\|(\partial_j\phi_{n,\h})_n\|_{L^{3+\delta}(\Tor^2;\ell^2)}^2\|\nabla_{\h}\bar{v}^\eps \|_{L^\rho(\Tor^2)}^2+ \|\bar{v}^\eps\|_{H^2(\Tor^2)}^2 \\
&\lesssim \|\bar{v}^\eps\|_{H^2(\Tor^2)}^2,
\end{align*}
where  $\rho\in(0,6)$ satisfies $\frac{1}{2}=\frac{1}{3+\delta}+\frac{1}{\rho}$, thus $H^2(\Tor^2)\into H^{1,\rho}(\Tor^2)$. 
Next, recall \eqref{eq:pzz est}:  
\begin{align*}
    \|(\phi^3_n \pz  v^\eps)_n\|_{H^1(\O;\ell^2)}^2& \lesssim  \| \pz v^\eps\|_{L^2(\O)}^2+\| \nabla\pz v^\eps\|_{L^2(\O)}^2.
\end{align*} 
Lastly, by  Remark \ref{rem:g_n H1},
\[
\|(\gvtn^\eps)_n\|_{H^1(\O;\ell^2)}^2\lesssim  \|\Xi\|_{L^2(\O)}^2+\|v^\eps\|_{H^1(\O)}^2.  
\] 
Combining the above estimates, we conclude that
\begin{align*}
    \eps\int_0^t \Big\|\Big(\pr &\big[(\phi_{n,\h}\cdot\nabla_{\h}) \bar{v}^\eps+\overline{\phi^3_n \pz  v^\eps }  +\overline{\gvtn^\eps}\big]\Big)_n \Big\|_{H^1(\Tor^2;\ell^2)}^2\dd s\\
    &\lesssim \int_0^t\eps \|\bar{v}^\eps\|_{H^2(\Tor^2)}^2+\eps \| \pz v^\eps\|_{L^2(\O)}^2+\eps\| \nabla\pz v^\eps\|_{L^2(\O)}^2+\eps\|\Xi\|_{L^2(\O)}^2+ \eps\|v^\eps\|_{H^1(\O)}^2\dd s. 
\end{align*}
Now we turn to \eqref{eq: quadr pz v} for $\pz v^\eps$. 
By Remarks \ref{rem:g_n H1} and \ref{rem:assumptionslocal},   $\pz \hhp f=\pz f$, $(\phi_{n,\h})$ is independent of $x_3$ (Assumption \ref{it:independence_z_variable}) and by \eqref{eq:pzz est}, we obtain 
\begin{align*}
   & \eps \int_0^t  \big\|\big(\pz\hhp \big[(\phi_{n}\cdot\nabla) v^\eps  +\gvn(\cdot,v^\eps)\big]\big)_n\big\|_{L^2(\O;\ell^2)}^2\dd s \\
    &\leq  \eps \int_0^t  \big\| (\phi_{n,\h}\cdot\nabla_{\h})\pz v^\eps  +\pz(\phi_{n}^3\pz v^\eps)  +\pz\gvn(\cdot,v^\eps)\big]\big)_n\big\|_{L^2(\O;\ell^2)}^2\dd s \\
    &\lesssim \eps\int_0^t \| \nabla \pz v^\eps\|_{L^2(\O)}^2  +\| \pz v\|_{L^2(\O)}^2+\|\Xi \|_{L^2(\O)}^2+ \|v^\eps\|_{H^1(\O)}^2 \dd s. 
\end{align*}
Concerning the term \eqref{eq: quadr tilde v} for $\tilde{v}^\eps$,  
 we have by Remark \ref{rem:assumptionslocal}:
\begin{align*}
\int_\O |\tilde{v}^\eps|^2 \big\| \big((\phi_{n}\cdot\nabla) \tilde{v}^\eps \big)_n\big\|_{\ell^2}^2\dd x   
&    \lesssim  \||\tilde{v}^\eps|  |\nabla \tilde{v}^\eps|\|_{L^2(\O)}^2 
\end{align*}
and, using also that  $[L^2(\Tor^2),H^1(\Tor^2)]_{1/2}=H^{1/2}(\Tor^2)\into L^4(\Tor^2)$ and using Remark \ref{rem: removing bar or tilde} and \eqref{eq:pzz est},  
\begin{align*}
&\int_\O |\tilde{v}^\eps|^2 \big\| \big(\overline{\phi^3_n \pz  v^\eps }\big)_n\big\|_{\ell^2}^2\dd x   
\lesssim  \|\tilde{v}^\eps\|_{L^4(\O)}^2  \big\| \big(\overline{\phi^3_n \pz  v^\eps }\big)_n\big\|_{L^4(\Tor^2;\ell^2)}^2\\
&\lesssim \|\tilde{v}^\eps\|_{L^4(\O)}^2\big\| \big(\overline{\phi^3_n \pz  v^\eps }\big)_n\big\|_{L^2(\Tor^2;\ell^2)}\Big( \big\| \big(\overline{\phi^3_n \pz  v^\eps }\big)_n\big\|_{L^2(\Tor^2;\ell^2)}+\big\| \big( \nabla_{\h}(\overline{\phi^3_n \pz  v^\eps }) \big)_n\big\|_{L^2(\Tor^2;\ell^2)}\Big)\\
&\lesssim \|\tilde{v}^\eps\|_{L^4(\O)}^2\|\pz v^\eps\|_{L^2(\O)}( \|\pz v^\eps\|_{L^2(\O)}+\|\nabla(\phi_n^3\pz v^\eps)\|_{L^2(\O)})\\
&\lesssim \|\tilde{v}^\eps\|_{L^4(\O)}^4\|\nabla v^\eps\|_{L^2(\O)}^2+ \|\pz v^\eps\|_{L^2(\O)}^2+\|\nabla\pz v^\eps\|_{L^2(\O)}^2.
\end{align*} 
Next, by  H\"older's inequality, Remark \ref{rem: removing bar or tilde} and the embedding $H^1(\O)\into L^4(\O)$, we have
\begin{align*}
\int_\O |\tilde{v}^\eps|^2 \big\| \big(\widetilde{\gvtn^\eps}\big)_n\big\|_{\ell^2 }^2\dd x   
&    \lesssim \|\tilde{v}^\eps \|_{L^4(\O)}^2\|\big({\gvtn^\eps}\big)_n\big\|_{L^4(\O;\ell^2)}^2     \lesssim (1+\|\tilde{v}^\eps \|_{L^4(\O)}^4)(\|\Xi\|_{L^2(\O)}^2+\|v^\eps\|_{H^1(\O)}^2). 
\end{align*}

Collecting the above estimates and noting that  \eqref{eq: quadr tilde v 2nd part} can be estimated by twice the quantity of \eqref{eq: quadr tilde v}, we conclude that  for all $\eps\in(0,\eps_0'')$: 
\begin{align*}
 &   2\eps \int_0^t\int_\O |\tilde{v}^\eps|^2\Big\| \Big((\phi_{n}\cdot\nabla) \tilde{v}^\eps-\overline{\phi^3_n \pz  v^\eps } +\widetilde{\gvtn^\eps}\,\Big)_n\Big\|_{\ell^2}^2\dd x\dd s \\
 &\qquad\qquad+
 4\eps \int_0^t\int_\O \sum_{n\geq 1}\big|\tilde{v}^\eps \cdot \big((\phi_{n}\cdot\nabla) \tilde{v}^\eps-\overline{\phi^3_n \pz  v^\eps } +\widetilde{\gvtn^\eps}\big)\big|^2\dd x\dd s\\
&    \lesssim  \eps \int_0^t \||\tilde{v}^\eps|  |\nabla \tilde{v}^\eps|\|_{L^2(\O)}^2+\|\nabla\pz v^\eps\|_{L^2(\O)}^2 \dd s\\
&\qquad+\eps    \int_0^t \|\pz v^\eps\|_{L^2(\O)}^2+ (1+\|\tilde{v}^\eps \|_{L^4(\O)}^4)(\|\Xi\|_{L^2(\O)}^2+\|v^\eps\|_{H^1(\O)}^2)  \dd s.
\end{align*} 
Now, for a sufficiently small $\eps_0'\in (0,\eps_0'')\subset(0,1)$, the above estimates give for all $\eps\in(0,\eps_0')$: 
\begin{align}
z^\eps(t)\lesssim C_{v_0}+\int_0^t (1+ z^\eps)\Big[1+\|\Xi\|_{L^2(\O)}^2&+\|\theta^\eps\|_{H^1(\O)}^2 +\|v^\eps\|_{H^1(\O)}^2 \notag\\
&+\|v^\eps\|_{L^2(\O)}^2 \|\nabla v^\eps\|_{L^2(\O)}^2+\|\bphieps\|_{\ell^2}^2)\Big]\dd s+M^\eps(t), \label{eq:z^epsest}
\end{align}
where $M^\eps$ is a continuous local martingale and 
\begin{align*}
z^\eps(t)\ceqq & \,\|\bar{v}^\eps(t)\|_{H^1(\Tor^2)}^2+\|\pz v^\eps(t)\|_{L^2(\O)}^2+\|\tilde{v}^\eps(t)\|_{L^4(\O)}^4\\
&+\int_0^t\|\bar{v}^\eps(s)\|_{H^2(\Tor^2)}^2\dd s+\int_0^t \|\nabla\pz v^\eps(s)\|_{L^2(\O)}^2\dd s+\int_0^t\||\tilde{v}^\eps(s)||\nabla\tilde{v}^\eps(s)|\|_{L^2(\O)}^2\dd s.
\end{align*}
Putting 
\[
a^\eps \ceqq\int_0^\cdot\Big[1+\|\Xi\|_{L^2(\O)}^2+\|\theta^\eps\|_{H^1(\O)}^2  +\|v^\eps\|_{H^1(\O)}^2 +\|v^\eps\|_{L^2(\O)}^2 \|\nabla v^\eps\|_{L^2(\O)}^2+\|\bphieps\|_{\ell^2}^2)\Big]\dd s,
\]
and applying the stochastic Gr\"onwall lemma \cite[Cor.\ 5.4b), (50)]{geiss24}  (note that $a^\eps$ is left-continuous and adapted with constant filtration $(\F_T)$, thus predictable),  
we   obtain for all $\eps\in(0,\eps_0')$ and $\gamma\geq 1$: 
\begin{align}
\P\Big(\sup_{t\in[0,T]}z_\eps(t)>\gamma\Big)& \leq  \frac{\exp(\frac{1}{2}\log\gamma)}{\gamma}\E[C_2(C_{v_0}+a^\eps(T))\wedge \gamma^{1/4}]  \notag\\
&+\P\big(C_2(C_{v_0}+a^\eps(T))> \gamma^{1/4}  \big)+\P\big(a^\eps(T)>\frac{\frac{1}{2}\log\gamma}{C_2}\big)\notag\\ 
&\leq \gamma^{-1/4} +\gamma^{-1/4} \hat{C}_{v_0,\theta_0,K,T}+(\frac{1}{2}\log\gamma)^{-1}\hat{C}_{v_0,\theta_0,K,T}. \label{eq: tilt SPDE intermed} 
\end{align} 
for some constant $\hat{C}_{v_0,\theta_0,K,T}$ depending on $C_2, C_{v_0},{C}_{v_0,\theta_0,K,T}, T,K$ and $\|\Xi\|_{L^2(0,T;L^2(\O)}$. 
In the  last line, we used \eqref{eq: tilt SPDE L^2} and the fact  that $\gamma\geq 1$ and
\[
a^\eps(T)\leq T+\|\Xi\|_{L^2(0,T;L^2(\O)}^2+y_\eps(T)+(T+y_\eps(T))\sup_{t\in[0,T]}y_\eps(t)+K^2. 
\] 

\subsection{$H^1$-estimate for the horizontal velocity}\label{sub:tilt H1} 
 
Now we use the intermediate estimate \eqref{eq: tilt SPDE intermed} and $L^2$-estimate \eqref{eq: tilt SPDE L^2} to prove that
$\lim_{\gamma\to\infty}\sup_{\eps\in(0,\eps_0)}\P(\|X^\eps\|_{\MR(0,T)}>\gamma)=0$, 
for some $\eps_0>0$, where $\MR(0,T)$ was defined in \eqref{eq: def MR space}.  Note that \eqref{eq: tilt SPDE L^2} already yields the needed estimates for $\theta$. Thus it remains to prove that 
\begin{equation}\label{eq:tilt to show H1}
\lim_{\gamma\to\infty}\sup_{\eps\in(0,\eps_0)}\P(\|v^\eps\|_{C([0,T];\H^1(\O))} +  \|v^\eps\|_{L^2(0,T;\H_{\n}^2(\O))}>\gamma)=0. 
\end{equation} 

Using the notation of \eqref{eq: Ito v H1 def}, we have 
\begin{align}
\|v^\eps(t)\|_{{H}^1(\O)}^2- \|v_0\|_{{H}^1(\O)}^2 & = 2I_1^{H^1}(t,v^\eps,\theta^\eps,\bphieps)\notag\\
&\qquad+2\sqrt{\eps}\sum_{n\geq 1} \int_0^t \big\<v^\eps, \hhp \big[(\phi_{n}\cdot\nabla) v^\eps  +\gvn(\cdot,v^\eps)\big]  \dd \beta^n_s\big\?_{H^1(\O)}\notag\\ 
&\qquad+\eps \int_0^t \big\|\big(\hhp \big[(\phi_{n}\cdot\nabla) v^\eps  +\gvn(\cdot,v^\eps)\big]\big)_n\big\|_{H^1(\O;\ell^2)}^2\dd s. \label{eq:J6}
\end{align} 
Applying the estimates from Subsection \ref{sub:skeleton H1}  for each fixed $\om\in\Om$, and using the expression and notation from \eqref{eq:H1gronwallprep} and \eqref{eq:defsgronwallfunc}), we find that 
\begin{align*}
   I_1^{H^1}&(t,v^\eps,\theta^\eps,\bphieps) \lesssim - \|v^\eps\|_{L^2(0,t;H^2(\O))}^2+
      Y_{v^\eps,\theta^\eps}(t)+Z_{v^\eps}(t)+ Z_{v^\eps}(t)^2+\|\Xi\|_{L^2(0,T;L^2(\O))}^2 \\
     &+\int_0^t  (1+\hat{Z}_{v^\eps}(s))\big(\|\bar{v}^\eps(s)\|_{H^1(\Tor^2)}^2+ \|\pz v^\eps(s)\|_{H^1(\O)}^2+\|\bphieps(s)\|_{\ell^2}^2\big)\sup_{r\in[0,s]}\|v^\eps(r)\|_{H^1(\O)}^2\dd s. 
\end{align*}   
Furthermore, we have by   \eqref{eq: coerc est} and Remarks \ref{rem:assumptionslocal} and \ref{rem:g_n H1}:
\begin{align*}
& \eps \int_0^t \big\|\big(\hhp \big[(\phi_{n}\cdot\nabla) v^\eps  +\gvn(\cdot,v^\eps)\big]\big)_n\big\|_{H^1(\O;\ell^2)}^2\dd s \lesssim \eps\int_0^t \|v^\eps\|_{H^2(\O)}^2+\|\Xi\|_{L^2(\O)}^2\dd s.
\end{align*}
Now let 
\[
x^\eps(t)\ceqq \|v^\eps(t)\|_{{H}^1(\O)}^2+\int_0^t \|v^\eps \|_{{H}^2(\O)}^2\dd s.
\]
The estimates above give that for a sufficiently small $\eps_0\in (0,\eps_0')\subset(0,1)$, we have for all $\eps\in(0,\eps_0)$: 
\begin{equation}\label{eq:x^epsest}
  x^\eps(t)\lesssim C_{v_0}+ h^\eps(t)+\int_0^t \sup_{r\in[0,s]}\|v^\eps(r)\|_{H^1(\O)}^2 b^\eps(s)\dd s +M^\eps(t),
\end{equation}
where $M^\eps$ is a continuous local martingale and 
\begin{align*}
&h^\eps(t)=Y_{v^\eps,\theta^\eps}(t)+Z_{v^\eps}(t)+ Z_{v^\eps}(t)^2+\|\Xi\|_{L^2(0,T;L^2(\O))}^2,\\
&b^\eps(t)=(1+\hat{Z}_{v^\eps}(t))\big(\|\bar{v}^\eps(t)\|_{H^1(\Tor^2)}^2+ \|\pz v^\eps(t)\|_{H^1(\O)}^2+\|\bphieps(t)\|_{\ell^2}^2\big). 
\end{align*}
The stochastic Gr\"onwall lemma \cite[Cor.\ 5.4b), (50)]{geiss24} thus  gives for all $\eps\in(0,\eps_0)$ and $\gamma >0$: 
\begin{align}\label{eq:tiltH1prep}
 \P\Big(\sup_{t\in[0,T]}x^\eps(t)>\gamma\Big) &\leq  \frac{\exp(\frac{1}{2}\log\gamma)}{\gamma}\E[h^\eps(T)\wedge \gamma^{1/4}]   +\P\big(h^\eps(T) > \frac{\gamma^{1/4}}{C_3}  \big)+\P\big( \int_0^T  b^\eps\dd s>\frac{\frac{1}{2}\log\gamma}{C_3}\big)\notag\\
&\leq \gamma^{-1/4}+ \P\big(h^\eps(T) > \frac{\gamma^{1/4}}{C_3}  \big)+\P\big(\|b^\eps\|_{L^1(0,T)}>\frac{\frac{1}{2}\log\gamma}{C_3}\big),
\end{align} 
where $C_3>0$ is a constant. 
Now observe the following bounds (recall \eqref{eq:defsgronwallfunc}):
\begin{align*}
 h^\eps(t)&\leq y_\eps(t)+z_\eps(t)+z_\eps(t)^2+\|\Xi\|_{L^2(0,T;L^2(\O))}^2, \\
 \|b^\eps\|_{L^1(0,T)}&\leq (1+\sup_{t\in[0,T]}y_\eps(t))(T\sup_{t\in[0,T]}z_\eps(t)+K^2).
\end{align*}
Combining these with \eqref{eq:tiltH1prep}, and \eqref{eq: tilt SPDE L^2} and \eqref{eq: tilt SPDE intermed}, we conclude that
\[
\lim_{\gamma\to\infty}\sup_{\eps\in(0,\eps_0)}\P\Big(\sup_{t\in[0,T]}x^\eps(t)>\gamma\Big)=0. 
\]
Consequently, \eqref{eq:tilt to show H1} holds, concluding the proof of Proposition \ref{prop:tilt}. 
\qed

\section{Proof of the LDP for Stratonovich noise}\label{sec: Strat}

In this section, we prove Theorem \ref{th:stratonovich}, using the abstract LDP result of  Theorem \ref{th:LDP general Strat}.

 \begin{proof}[Proof of Theorem \ref{th:stratonovich}]
 We begin by reformulating the stochastic primitive equations with Stratonovich transport noise  \eqref{eq:primitive_Stratonovich}--\eqref{eq:primitive_Stratonovich_boundary_conditions} as an abstract stochastic evolution equation of the form \eqref{eq:SPDEappStrat}. 
By \cite[\S8.2]{AHHS22}, the following identities hold: 
\begin{align}
\label{eq:Ito_corrections_strong_weak}
\begin{split}
\sqrt{\eps}\sum_{n\geq 1}\hhp[(\phi_n \cdot \nabla) v^\eps]\circ \dd\beta^n_t
&=\eps 
\hhp\big[\Lvphi v^\eps+ \Lvp v^\eps \big] \dd t+ \sqrt{\eps}\sum_{n\geq 1} \hhp[(\phi_n \cdot \nabla) v^\eps]  \dd\beta^n_t,\\
\sqrt{\eps}\sum_{n\geq 1} (\psi_n \cdot \nabla ) \theta^\eps \circ \dd\beta_t^n 
&=\eps 
\LTphi \theta^\eps \dd t
+
\sqrt{\eps}\sum_{n\geq 1} (\psi_n \cdot \nabla ) \theta^\eps  \dd\beta_t^n,
\end{split}
\end{align}
with
\begin{align}\label{eq:def_LTp}
\begin{cases}
\Lvphi v
&\ceqq
 \sum_{i,j=1}^3 \Big(a_{\phi}^{i,j} \partial_{ij} v + \frac{1}{2}\sum_{n\geq 1}(\partial_i \phi^j_n) \phi^i_n \partial_j v\Big),\\
\Lvp v
&\ceqq\Big(\sum_{n\geq 1}\sum_{i=1}^2\partial_j \phi_n^i  (\Q[(\phi_n\cdot\nabla) v])^i \Big)_{j=1}^2, \\
\LTphi \theta
&\ceqq
\div(a_{\psi} \cdot\nabla \theta)-\frac{1}{2}\sum_{n\geq 1} (\div \,\psi_n )[ (\psi_n\cdot \nabla)\theta],
\end{cases}
\end{align}
where  $a_{\phi}=(a^{i,j}_{\phi})_{i,j=1}^3$ and $a_{\psi}=(a^{i,j}_{\psi})_{i,j=1}^3$ are defined by 
\begin{equation*}
\label{eq:def_a_phi_psi_strong_weak}
a_{\phi}^{i,j}(x)\ceqq\delta^{i,j}+\tfrac{1}{2}\textstyle{\sum_{n\geq 1}} \phi^i_n (x)\phi^j_n (x),
\quad 
a_{\psi}^{i,j} (x)\ceqq\delta^{i,j}+\tfrac{1}{2}\textstyle{\sum_{n\geq 1}} \psi^i_n (x)\psi^j_n (x),\qquad x\in\O.
\end{equation*}
In \eqref{eq:def_LTp},   $(\Q[\cdot])^i$ denotes the $i$-th component of the vector $\Q[\cdot]$. 

Now we will apply Theorem \ref{th:LDP general Strat} as follows: we use the Gelfand triple from \eqref{eq:Gelfand} and define $(A,B)$ by \eqref{eq:defAB}, where $G\ceqq 0$ and $(A_0,B_0,F)$ are defined as in \eqref{eq: defs coeff},  where we put  $\gamma_n^{l,m}=0$ for all $n\geq 1$ and $l,m\in\{1,2\}$, so  $\mathcal{P}_{\gamma,\phi}=\mathcal{P}_{\gamma,G} =0$. 
Moreover, we define $L\in \mathcal{L}(V,V^*)$ by 
\[ 
L(v,\theta)\ceqq 
\begin{pmatrix}  
\hhp\big[\Lvphi v + \Lvp v  \big] \\
\LTphi \theta 
\end{pmatrix}.
\] 
To see  that $L\in \mathcal{L}(V,V^*)$, we recall \eqref{eq:def_LTp} and use Assumption \ref{ass:locglob_Stratonovich_strong_weak}\ref{it:well_posedness_primitive_phi_psi_smoothness_Stratonovich} and Remark \ref{rem:assumptionslocal}. Lastly, we let $W\ceqq\Br_{\ell^2}$ be defined by \eqref{eq:cyl}. 
Then, \eqref{eq:Ito_corrections_strong_weak} yields that \eqref{eq:primitive_Stratonovich}--\eqref{eq:primitive_Stratonovich_boundary_conditions} is equivalent to the abstract stochastic evolution equation \eqref{eq:SPDEappStrat}. Furthermore,   assumption \ref{it:1strat} of Theorem \ref{th:LDP general Strat} is fulfilled by \cite[Th.\ 8.5]{AHHS22}, and well-posedness holds even for all $\eps\in\R$.

Next, we turn to the most difficult condition, namely, we   verify that \ref{it:3strat} of Theorem \ref{th:LDP general Strat} is satisfied. The latter can be proved with  minor adaptations of the proof of Proposition \ref{prop:tilt}, which was given in Section \ref{sec:tiltedspde}. 
Observe that Assumption \ref{it:well_posedness_primitive_parabolicity} (which is not implied by Assumption \ref{ass:locglob_Stratonovich_strong_weak}) has not been used throughout the whole of Section \ref{sec:tiltedspde}. Thus, the estimates given therein are still valid, and the proof boils down to estimating a few extra terms, which we will now list.  

In Subsection \ref{sub:tilt L2}, when the It\^o formula is applied, there is one additional term in \eqref{eq:J1} and in \eqref{eq:J2}, respectively:
\[
 2\eps\int_0^t \int_\O v^\eps \cdot \hhp[\Lvphi v^\eps + \Lvp v^\eps]\dd x\dd s\eqqc J_1(t)\quad \text{ and }\quad
 2\eps\int_0^t \int_\O \theta^\eps \LTphi \theta ^\eps \dd x\dd s\eqqc J_2(t),
\]
which can be estimated by
\begin{align*}
  J_1(t)+J_2(t) 
  &\lesssim\eps \int_0^t\|v^\eps\|_{L^2(\O)}^2+\|\nabla v^\eps\|_{L^2(\O)}^2+\|\theta^\eps\|_{L^2(\O)}^2+\|\nabla \theta^\eps\|_{L^2(\O)}^2\dd s,
\end{align*}
where we use integration by parts,  Assumption \ref{ass:locglob_Stratonovich_strong_weak}\ref{it:well_posedness_primitive_phi_psi_smoothness_Stratonovich}, H\"older's inequality (with $\frac{1}{3+\delta}+\frac{1}{\rho}=\frac{1}{2}$) and the Sobolev embedding $H^1(\O)\into L^\rho(\O)$. The remainder of Subsection \ref{sub:tilt L2} now remains valid, as $\eps_0''$ therein can be picked sufficiently small. 

In Subsection \ref{sub:tilt intermed}, equations \eqref{eq: quadr bar v}, \eqref{eq: quadr pz v}, \eqref{eq: quadr tilde v 2nd part} will respectively contain the additional terms
\begin{align*}
&J_3(t)\ceqq 2\eps\int_0^t \int_{\Tor^2} \bar{v}^\eps \cdot  \overline{\hhp[\Lvphi  {v}^\eps + \Lvp {v}^\eps]}\dd x\dd s-2\eps\int_0^t \int_{\Tor^2} \Delta_{\h}\bar{v}^\eps \cdot  \overline{\hhp[\Lvphi {v}^\eps + \Lvp {v}^\eps]}\dd x\dd s,\\
&J_4(t)\ceqq-2\eps\int_0^t \int_\O \pzz v^\eps \cdot \hhp[\Lvphi v^\eps + \Lvp v^\eps]\dd x\dd s,\\
&J_5(t)\ceqq4\eps\int_0^t \int_\O \tilde{v}^\eps|\tilde{v}^\eps|^2 \cdot \widetilde{ \hhp[\Lvphi {v}^\eps + \Lvp {v}^\eps]}\dd x\dd s.
\end{align*}
Below, we comment on how to estimate the above terms. Firstly, due to Assumption \ref{ass:locglob_Stratonovich_strong_weak}\ref{it:independence_z_variable_Stratonovich}, it follows that $\Lvp v^\eps 
= \Lvp \overline{v^\varepsilon}$. Moreover, a similar argument applies to the horizontal terms in $\Lvphi  {v}^\eps$, and therefore one can check that 
$$
J_3(t)\lesssim \eps\int_0^t\big(1+ \|\bar{v}^\eps\|_{H^2(\Tor^2)}^2+\|\partial_3 v\|_{H^1(\O)}^2\big)\dd s.
$$
As for $J_4$ and $J_5$, recall that $\wt{\hhp f}= \wt{f}$ for any $f\in L^2(\O;\R^2)$ and $
 \Lvp v^\eps 
$
does not depend on $x_3$. Hence, 
\begin{align*}
&J_4(t)=-2\eps\int_0^t \int_\O \pzz v^\eps \cdot \Lvphi v^\eps \dd x\dd s,\\
&J_5(t)=\eps\int_0^t \int_\O \tilde{v}^\eps|\tilde{v}^\eps|^2 \cdot \widetilde{\Lvphi {v}^\eps}\dd x\dd s.
\end{align*} 
Let us first consider $J_4$. Using Assumption \ref{ass:locglob_Stratonovich_strong_weak}\ref{it:well_posedness_primitive_phi_psi_smoothness_Stratonovich}\ref{it:BCpsiphi}\ref{it:psi_3_null_regularity_assumption_phi_boundary},  integration by parts in the vertical variable and the smoothness of the transport noise coefficients readily yield 
$$
J_4(t)\lesssim \eps\int_0^t\big(\| v^\eps\|_{H^1(\O)}^2+\|\nabla\pz v^\eps\|_{L^2(\O)}^2\big)\dd s,
$$
see also the proof of \cite[Th.\ 7.5]{AHHS22} for a similar estimate.  
Finally, we consider $J_5$. We focus only on the highest order differential operators in $\Lvphi$. For the contributions with only horizontal derivatives, an integration by parts shows that these can be bounded by $\eps\int_0^t\||\tilde{v}^\eps||\nabla \tilde{v}^\eps|\|_{L^2(\O)}^2\dd s$. Meanwhile, for the vertical derivatives, recall that $\pz v^\varepsilon=\pz\widetilde{v^\varepsilon}$. Hence, the boundary condition in Assumption \ref{ass:locglob_Stratonovich_strong_weak}\ref{it:psi_3_null_regularity_assumption_phi_boundary} and an integration by parts yield again a bound of the form 
$\lesssim\eps\int_0^t\||\tilde{v}^\eps||\nabla \tilde{v}^\eps|\|_{L^2(\O)}^2\dd s$. Thus, putting all together and keeping in mind also the lower order terms in the operator $\Lvphi$,
$$
J_5(t)\lesssim \eps\int_0^t\big(1+\||\tilde{v}^\eps||\nabla \tilde{v}^\eps|\|_{L^2(\O)}^2+ \|\wt{v^\varepsilon}\|_{L^4(\O)}^4\big)\dd s .
$$ 
Combining the previous estimates, we proved that, for a sufficiently small $\eps_0'$, \eqref{eq:z^epsest} and its implications remain valid.

In Subsection \ref{sub:tilt H1}, we have one extra term in \eqref{eq:J6}, namely
\[
J_6(t)\ceqq J_1(t)- 2\eps\int_0^t \int_\O \Delta v^\eps \cdot \hhp[\Lvphi v^\eps + \Lvp v^\eps]\dd x\dd s \lesssim \eps\int_0^t\| {v}^\eps\|_{H^2(\O)}^2\dd s .
\]
Thus again, for a small enough $\eps_0$, \eqref{eq:x^epsest} and its implications hold, finishing the verification of condition  \ref{it:3strat} of Theorem \ref{th:LDP general Strat}. 

It remains to address Assumptions \ref{ass:critvarsettinglocal} and  \ref{ass:coer replace}\ref{it:2}. 
  There is one issue: the parabolicity of Assumption \ref{it:well_posedness_primitive_parabolicity} does not necessarily hold under Assumption \ref{ass:locglob_Stratonovich_strong_weak}, so we cannot apply Sections \ref{sec:proof main result} and \ref{sec:skeleton} to $(A,B)$. 
  However, by a perturbation argument as in \cite[Cor.\ 5.7]{T25L2LDP}, we resolve this.
  
  First, observe that   for a sufficiently small $\delta>0$, $(A,\delta B)$ satisfies the full Assumption   \ref{ass:well_posedness_primitive} (including the parabolicity in \ref{it:well_posedness_primitive_parabolicity}) as well as  Assumption \ref{ass:global_primitive}. Therefore, Section \ref{sec:proof main result} yields that Assumption \ref{ass:critvarsettinglocal} is satisfied for $(A,\delta B)$, and Section \ref{sec:skeleton} yields that Assumption \ref{ass:coer replace}\ref{it:2} is satisfied for  $(A,\delta B)$. Then, note that $(A,\delta B)$ also satisfies Assumption \ref{ass:locglob_Stratonovich_strong_weak}, so the proof above also applies to $(A,\delta B)$, giving that \ref{it:1strat} and \ref{it:3strat} are satisfied. Combining these, we can apply Theorem \ref{th:LDP general Strat} to $(A,\delta B)$ and obtain the LDP for the solutions to \eqref{eq:primitive_Stratonovich}--\eqref{eq:primitive_Stratonovich_boundary_conditions} in which $(\phi_n)$ and $(\psi_n)$ are replaced by  $(\delta\phi_n)$ and $(\delta\psi_n)$ respectively. Finally, the latter implies the LDP for the original equation \eqref{eq:primitive_Stratonovich}--\eqref{eq:primitive_Stratonovich_boundary_conditions}, by doing the substitutions  $\eps\mapsto \eps\delta^2$ in \eqref{eq:primitive_Stratonovich}--\eqref{eq:primitive_Stratonovich_boundary_conditions} and $\varphi\mapsto \delta^{-1}\varphi$  in the skeleton equation \eqref{eq:primitive_skeleton}, \eqref{eq:boundary_conditions_strong_weak}. For details on this last  argument, we refer to the end of the proof of \cite[Cor.\ 5.7]{T25L2LDP}. This completes the proof of Theorem \ref{th:stratonovich}. 
\end{proof}

\appendix

\section{An abstract LDP theorem}\label{appendix}

Here we discuss  special cases of the results in \cite{T25L2LDP}, which are used to prove the LDP for the stochastic primitive equations. Consider an abstract stochastic evolution equation: 
\begin{equation}\label{eq:SPDEapp}
   \begin{cases}
  &\dd Y^\eps(t)=-A(t,Y^\eps(t))\dd t+\sqrt{\eps}B(t,Y^\eps(t))\dd W(t), \quad t\in[0,T], \\
  &Y^\eps(0)=u_0\in H,
\end{cases}
\end{equation}  
equipped with a Gelfand triple $(V,H,V^*)$. We consider solutions $Y^\eps$ taking values in the following maximal regularity space:
\begin{equation*}
    \MR(0,T)\coloneqq C([0,T];H)\cap L^2(0,T;V), \quad \|\cdot\|_{\MR(0,T)}\coloneqq \|\cdot\|_{C([0,T];H)}+\|\cdot\|_{L^2(0,T;V)}.
\end{equation*}

We will make  two assumptions for the coefficients $A$ and $B$. The following assumption  is a special case of  \cite[Ass.\ 2.1]{T25L2LDP}, taking into account Remark \ref{rem:growth from lip} below. 
 
\begin{assumption}\noindent\phantomsection\label{ass:critvarsettinglocal} 

\noindent
\begin{enumerate}[label=\textit{(\arabic*)},ref=\ref{ass:critvarsettinglocal}\textit{(\arabic*)}]
\item\label{it:gen0} $A(t,v)=A_0(t)v-F(t,v)$ and  $B(t,v)=B_0(t)v+G(t,v)$, where
\begin{itemize}
    \item $ 
A_0\colon \R_+ \to\mathcal{L}(V,V^*) \text{ and } B_0\colon \R_+ \to \mathcal{L}(V,\UH)) 
$ 
are $\BB(\R_+)$-measurable, 
\item $
F\colon \R_+\times V\to V^* \text{ and } G\colon \R_+\times V\to\UH
$ 
are $\BB(\R_+)\otimes \mathcal{B}(V)$-measurable, 
\item $
F(\cdot,0)\in L^2_\loc(\R_+;V^*) \text{ and } G(\cdot,0)\in L^2_\loc(\R_+;\UH).
$
\end{itemize}

 \item\label{it:coercivelinear}  
  There exist  $\theta,M>0$ such that for all $t\in\R_+,v\in V$: 
  \[
  \<A_0(t)v,v\?-\frac{1}{2}\| B_0(t)v\|_{\UH}^2\geq \theta \|v\|_V^2-M \|v\|_H^2.
  \]  
  
\item \label{it:growth AB} There exist $m\in\N$ and $\rho_j\geq 0$, $\beta_j\in(\frac{1}{2},1)$ such that 
\[
2\beta_j\leq 1+\frac{1}{1+\rho_j},\quad j\in\{1,\ldots, m\},
\]
and  there exists a constant $C $ such that for all $t\in\R_+$ and $u,v\in V$: 
\begin{align*}
&\|A_0(t)\|_{\mathcal{L}(V,V^*)}+\| B_0(t)\|_{\mathcal{L}(V,\UH)} \leq C,\\
&\|F(t,u)-F(t,v)\|_{V^*}+\| G(t,u)-G(t,v)\|_{\UH} \leq C\textstyle{\sum_{j=1}^{m}} (1+\|u\|_{\beta_j}^{\rho_j}+\|v\|_{\beta_j}^{\rho_j})\|u-v\|_{\beta_j}. 
\end{align*}
\end{enumerate}
\end{assumption}

Note that Assumption  \ref{it:coercivelinear} encodes coercivity of the linear part $(A_0,B_0)$ of $(A,B)$, but coercivity need not hold for the full pair $(A,B)$. 

\begin{remark}\label{rem:growth from lip}
If $F(\cdot,0)\in  L^2_\loc(\R_+;V^*) $ and $G(\cdot,0)\in L^2_\loc(\R_+;\UH)$, and if $F$ and $G$ satisfy the Lipschitz conditions of Assumption \ref{it:growth AB}, then it is no longer necessary to verify the growth bounds for $F$ and $G$ in  \cite[Ass.\ 2.1(3)(4)]{T25L2LDP}.   Indeed, define $F_1(\cdot,u)\ceqq F(\cdot,u)-F(\cdot,0)$, $f_1\ceqq F(\cdot,0)+f$. Then $F_1$ satisfies both the growth and Lipschitz bounds of \cite[Ass.\ 2.1(4)]{T25L2LDP}, and $A= A_0-F-f= A_0-F_1-f_1$. The analogous trick works for $G$. 
\end{remark}

For the LDP, we consider two additional equations related to \eqref{eq:SPDEapp}. 
The skeleton equation for given $\varphi\in L^2(0,T;U)$ and $u_0\in H$, is:  
\begin{equation}\label{eq:skeleton}
   \begin{cases}
  &\dd u^\varphi(t)=-A(t,u^\varphi(t))\dd t+B(t,u^\varphi(t))\varphi(t)\dd t, \quad t\in[0,T], \\
  &u^\varphi(0)=u_0.
\end{cases}
\end{equation}  
Also, we consider the following stochastic evolution equation:
\begin{equation}\label{eq:SPDE tilted}
   \begin{cases}
  &\dd X^\eps(t)=-A(t,X^\eps(t))\dd t+B(t,X^\eps(t))\bphieps(t)\dd t+\sqrt{\eps}B(t,X^\eps(t))\dd W, \quad t\in[0,T], \\
  &X^\eps(0)=u_0,
\end{cases}
\end{equation} 
where $\bphieps$ is  a predictable stochastic process taking a.s.\ values in $L^2(0,T;U)$. 

The following solution concept is used for \eqref{eq:SPDEapp}--\eqref{eq:SPDE tilted}. 

\begin{definition}\label{def:sol}
Let $T>0$, let $(V,H,V^*)$ be a Gelfand triple and let $U$ be a real separable Hilbert space. 
Let $A\colon[0,T]\times V\to V^*$, $B\colon[0,T]\times V\to \UH$ and let $u_0\in H$. Let $W$ be a $U$-cylindrical Brownian motion on a filtered probability space $(\Om,\mathcal{F},(\mathcal{F}_t)_{t\geq 0},\P)$. Consider
\begin{equation}\label{eq:SPDE plain}
     \begin{cases}
   &\dd u(t)=-A(t,u(t))\dd t+B(t,u(t))\dd W(t), \quad t\in[0,T], \\
   &u(0)=u_0,
\end{cases}
\end{equation}
We say that a strongly progressively measurable process $u\colon [0,T]\times \Om\to V$ is a \emph{strong solution} to \eqref{eq:SPDE plain} if a.s.
$
u\in\MR(0,T)$, $A(\cdot,u(\cdot))\in L^2(0,T;V^*)+L^1(0,T;H)$, $B(\cdot,u(\cdot))\in L^2(0,T;\UH)$ and a.s.:
\begin{align}\label{eq:strong sol}
  u(t)=u_0-\int_0^t A(s,u(s))\dd s+\int_0^t B(s,u(s))\dd W(s) \: \text{ in }V^* \text{ for all }t\in[0,T].
\end{align}

If $B=0$, we write $u'(t)=-A(t,u(t))$ instead of $\dd u(t)=-A(t,u(t))\dd t$ in \eqref{eq:SPDE plain} and we call $u\in \MR(0,T)$ a strong solution  if $A(\cdot,u(\cdot))\in L^2(0,T;V^*)+L^1(0,T;H)$ and \eqref{eq:strong sol} holds.
\end{definition}
 
The following assumption is made to compensate for the missing coercivity of $(A,B)$. 
   
\begin{assumption}\label{ass:coer replace}
$(V,H,V^*)$ is a Gelfand triple of real separable Hilbert spaces, $U$ is a real separable Hilbert space and $u_0\in H$. Moreover, 
$(A,B)$ satisfies Assumption  \ref{ass:critvarsettinglocal}.  
In addition, the following conditions are satisfied for some $\eps_0>0$:
\begin{enumerate}[label=\textit{(\Roman*)},ref=\textit{(\Roman*)}]
  \item\label{it:1} For every $U$-cylindrical Brownian motion $W$ on a filtered probability space $(\Om,\mathcal{F},(\mathcal{F}_t)_{t\geq 0},\P)$ and $\eps\in (0,\eps_0)$, \eqref{eq:SPDEapp} has a unique strong solution $Y^\eps$ in the sense of Definition \ref{def:sol}. 
  \item\label{it:2} For every ${T}>0$ and $\varphi\in L^2(0,{T};U)$: if   $u^{\varphi}$ is a strong solution to \eqref{eq:skeleton}, then it satisfies 
  \begin{equation}\label{eq:skeleton a priori}
   \|u^\varphi\|_{\MR(0,{T})}\leq C_{u_0}({T},\|\varphi\|_{L^2(0,{T};U)}),
  \end{equation}
   with $C_{u_0}\col \R_+\times \R_+\to\R_+$  non-decreasing in both components. 
\item\label{it:3} For every $T,K>0$ and for every collection $(\bphieps)_{\eps\in(0,\eps_0)}$ of predictable stochastic processes $\bphieps\col [0,T]\times \Om\to U$ with   $\|\bphieps\|_{L^2(0,T;U)}\leq K$ a.s., the following holds: 
if $(X^\eps)_{\eps\in(0,\eps_0)}$ is a family such that for all $\eps\in(0,\eps_0)$, $X^\eps$ is a strong solution to \eqref{eq:SPDE tilted}, then 
\begin{equation*}
    \lim_{\gamma\to\infty}\sup_{\eps\in(0,\eps_0)}\P(\|X^\eps\|_{\MR(0,T)}>\gamma)=0. 
\end{equation*}
\end{enumerate} 
\end{assumption}

A special case of the result \cite[Th.\ 2.8]{T25L2LDP} now reads as follows. It coincides with the latter result, for the special case of semilinear equations, $m=1,\alpha_1=0,\hat{f}=g=0$ and with uniform constants.  

\begin{theorem}\label{th:LDP general}
    Let Assumptions \ref{ass:critvarsettinglocal} and \ref{ass:coer replace} hold and let $T>0$. Then the family $(Y^\eps)_{\eps\in(0,\eps_0)}$ of solutions to \eqref{eq:SPDEapp} satisfies the large deviation principle on $L^2(0,T;V)\cap C([0,T];H)$ with rate function  $I\colon L^2(0,T;V)\cap C([0,T];H)\to [0,+\infty]$ given by
\begin{equation*} 
I(z)=\frac{1}{2}\inf\Big\{\textstyle{\int_0^T}\|\varphi(s)\|_U^2\dd s : \varphi\in L^2(0,T;U), \, z=u^{\varphi}\Big\},
\end{equation*}
where $\inf\varnothing\coloneqq +\infty$ and $u^\varphi$ is a strong solution to \eqref{eq:skeleton}, which exists uniquely. 
\end{theorem}

Next, we consider another abstract evolution equation, which corresponds to \eqref{eq:SPDEapp} with an additional term $\eps L\tilde{Y}^\eps$. The latter can be chosen as an It\^o--Stratonovich correction term when treating Stratonovich noise. 
\begin{equation}\label{eq:SPDEappStrat}
   \begin{cases}
  &\dd \tilde{Y}^\eps(t)=-A(t,\tilde{Y}^\eps(t))\dd t+\eps L\tilde{Y}^\eps(t)\dd t+\sqrt{\eps}B(t,\tilde{Y}^\eps(t))\dd W(t), \quad t\in[0,T], \\
  &\tilde{Y}^\eps(0)=u_0\in H.
\end{cases}
\end{equation} 
To formulate the result, we also consider the following associated problem:
\begin{equation}\label{eq:SPDE tiltedStrat}
   \begin{cases}
  &\dd \tilde{X}^\eps(t)=-A(t,\tilde{X}^\eps(t))\dd t+\eps L\tilde{X}^\eps(t)\dd t+B(t,\tilde{X}^\eps(t))\bphieps(t)\dd t+\sqrt{\eps}B(t,\tilde{X}^\eps(t))\dd W, \quad t\in[0,T], \\
  &\tilde{X}^\eps(0)=u_0. 
\end{cases}
\end{equation} 

A special case of \cite[Th.\ 5.3]{T25L2LDP} is stated below. It coincides with the latter result in the special case of semilinear equations, $m=1,\alpha_1=0,\hat{f}=g=0$, uniform constants, $\eta(\eps)=\eps$ and $\tilde{A}(t,\cdot)=L\in\mathcal{L}(V,V^*)$. We note that these $\eta(\eps)$ and $\tilde{A}$  satisfy \cite[Ass.\ 5.1(i)(ii)]{T25L2LDP}.

\begin{theorem}\label{th:LDP general Strat} Let $T>0$. 
    Let Assumptions \ref{ass:critvarsettinglocal} and \ref{ass:coer replace} hold, where   
     Assumption \ref{ass:coer replace}\ref{it:1} and \ref{it:3} are replaced by: 
    \begin{enumerate}[label=\textit{( $\widetilde{\!\!\text{\Roman*}}$)},ref=\textit{( $\widetilde{\!\!\text{\Roman*}}$)}]
  \item\label{it:1strat} For every $U$-cylindrical Brownian motion $W$ on a filtered probability space $(\Om,\mathcal{F},(\mathcal{F}_t)_{t\geq 0},\P)$ and $\eps\in (0,\eps_0)$, \eqref{eq:SPDEappStrat} has a unique strong solution $\tilde{Y}^\eps$ in the sense of Definition \ref{def:sol}.  
      \setcounter{enumi}{2}
\item\label{it:3strat} For every $T,K>0$ and for every collection $(\bphieps)_{\eps\in(0,\eps_0)}$ of predictable stochastic processes $\bphieps\col [0,T]\times \Om\to U$ with   $\|\bphieps\|_{L^2(0,T;U)}\leq K$ a.s., the following holds: 
if $(\tilde{X}^\eps)_{\eps\in(0,\eps_0)}$ is a family such that for all $\eps\in(0,\eps_0)$, $\tilde{X}^\eps$ is a strong solution to \eqref{eq:SPDE tiltedStrat}, then 
\begin{equation*}
    \lim_{\gamma\to\infty}\sup_{\eps\in(0,\eps_0)}\P(\|\tilde{X}^\eps\|_{\MR(0,T)}>\gamma)=0. 
\end{equation*}
\end{enumerate} 
    
     Then the family $(\tilde{Y}^\eps)_{\eps\in(0,\eps_0)}$ of solutions to \eqref{eq:SPDEappStrat} satisfies the large deviation principle on $L^2(0,T;V)\cap C([0,T];H)$ with the same rate function as in Theorem \ref{th:LDP general}. 
\end{theorem}

\noindent
{\bf Data availability.} This manuscript has no associated data.
\smallskip

\noindent
{\bf Declaration – Conflict of interest.} The authors have no conflict of interest.

\printbibliography

\end{document}